\documentclass[10pt]{amsart}

\calclayout

\usepackage{amsmath, amssymb}
\usepackage{amsfonts}
\usepackage{mathrsfs}
\usepackage[color, arrow,matrix,curve,cmtip,ps]{xy}
\usepackage{paralist}
\usepackage{quiver}
\usepackage{amsthm}
\usepackage{caption}
\usetikzlibrary{arrows,calc,matrix}
\tikzset{
curvarr/.style={
  to path={ -- ([xshift=2ex]\tikztostart.east)
    |- (#1) [near end]\tikztonodes
    -| ([xshift=-2ex]\tikztotarget.west)
    -- (\tikztotarget)}
  }
}
\tikzset{%
    symbol/.style={%
        draw=none,
        every to/.append style={%
            edge node={node [sloped, allow upside down, auto=false]{$#1$}}}
    }
}

\usepackage[final]{pdfpages}
\usepackage{rotating}
\PassOptionsToPackage{hyphens}{url}\usepackage{hyperref}
\usepackage{enumitem}
\usepackage{graphicx}
\usepackage{mathtools}
\usetikzlibrary{arrows,chains,matrix,positioning,scopes}

\usepackage[T1]{fontenc}
\usepackage[variablett]{lmodern}
\usepackage{xcolor}
\usepackage[citestyle=alphabetic,bibstyle=alphabetic]{biblatex}
\usepackage{lipsum}
\usepackage[most]{tcolorbox}

\newtheorem{theorem}{Theorem}[section]
\theoremstyle{definition}
\newtheorem{lemma}[theorem]{Lemma}

\newtheorem{warning}[theorem]{Warning}

\newtheorem{proposition}[theorem]{Proposition}
\newtheorem{corollary}[theorem]{Corollary}
\newtheorem{definition}[theorem]{Definition}

\newtheorem{remark}[theorem]{Remark}

\newtheorem{example}[theorem]{Example}

\allowdisplaybreaks

\newcommand{\Set}{\text{Set}}

\newcommand{\An}{\mathrm{An}}
\newcommand{\colim}{\mathrm{colim}}
\newcommand{\Map}{\mathrm{Map}}

\newcommand{\op}{\mathrm{op}}

\DeclareFieldFormat{postnote}{#1}
\DeclareFieldFormat{multipostnote}{#1}

\begin{document}
\title{Shape Theory of $\infty$-Topoi: Inverse Limits, Products, and (Co)homology}
\author{Georg Lehner}
\begin{abstract}
We give a systematic account of the shape theory of $\infty$-topoi, viewing
the shape of an $\infty$-topos as its generalized homotopy type. We establish
the basic functorial properties of the shape, including preservation of colimits, descent,
and homotopy invariance. We then prove that shape preserves cofiltered limits
under compactness and perfectness hypotheses and establish Künneth-type
formulas for products. Finally, we give conditions under which the shape of an
$\infty$-topos determines its cohomology and homology.
\end{abstract}
\maketitle

\begin{figure}[h!]
    \centering
    \includegraphics[width=0.45\textwidth]{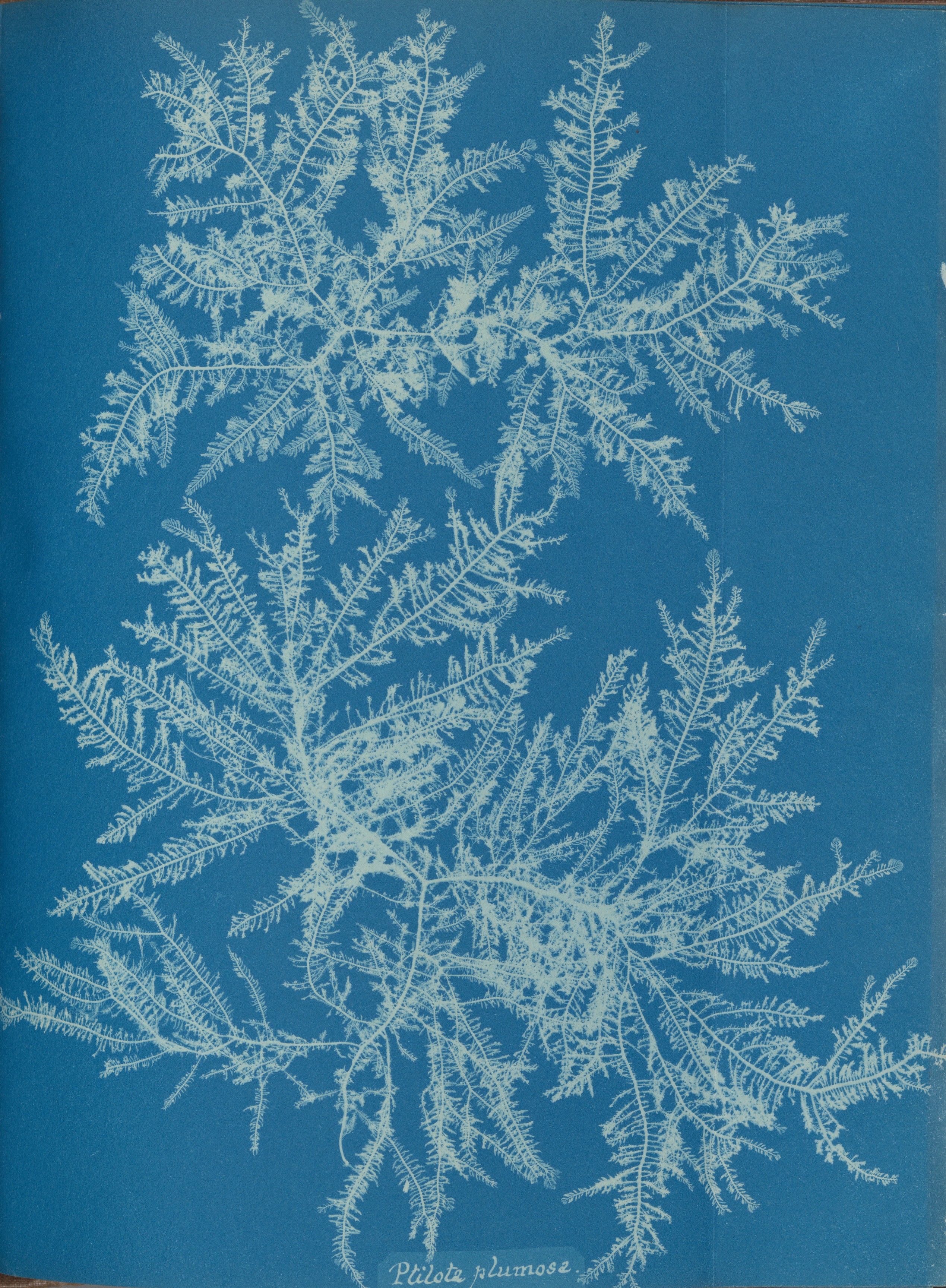}
    \caption*{\tiny Anna Atkins, \textit{Ptilota plumosa}, ca.~1853, cyanotype, from \textit{Photographs of British Algae: Cyanotype Impressions}, The Metropolitan Museum of Art, acc.~no.~2005.100.557 (92), Public domain, via The Met Open Access}
\end{figure}
    
\tableofcontents

\section{Introduction}

\emph{The set of connected components of a topological space is not a set, it is a pro-object of sets.} This perhaps controversial statement is substantiated by examples. Consider for instance the Cantor space $C$.
\begin{figure}[ht]
    \centering
    \includegraphics[width=0.4\textwidth]{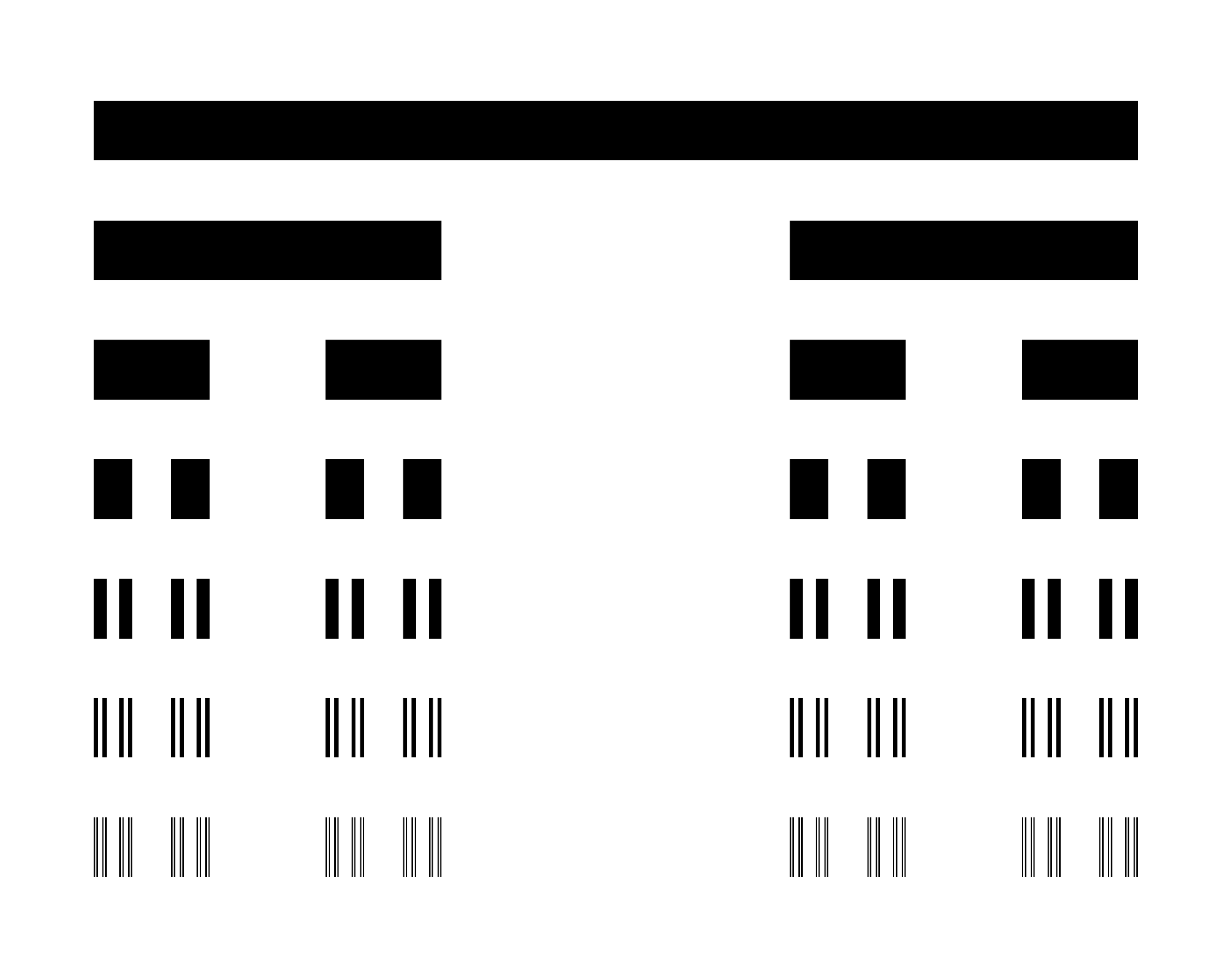}
    \caption{The Cantor space. Source: Wikimedia Commons.}\label{cantor_space}
\end{figure}

This space does not contain any non-empty connected open subsets. However, by thickening up the Cantor space just slightly, one obtains a finite union of closed intervals, dependent on the parameter of the thickening. Thus, as the thickening parameter tends to zero, one obtains a pro-set. This intuition can be made precise, and results in the adjunction
\[ \begin{tikzcd}
	{\mathrm{Top}} & {\mathrm{Pro}(\mathrm{Set})}
	\arrow[""{name=0, anchor=center, inner sep=0}, "{\pi_0}", curve={height=-12pt}, from=1-1, to=1-2]
	\arrow[""{name=1, anchor=center, inner sep=0}, "\beta", curve={height=-12pt}, from=1-2, to=1-1]
	\arrow["\dashv"{anchor=center, rotate=-90}, draw=none, from=0, to=1]
\end{tikzcd} \]
where $\beta$ is the pro-extension of the functor $(-)^\mathrm{disc} : \mathrm{Set} \rightarrow \mathrm{Top}$ that equips a set with its discrete topology.

The field of \emph{shape theory} applies this intuition further to homotopy theory. If one wants to understand the homotopy type of a space with poor local behaviour, one should rather assign to the space a pro-system of homotopy types that sees how the homotopy type changes at progressively finer resolutions of the space. Shape theory as a field grew in activity in the 1970s until the 1990s, and initially dealt mostly with the context of compact metric spaces, later generalized to general topological spaces. We recommend \cite{mardesic1982shape}, \cite{borsuk1975theory} and \cite{Gevorgyan2021ShapeTheory} for an overview.

A far-reaching generalization of the theory of shape occurred when Toën and Vezzosi \cite[Definition 5.3.2]{toen2003segaltopoistackssegal} and then Lurie \cite[Section 7.1.6]{Lurie2009HTT} and \cite[Appendix A]{Lurie2017HA}, building on ideas of Grothendieck, defined the \emph{shape} of an $\infty$-topos. It was realized that this notion is in fact a conservative extension of the so-called \emph{strong shape} of a compact Hausdorff space; however, this generalization is extremely far-reaching. Since \emph{every} topos has a shape, this implies that the theory is meaningful for arbitrary topological spaces (via taking their sheaf topoi), or more generally locales. It also works immediately in equivariant contexts, for example, by considering equivariant sheaves on a space equipped with a $G$-action. Hoyois developed the connection to classical Galois theory using the shape of the étale topos of a scheme, and showed that the shape of the latter agrees with the étale topological type of the scheme \cite{Hoyois2018HigherGaloisTheory}. The author’s motivations include applications to dynamical systems arising in geometric group theory as well as measure theory, and the study of the large-scale properties of spaces via the \emph{shape at infinity}, which will be the topic of the follow-up article \cite{Lehner2026Excision}.

Systematic treatments for topos-theoretic shape theory remain relatively few at the time of writing. Let us mention for instance \cite{Volpe2021SixOperationsInTopology} as well as \cite[Chapter 4]{BarwickGlasmanHaineExodromy}. The purpose of this article is twofold. First, we give a streamlined
account of the general theory of shape via higher topoi. Second, we establish preservation results for
cofiltered limits and products and show that, under suitable compactness
hypotheses, shape determines topos cohomology and homology.

In the following, write $\mathrm{An}$ for the $\infty$-category of \emph{animae}, equivalently homotopy types, equivalently $\infty$-groupoids, and write $\mathrm{Pro}(\mathrm{An})$ for the $\infty$-category of pro-objects in animae. For a given (higher) topos $\mathcal{X}$, such as the $\infty$-category $\mathrm{Sh}(X)$ of sheaves with values in $\mathrm{An}$ on a topological space $X$, there is an associated pro-anima
\[ \Pi_\infty( \mathcal{X} ) \in \mathrm{Pro}(\mathrm{An}) \]
called the \emph{shape} of $\mathcal{X}$. In the following, the word \emph{topos} refers to a \emph{Grothendieck $\infty$-topos} as defined by Lurie. Let $\mathrm{RTop}$ be the $\infty$-category of topoi and geometric morphisms, i.e.\ we view arrows as aligned with the right adjoint directions. The shape functor has the following properties. 

\begin{theorem}[See Proposition \ref{goodopencover}, Example \ref{abstractsimplicialcomplex}, Theorem \ref{shapeofstablycompact}, Theorem \ref{homotopyinvariance} and Theorem \ref{shapeadjunction}]
There exists a functor
\[ \Pi_\infty : \mathrm{RTop} \rightarrow \mathrm{Pro}(\mathrm{An}) \]
with the following properties.
\begin{itemize}
\item If $X$ is a CW complex, then $\Pi_\infty(\mathrm{Sh}(X)) \cong |X|$, where the latter refers to the homotopy type of $X$, viewed as an anima.
\item If $X$ is a compact Hausdorff space, then
\[ \Pi_\infty(\mathrm{Sh}(X)) \simeq \text{``}\lim_{(P,\mathcal{U}) \emph{\text{ finite }}P\emph{\text{-indexed stratification of }}X} \text{''} ~|P| \]
where the limit in question goes over the category of finite-poset-indexed stratifications of $X$.
\item The functor $\Pi_\infty$ is \emph{homotopy invariant}: For every topos $\mathcal{X}$, the natural geometric morphism 
\[ \mathrm{Sh}([0,1]) \times^\mathrm{RTop} \mathcal{X} \rightarrow \mathcal{X} \]
induces an isomorphism
\[ \Pi_\infty(  \mathrm{Sh}([0,1]) \times^\mathrm{RTop} \mathcal{X} ) \xrightarrow{ \cong } \Pi_\infty( \mathcal{X} ). \]
\item The functor $\Pi_\infty$ admits a right adjoint $\beta : \mathrm{Pro}(\mathrm{An}) \rightarrow \mathrm{RTop}$. In particular, it preserves colimits of topoi.
\end{itemize}
\end{theorem}

It is thus reasonable to think of the shape of a topos as the correct notion of \emph{homotopy type} of a topos, even away from locally contractible contexts. What is quite poignant is that while classically, the ``homotopy type'' functor, as defined on a suitable category of spaces, tends to be quite poorly behaved, for the shape instead one has the following (non-composable!) chain of adjunctions.

\[ \begin{tikzcd}
	{\mathrm{Top}} & {\mathrm{Loc}} & {\mathrm{RTop}} & {\mathrm{Pro}(\mathrm{An})}
	\arrow[""{name=0, anchor=center, inner sep=0}, "{\mathrm{Loc}}", curve={height=-12pt}, from=1-1, to=1-2]
	\arrow[""{name=1, anchor=center, inner sep=0}, "{\mathrm{pts}}", curve={height=-12pt}, from=1-2, to=1-1]
	\arrow[""{name=2, anchor=center, inner sep=0}, "{\mathrm{Sh}}"', curve={height=12pt}, hook', from=1-2, to=1-3]
	\arrow[""{name=3, anchor=center, inner sep=0}, "{\mathrm{Open}}"', curve={height=12pt}, from=1-3, to=1-2]
	\arrow[""{name=4, anchor=center, inner sep=0}, "{\Pi_\infty}", curve={height=-12pt}, from=1-3, to=1-4]
	\arrow[""{name=5, anchor=center, inner sep=0}, "\beta", curve={height=-12pt}, from=1-4, to=1-3]
	\arrow["\dashv"{anchor=center, rotate=-90}, draw=none, from=0, to=1]
	\arrow["\dashv"{anchor=center, rotate=-90}, draw=none, from=3, to=2]
	\arrow["\dashv"{anchor=center, rotate=-90}, draw=none, from=4, to=5]
\end{tikzcd} \]

Let us make two comments.
\begin{itemize}
\item The category $\mathrm{Loc}$ refers to the category of \emph{locales}, which are an abstraction of the notion of topological space, where only the information of the lattice of open sets enters. The functor $\mathrm{Top} \rightarrow \mathrm{Loc}$ is the functor that forgets the underlying set of points of a topological space. It is fully faithful when restricted to the class of \emph{sober} topological spaces. This is a rather mild separation condition and is satisfied by Hausdorff spaces or spectra of rings, for example.
\item The functor $\mathrm{Sh} : \mathrm{Loc} \rightarrow \mathrm{RTop}$ sends a locale $L$ to its topos of sheaves $\mathrm{Sh}(L)$. It is in fact fully faithful. Its left adjoint sends a topos $\mathcal{X}$ to its locale of \emph{open objects}, that is subobjects of the terminal object.
\end{itemize}	

From this one can see that preservation, or failure thereof, of certain limits or colimits of topological spaces or locales under taking the shape can be understood on a granular level, by understanding at which stage in the passage of a topological space to its sheaf topos, and then to its shape a failure can happen. For instance, the fact that shape is a left adjoint immediately produces the following Mayer--Vietoris-type statement.

\begin{theorem}[Descent for shape, see Theorem \ref{shapemayervietorisdescent}]
Let $\mathcal{X}$ be a topos and suppose $U \twoheadrightarrow 1$ is an effective epimorphism. Then
\[
\Pi_\infty(\mathcal{X}) \simeq \colim_{[n] \in \Delta^{\mathrm{op}}} \Pi_\infty( \mathcal{X}_{/U^{\times n+1}}),
\]
where the colimit in question is taken in the $\infty$-category $\mathrm{Pro}(\mathrm{An})$.
\end{theorem}

The particular application to topological spaces gives the following.

\begin{proposition}[See Corollary \ref{shapemayervietorisopens}]
Let $X$ be a topological space, and $\{U_i\}_{i\in I}$ a cover of $X$ by open subsets. Then
\[ \Pi_\infty(\mathrm{Sh}(X)) \cong \mathrm{colim} \Pi_\infty(\mathrm{Sh}(U_{i_0} \cap U_{i_1} \cap \cdots \cap U_{i_n}) ) \]
where the colimit in question is indexed by the \v{C}ech nerve of the cover.
\end{proposition}

One of the main points of the article is that shape is also well-behaved with certain classes of limits. First, there is the following theorem about inverse limits. We note that a geometric morphism $f : \mathcal{X} \rightarrow \mathcal{Y}$ is called \emph{perfect}, if the pushforward $f_*$ preserves filtered colimits. A topos $\mathcal{X}$ is called \emph{compact}, if the unique geometric morphism $\mathcal{X} \rightarrow \mathrm{An}$ is perfect.

\begin{theorem}[See Theorem \ref{shapeinverselimits}] \label{shapeinverselimitspreview}
Let $\mathcal{X}_\bullet : I \rightarrow \mathrm{RTop}$ be a diagram of topoi, with $I$ cofiltered, such that:
\begin{itemize}
\item for all $i \in I$ the topos $\mathcal{X}_i$ is compact, and 
\item for every morphism $i \rightarrow j$ in $I$ the geometric morphism $\mathcal{X}_i \rightarrow \mathcal{X}_j$ is perfect.
\end{itemize}
Let $\mathcal{X} = \mathrm{lim}_{i \in I} \mathcal{X}_i$. Then
\[ \Pi_\infty( \mathcal{X} ) \cong \mathrm{lim}_{i \in I} \Pi_\infty( \mathcal{X}_i ), \]
with the limit taken in $\mathrm{Pro}(\mathrm{An})$.
\end{theorem}

Next, we consider products. Here, we note that there is a canonical monoidal structure on $\mathrm{Pro}(\mathrm{An})$ called the \emph{substitution tensor} $\triangleright$, which has the property that
\[ X \triangleright Y \cong X \times Y \in \mathrm{Pro}(\mathrm{An}) \]
holds whenever $X$ is pro-compact or $Y$ is (equivalent to) a constant pro-object, see Proposition \ref{substitutiontensor}.

\begin{theorem}[Künneth theorem for shape, see Theorem \ref{kunnethlocallycontractible} and Theorem \ref{shapekuenneth}] \label{Kunnethpreview}
Let $\mathcal{X}, \mathcal{Y}$ be topoi and suppose that either
\begin{itemize}
\item $\mathcal{X}$ is a locally contractible topos, or
\item $\mathcal{Y}$ is a compact topos.
\end{itemize}
Then
\[ \Pi_\infty( \mathcal{X} \times^\mathrm{RTop} \mathcal{Y} ) \cong \Pi_\infty( \mathcal{X} ) \triangleright \Pi_\infty( \mathcal{Y} ). \]
If, furthermore, either 
\begin{itemize}
\item $\mathcal{X}$ has pro-compact shape, or
\item $\mathcal{Y}$ has constant shape
\end{itemize}
then \[ \Pi_\infty( \mathcal{X} \times^\mathrm{RTop} \mathcal{Y} ) \cong \Pi_\infty( \mathcal{X} ) \times \Pi_\infty( \mathcal{Y} ). \]
\end{theorem}

Some hypotheses of this kind are necessary: Example~\ref{counterexample}
shows that the formula can fail when the given compactness and local
contractibility assumptions are absent.

\begin{remark}
Less general versions of Theorem \ref{shapeinverselimitspreview} as well as Theorem \ref{Kunnethpreview} have appeared in \cite[Proposition 2.11]{Hoyois2018HigherGaloisTheory}.
\end{remark}

Finally, we relate the shape to cohomology and homology. For a given topos $\mathcal{X}$, a presentable $\infty$-category $\mathcal{E}$ and an object $E \in \mathcal{E}$ denote by
\[ H^\bullet(\mathcal{X},E) \in \mathcal{E}, \qquad \text{ and } \qquad H_\bullet(\mathcal{X},E) \in \mathcal{E} \]
the \emph{topos cohomology} as well as \emph{topos homology} of $\mathcal{X}$ with values in $E$, see Sections \ref{cohomology} and \ref{homology}. These are generalizations of ordinary sheaf (co)-homology of spaces. Define for a pro-anima $\text{``} \mathrm{lim}_{i \in I} \text{''} X_i \in \mathrm{Pro}(\mathrm{An})$ the cotensoring and tensoring as
\[ E^{\text{``} \mathrm{lim}_{i \in I} \text{''} X_i} := \mathrm{colim}_{i \in I} E^{X_i} \qquad (\text{``} \mathrm{lim}_{i \in I} \text{''} X_i ) \otimes E := \mathrm{lim}_{i \in I} (X_i \otimes E). \]
The main results are the following.

\begin{theorem}[See Theorem \ref{shapeandcohomologylocallycontractible} and Theorem \ref{shapeandcohomologycompactlyassembled}]
Let $\mathcal{X}$ be a topos, and $\mathcal{E}$ a presentable $\infty$-category. Let $E \in \mathcal{E}$. Assume that either
\begin{itemize}
\item $\mathcal{X}$ is locally contractible, or
\item $\mathcal{E}$ is compactly assembled.
\end{itemize}
Then there exists a natural isomorphism
\[ H^\bullet(\mathcal{X},E) \cong E^{\Pi_\infty(\mathcal{X})}. \]
\end{theorem}

\begin{theorem}[See Theorem \ref{localcontractiblehomology} and Theorem \ref{compacthomology}]
Let $\mathcal{X}$ be a topos, $\mathcal{E}$ a presentable $\infty$-category, and $E \in \mathcal{E}$. Assume that either
\begin{itemize}
\item $\mathcal{X}$ is locally contractible, or
\item $\mathcal{X}$ is compact and $\mathcal{E}$ is compactly assembled and stable.
\end{itemize}
Then
\[ H_\bullet( \mathcal{X}, E ) \cong \Pi_\infty(\mathcal{X}) \otimes E. \]
\end{theorem}

\begin{remark}
The material presented in this article is part of the larger book project \cite{Lehner2026ShapeTheory}.
\end{remark}

\subsection{Acknowledgements}
We thank Thomas Nikolaus, Thorger Geiß, Phil Pützstück, Marco Volpe, and Koen Bresters for helpful comments and discussions. Furthermore, the author was funded by the Deutsche Forschungsgemeinschaft (DFG, German Research Foundation) – Project-ID 427320536–SFB 1442, as well as under Germany's Excellence Strategy EXC 2044/2 - 390685587, Mathematics Münster: Dynamics-Geometry-Structure. \\
Large language models were used for conceptual exploration, literature
searches, and editing. The text as presented, as well as the statements and mathematical proofs given, are written by the author, who assumes full responsibility for the contents.

\section{Preliminaries}

We will use the language of $\infty$-categories throughout the text, as developed by Lurie \cite{Lurie2009HTT} and \cite{Lurie2017HA}. See \cite{Land2021InfinityCategories}, \cite{Hebestreit2019HigherCategoriesI}, \cite{Hebestreit2020HigherCategoriesII}, and \cite{Haugseng2025} for additional references. The term $\mathrm{An}$ refers to the $\infty$-category of anima, equivalently the $\infty$-category of spaces or also the $\infty$-category of $\infty$-groupoids. The $\infty$-category of small $\infty$-categories is denoted by $\mathrm{Cat}_\infty$, the $\infty$-category of large $\infty$-categories is denoted as $\widehat{\mathrm{Cat}}_\infty$. The term $\mathrm{Pr}^L$ denotes the $\infty$-category of presentable $\infty$-categories and left adjoint functors, and $\mathrm{Pr}^R \simeq (\mathrm{Pr}^L)^\mathrm{op}$ denotes the $\infty$-category of presentable $\infty$-categories and right adjoint functors. A functor $f : C \rightarrow D$ between categories with finite limits is called \emph{left-exact} if $f$ preserves finite limits.

Given a small $\infty$-category $C$, we denote its $\mathrm{Ind}$-completion by $\mathrm{Ind}(C)$. It admits a fully faithful embedding $y : C \hookrightarrow \mathrm{Ind}(C)$ that is characterized by the universal property that $\mathrm{Ind}(C)$ admits filtered colimits and for any $\infty$-category $D$ admitting filtered colimits the $\infty$-category of functors $C \rightarrow D$ is equivalent to the $\infty$-category of filtered colimit preserving functors $\mathrm{Ind}(C) \rightarrow D$. If $C$ admits finite colimits, then we have an equivalence
\[ \mathrm{Ind}(C) \cong \mathrm{Fun}^\mathrm{lex}( C^{op}, \mathrm{An} ), \]
where the right-hand side refers to left exact functors. The $\infty$-category $\mathrm{Pro}(C)$ is defined as $\mathrm{Pro}(C) := \mathrm{Ind}(C^\mathrm{op})^\mathrm{op}$ and satisfies the dual universal property with respect to cofiltered limits.

A \emph{locale} $L$ is defined via its complete lattice $(\mathcal{O}(L), \leq )$ of \emph{abstract open subsets}, and is required to satisfy the distributivity relation
\[ V \wedge \bigvee_{ \alpha } U_\alpha = \bigvee_{ \alpha } V \wedge U_\alpha \]
for any $V, U_\alpha \in \mathcal{O}(L)$. If $X$ is a topological space, then the lattice $\mathcal{O}(X)$ defines a locale. A locale homomorphism $f : L \rightarrow L'$ consists of the datum of a right adjoint functor $f_* : \mathcal{O}(L) \rightarrow \mathcal{O}(L')$, together with its left adjoint $f^* : \mathcal{O}(L') \rightarrow \mathcal{O}(L)$, which is required to preserve finite meets. For references on locale theory, see \cite{PicadoPultr}.

Let us give a summary of topoi. The standard reference for higher topos theory is of course \cite{Lurie2009HTT}. See also \cite{CnossenHigherToposTheory} for an additional reference. If $C$ is a small $\infty$-category, then the $\infty$-category $\mathrm{PSh}(C) := \mathrm{Fun}(C^\mathrm{op},\mathrm{An})$ denotes the $\infty$-category of presheaves on $C$. A \emph{topos} $\mathcal{X}$ is an $\infty$-category arising as a left exact, accessible localization of a presheaf $\infty$-category. A geometric morphism $f : \mathcal{X} \rightarrow \mathcal{Y}$ between topoi consists of the datum of a right adjoint functor $f_* : \mathcal{X} \rightarrow \mathcal{Y}$ called \emph{pushforward} and its left adjoint $f^* : \mathcal{Y} \rightarrow \mathcal{X}$ called \emph{pullback}, which is assumed to be left exact. The resulting non-full subcategory of $\widehat{\mathrm{Cat}}_\infty$ is denoted as $\mathrm{RTop}$ and called the \emph{$\infty$-category of topoi}. Similarly, we write $\mathrm{LTop} \simeq (\mathrm{RTop})^\mathrm{op}$ as the non-full subcategory of $\widehat{\mathrm{Cat}}_\infty$ spanned by topoi and the left adjoint pullback functors as morphisms. The topos $\mathrm{An} \cong \mathrm{PSh}(\ast)$ is the terminal topos. If $\mathcal{X}$ is a topos, we also denote the unique geometric morphism $\mathcal{X} \rightarrow \mathrm{An}$ by $\mathcal{X}$.

If $\mathcal{X}$ is a topos, and $X \in \mathcal{X}$, then also the overcategory $\mathcal{X}_{/X}$ is a topos. It comes equipped with a natural geometric morphism $X : \mathcal{X}_{/X} \rightarrow \mathcal{X}$. Geometric morphisms equivalent to $\mathcal{X}_{/X} \rightarrow \mathcal{X}$ for some $X \in \mathcal{X}$ as geometric morphisms over $\mathcal{X}$ are also called \emph{étale geometric morphisms}.

If $\mathcal{X}$ is a topos, then we refer to a monomorphism $f : U \hookrightarrow 1$, where $1 \in \mathcal{X}$ is the terminal object, as an \emph{open} of $\mathcal{X}$. The collections of opens, up to isomorphism, of a topos always describes a locale. Conversely, if $L$ is a locale, then the $\infty$-category of \emph{sheaves on $L$}, $\mathrm{Sh}(L)$ is a topos. The assignment of a topological space to its locale, and of a locale to its topos of sheaves forms two adjunctions
\[ \begin{tikzcd}
	{\mathrm{Top}} & {\mathrm{Loc}} & {\mathrm{RTop}}
	\arrow[""{name=0, anchor=center, inner sep=0}, "{\mathrm{Loc}}", curve={height=-12pt}, from=1-1, to=1-2]
	\arrow[""{name=1, anchor=center, inner sep=0}, "{\mathrm{pts}}", curve={height=-12pt}, from=1-2, to=1-1]
	\arrow[""{name=2, anchor=center, inner sep=0}, "{\mathrm{Sh}}"', curve={height=12pt}, hook', from=1-2, to=1-3]
	\arrow[""{name=3, anchor=center, inner sep=0}, "{\mathrm{Open}}"', curve={height=12pt}, from=1-3, to=1-2]
	\arrow["\dashv"{anchor=center, rotate=-90}, draw=none, from=0, to=1]
	\arrow["\dashv"{anchor=center, rotate=-90}, draw=none, from=3, to=2]
\end{tikzcd} \]
where the \emph{sheaves} functor $\mathrm{Loc} \rightarrow \mathrm{RTop}$ is fully faithful.

A morphism $f : X \twoheadrightarrow Y$ in a topos $\mathcal{X}$ is called an \emph{effective epimorphism} if the canonical map
\[ \mathrm{colim}_{[n] \in \Delta^{\mathrm{op}}} X \times_Y \cdots \times_Y X \rightarrow Y \]
is an isomorphism in $\mathcal{X}$, where the colimit in question goes over the \v{C}ech nerve of $f$. In any topos, an arbitrary morphism $f : X \rightarrow Y$ decomposes as $f : X \twoheadrightarrow \mathrm{im}(f) \hookrightarrow Y$, where the first arrow is an effective epimorphism, and the second a monomorphism. If $\{ U_i \}_{i \in I}$ is an open covering of a topological space or locale $X$, then the morphism
\[ \coprod_{i \in I} y_{U_i} \rightarrow 1 \]
is an effective epimorphism in $\mathrm{Sh}(X)$.

The $\infty$-category $\mathrm{RTop}$ of topoi admits all limits and colimits.
\begin{itemize}
\item Colimits in $\mathrm{RTop}$ are computed as limits along the left adjoints, i.e.\ the functor
\[ \mathrm{RTop}^{\mathrm{op}} \rightarrow \widehat{\mathrm{Cat}}_\infty \]
that sends a geometric morphism to its left adjoint part preserves and creates limits; \cite[Corollary 6.3.1.8 and Proposition 6.3.2.3]{Lurie2009HTT}.
\item Products in $\mathrm{RTop}$ are computed using the \emph{Lurie tensor product}, i.e. for topoi $\mathcal{X}, \mathcal{Y}$ it holds that
\[ \mathcal{X} \times^\mathrm{RTop} \mathcal{Y} \simeq \mathcal{X} \otimes \mathcal{Y}, \]
see \cite[Example 4.8.1.19]{Lurie2017HA}.
\item Cofiltered limits in $\mathrm{RTop}$ are computed as cofiltered limits along the right adjoints, i.e.\ the functor 
\[ \mathrm{RTop} \rightarrow \widehat{\mathrm{Cat}}_\infty \]
that sends a geometric morphism to its right adjoint part preserves and creates cofiltered limits; see \cite[Theorem 6.3.3.1]{Lurie2009HTT}.
\item Pullbacks in $\mathrm{RTop}$ exist; see \cite[Proposition 6.3.4.6]{Lurie2009HTT}. Lurie observes that his construction is somewhat inexplicit \cite[Remark 6.3.4.8]{Lurie2009HTT}.
An explicit description is given by Martini--Wolf \cite[Corollary 3.2.7.3 and Remark 3.2.7.4]
{martini2025presentabilitytopoiinternalhigher}.
\end{itemize}

Let us finally mention a crucial property of topoi: \emph{Colimits satisfy descent.} This is expressed as saying that for any topos $\mathcal{X}$, the functor
\[ \begin{array}{rcl}
\mathcal{X}^\mathrm{op} & \rightarrow & \widehat{\mathrm{Cat}}_{\infty} \\
X & \mapsto & \mathcal{X}_{/X}
\end{array} \]
where a morphism $f : X \rightarrow Y$ is sent to the pullback functor $f^* : \mathcal{X}_{/Y} \rightarrow \mathcal{X}_{/X}$, preserves \emph{limits}; \cite[Section 2]{CnossenHigherToposTheory}.

\section{Shape theory of topoi}

Recall that for any small $\infty$-category $C$, we define $\mathrm{Pro}(C) = \mathrm{Ind}(C^\mathrm{op})^\mathrm{op}$. If $C$ admits finite limits, there is an equivalence
\[ \mathrm{Fun}^{\mathrm{lex}}( C, \mathrm{An} )^{\op} \simeq \mathrm{Pro}( C ) \]
which sends a pro-object in $C$ to
\[ \text{``} \lim_{i \in I} \text{''} X_i \mapsto \mathrm{colim}_{i \in I} \mathrm{Map}( X_i, - ). \]
If $C$ is not small, but presentable, we can still use this equivalence by restricting the left-hand side to \emph{accessible} functors, and get
\[ \mathrm{Fun}^{\mathrm{lex, acc}}( C, \mathrm{An} )^{\op} \simeq \mathrm{Pro}( C ). \]
The reasoning here is that, for any cardinal $\kappa$, the subcategory $C^\kappa \subset C$ of $\kappa$-compact objects is a small subcategory. If $F : C \rightarrow \mathrm{An}$ is $\kappa$-accessible, it is determined by its restriction to $C^\kappa$, and we get the corresponding pro-object in $\mathrm{Pro}(C^\kappa)$. In particular, for $C = \mathrm{An}$ we obtain an equivalence
\[ \mathrm{Fun}^{\mathrm{lex, acc}}( \mathrm{An}, \mathrm{An} )^{\op} \simeq \mathrm{Pro}( \mathrm{An} ), \]
see also \cite[Definition 7.1.6.1]{Lurie2009HTT}.

\begin{definition}
Let $\mathcal{X}$ be a topos, and let
\[ \begin{tikzcd}
	{\mathcal{X}} & {\mathrm{An}}
	\arrow[""{name=0, anchor=center, inner sep=0}, "{\mathcal{X}_* = \mathrm{Map}_\mathcal{X}(1,-)}"', curve={height=12pt}, from=1-1, to=1-2]
	\arrow[""{name=1, anchor=center, inner sep=0}, "{\mathcal{X}^*}"', from=1-2, to=1-1]
	\arrow["\dashv"{anchor=center, rotate=-90}, draw=none, from=1, to=0]
\end{tikzcd}
\]
be the associated geometric morphism to $\mathrm{An}$. The shape of $\mathcal{X}$ is the pro-anima $\Pi_{\infty}( \mathcal{X} ) \in \mathrm{Pro}(\mathrm{An})$ associated to the left-exact and accessible functor
\[ \mathcal{X}_* \mathcal{X}^* : \mathrm{An} \rightarrow \mathrm{An}. \]
\end{definition}

If $X$ is a topological space or a locale, we write $\Pi_\infty(X) := \Pi_\infty(\mathrm{Sh}(X))$.

\begin{example}
If $L$ is a locale, and $\mathcal{X} = \mathrm{Sh}(L)$, then the endofunctor $\mathcal{X}_* \mathcal{X}^*$ is computed by sending an anima $A \in \mathrm{An}$ first to the constant presheaf $\underline{A} : \mathcal{O}(L)^\mathrm{op} \rightarrow \mathrm{An}$, then applying sheafification and evaluating at the top element, i.e.
\[
\mathcal{X}_* \mathcal{X}^*(A) = \underline{A}^{\mathrm{sh}}(1).
\]
Thus the endofunctor corresponding to the shape of $L$ should be thought of as the unstable analogue of sheaf cohomology of $L$. As a slogan: \emph{shape measures the effect of sheafification on constant presheaves.}
\end{example}

\begin{example}
Let $C$ be a small $\infty$-category and consider the associated topos $\mathrm{PSh}(C)$. Then the associated geometric morphism to $\mathrm{An}$ is given by
\[
\begin{tikzcd}
	{\mathrm{PSh}(C)} & {\mathrm{An}}
	\arrow[""{name=0, anchor=center, inner sep=0}, "{\lim_{C^\mathrm{op}}}"', curve={height=12pt}, from=1-1, to=1-2]
	\arrow[""{name=1, anchor=center, inner sep=0}, "{\mathrm{const}}"', from=1-2, to=1-1]
	\arrow["\dashv"{anchor=center, rotate=-90}, draw=none, from=1, to=0]
\end{tikzcd}
\]
Therefore the endofunctor corresponding to the shape of $\mathrm{PSh}(C)$ is computed as
\[
\mathcal{X}_* \mathcal{X}^*(A)
= \lim_{C^\mathrm{op}} \mathrm{const}(A)
\simeq \mathrm{Map}(|C|, A).
\]
We thus see that the pro-object corresponding to the endofunctor is in fact constant, and given by
\[ \Pi_{\infty}( \mathrm{PSh}(C) ) \simeq |C|. \]
\end{example}

Let us talk about functoriality of the shape. If $f : \mathcal{X} \rightarrow \mathcal{Y}$, one has a commuting triangle
\[
\begin{tikzcd}
	{\mathcal{X}} && {\mathcal{Y}} \\
	& {\mathrm{An}}
	\arrow["{f_*}", from=1-1, to=1-3]
	\arrow["{\mathcal{X}_*}"', from=1-1, to=2-2]
	\arrow["{\mathcal{Y}_*}", from=1-3, to=2-2]
\end{tikzcd}
\]
Thus we can take the unit $ \mathrm{id}_\mathcal{Y} \implies f_* f^* $ to get a natural transformation
\[
\mathcal{Y}_* \mathcal{Y}^*
\implies
\mathcal{Y}_* f_* f^* \mathcal{Y}^*
\simeq
\mathcal{X}_* \mathcal{X}^*.
\]
Under the equivalence $ \mathrm{Fun}^{\mathrm{lex, acc}}( \mathrm{An}, \mathrm{An} )^{\op} \simeq \mathrm{Pro}( \mathrm{An} )$, we thus obtain a map of pro-anima
\[ \Pi_\infty(f) : \Pi_\infty(\mathcal{X}) \rightarrow \Pi_\infty(\mathcal{Y}). \]
We can actually see that this provides us with a covariant functor
\[ \Pi_\infty : \mathrm{RTop} \rightarrow \mathrm{Pro}( \mathrm{An} ), \]
a fact that we can make precise with the following remark.

\begin{remark}
We briefly sketch how to actually construct the shape of topoi as a functor. There exists a universal evaluation functor
\[
\begin{array}{rcl}
     \mathrm{Fun}( \Delta^1, \mathrm{Cat}_\infty ) \times_{\mathrm{Cat}_\infty} {\mathrm{Cat}_\infty}_{*/} & \rightarrow & {\mathrm{Cat}_\infty}_{*/}  \\
     ( F : \mathcal{C} \rightarrow \mathcal{D}, c \in \mathcal{C} ) & \mapsto & ( \mathcal{D}, F(c) )
\end{array}
\]
constructed from the universal cocartesian fibration $(\mathrm{Cat}_\infty)_{*//} \rightarrow \mathrm{Cat}_\infty$. From this observation, we can use the fact that $\mathrm{An}$ is terminal in $\mathrm{RTop}$ to construct a functor
\[
\begin{array}{rcl}
     \mathrm{RTop}^{\mathrm{op}} \times \mathrm{An} & \rightarrow & {\mathrm{Cat}_\infty}_{*/}  \\
     ( \mathcal{X}, K ) & \mapsto & ( \mathcal{X}, \mathcal{X}^*(K) \in  \mathcal{X} )
\end{array}
\]
Applying the mapping space functor, we get
\[
\begin{array}{rcl}
     \mathrm{RTop}^{\mathrm{op}} \times \mathrm{An} & \rightarrow & \mathrm{An}  \\
     ( \mathcal{X}, K ) & \mapsto & \mathrm{Map}_{\mathcal{X}}( 1_\mathcal{X}, \mathcal{X}^*(K) )
\end{array}
\]
Under currying, we observe that the resulting functor lands in the full subcategory $\mathrm{Fun}^{\mathrm{lex, acc}}(\mathrm{An}, \mathrm{An})$.
\end{remark}

Now that we have the shape functor, it makes sense to give the following definition.

\begin{definition}
Let $f : \mathcal{X} \rightarrow \mathcal{Y}$ be a geometric morphism between topoi. We call $f$ a \emph{shape equivalence} if $\Pi_\infty(f)$ is an equivalence of pro-anima.
\end{definition}

The following lemma is an easy but very useful observation.

\begin{lemma} \label{contractibleshapeequivalence}
Let $f : \mathcal{X} \rightarrow \mathcal{Y}$ be a geometric morphism between topoi such that $f^*$ is fully faithful. Then $f$ is a shape equivalence.
\end{lemma}

\begin{proof}
Full faithfulness of $f^*$ is equivalent to the unit
\[  \mathrm{id}_\mathcal{Y} \implies f_* f^* \]
being a natural equivalence. But this means that we have equivalences
\[
\mathcal{Y}_* \mathcal{Y}^*
\simeq
\mathcal{Y}_* f_* f^* \mathcal{Y}^*
\simeq
\mathcal{X}_* \mathcal{X}^*,
\]
in other words, $\Pi_\infty(\mathcal{X}) \simeq \Pi_\infty(\mathcal{Y})$.
\end{proof}

This observation motivates the following terminology.

\begin{definition}
A geometric morphism $f : \mathcal{X} \rightarrow \mathcal{Y}$ is called \emph{contractible} if $f^*$ is fully faithful. A topos $\mathcal{X}$ is called \emph{contractible} if the unique geometric morphism $\mathcal{X} \rightarrow \mathrm{An}$ is contractible.
\end{definition}

\begin{example}
Let $C$ be a small $\infty$-category. Then the geometric morphism induced by $\mathrm{can} : C \rightarrow |C| = C[C^{-1}]$,
\[ \mathrm{PSh}(C) \rightarrow \mathrm{PSh}(|C|), \]
is contractible. This can be seen by observing that $\mathrm{can}^*$ identifies
\[ \mathrm{PSh}(|C|) = \mathrm{Fun}(|C|^{\op}, \mathrm{An} ) \]
with the full subcategory of
\[ \mathrm{PSh}(C) = \mathrm{Fun}(C^{\op}, \mathrm{An} ) \]
spanned by those presheaves $F : C^{\op} \rightarrow \mathrm{An}$ that send all arrows to equivalences.
\end{example}

In case we are talking about contractibility of a topos $\mathcal{X}$, Lemma \ref{contractibleshapeequivalence} becomes an if and only if statement.

\begin{proposition}
A topos $\mathcal{X}$ is contractible iff $\Pi_\infty( \mathcal{X} ) \simeq \mathrm{pt}$.
\end{proposition}

\begin{proof}
The functor $\mathcal{X}^*$ is fully faithful iff the unit map $\mathrm{id}_{\mathrm{An}} \implies \mathcal{X}_* \mathcal{X}^*$ is an equivalence. Translating to pro-anima, this just means $\Pi_\infty( \mathcal{X} ) \simeq \mathrm{pt}$.
\end{proof}

\begin{remark}
It is actually true that the shape
\[ \Pi_\infty : \mathrm{RTop} \rightarrow \mathrm{Pro}(\mathrm{An}) \]
is a $2$-functor, with $\mathrm{RTop}$ viewed as an $(\infty,2)$-category, and with $2$-cells given by natural transformations. Here, $\mathrm{Pro}(\mathrm{An})$ is viewed as a truncated $(\infty,2)$-category. Since we do not want to set up the theory of $(\infty,2)$-categories for the present article, let us show how this functoriality looks like on natural transformations. If $f, g : \mathcal{X} \rightarrow \mathcal{Y}$ are two geometric morphisms, and $\alpha^* : f^* \implies g^*$ is a natural transformation with mate $\alpha_* : g_* \implies f_*$, we have the commutative square
\[ \begin{tikzcd}
	{\mathrm{id}_\mathcal{Y}} & {f_*f^*} \\
	{g_*g^*} & {f_*g^*}
	\arrow[from=1-1, to=1-2]
	\arrow[from=1-1, to=2-1]
	\arrow["{f_*( \alpha^*)}", from=1-2, to=2-2]
	\arrow["{(\alpha_*)_{g^*}}"', from=2-1, to=2-2]
\end{tikzcd} \]
of natural transformations. Pre- and post-composing with $\mathcal{Y}^*$, respectively $\mathcal{Y}_*$, we obtain the commutative diagram
\[\begin{tikzcd}
	{\mathcal{Y}_*\mathcal{Y}^*} & {\mathcal{Y}_*f_*f^*\mathcal{Y}^*} & \\
	{\mathcal{Y}_*g_*g^*\mathcal{Y}^*} & {\mathcal{Y}_*f_*g^*\mathcal{Y}^*} \\
	&& {\mathcal{X}_* \mathcal{X}^*}
	\arrow[from=1-1, to=1-2]
	\arrow[from=1-1, to=2-1]
	\arrow["\sim", from=1-2, to=2-2]
	\arrow["\sim", curve={height=-12pt}, from=1-2, to=3-3]
	\arrow["\sim"', from=2-1, to=2-2]
	\arrow["\sim"', curve={height=12pt}, from=2-1, to=3-3]
	\arrow["\sim"', from=2-2, to=3-3]
\end{tikzcd}\]
But this means that the induced morphisms
\[ f, g : \Pi_\infty(\mathcal{X}) \rightarrow \Pi_\infty(\mathcal{Y}) \]
are homotopic in $\mathrm{Pro}(\mathrm{An})$. In particular, we see that the functor $\Pi_\infty$ inverts internal adjunctions in $\mathrm{RTop}$, in other words, internal adjoints in $\mathrm{RTop}$ are automatically shape equivalences. Of course, a precise $2$-categorical formulation of the functoriality of the shape should be addressed in future work.\footnote{We thank Thorger Geiß for a clarifying discussion on this point.}
\end{remark}

\subsection{The local vs.\ global picture on shape}

The adjoint functor theorem states that if $R : \mathcal{D} \rightarrow \mathcal{C}$ is an accessible and limit-preserving functor between presentable $\infty$-categories $\mathcal{C}$ and $\mathcal{D}$, then $R$ has a left adjoint $L$. So what should one do if $R$ does not preserve all limits, yet one still wants to use a potential left adjoint?

\begin{definition} \label{proleft}
Suppose $\mathcal{C}$ and $\mathcal{D}$ are presentable $\infty$-categories, and let $F : \mathcal{D} \rightarrow \mathcal{C}$ be an accessible and finite-limit-preserving functor. The pro-left adjoint to $F$ is defined as the functor
\[
\begin{array}{rcl}
F^{\mathrm{pro}\text{-}L} : \mathcal{C} & \rightarrow & \mathrm{Fun}^{\mathrm{lex, acc}}( \mathcal{D}, \mathrm{An} )^{\op} \simeq \mathrm{Pro}(\mathcal{D}) \\
c & \mapsto & \mathrm{Map}_\mathcal{C}( c, F(-)).
\end{array}
\]
\end{definition}

It is straightforward to construct an actual functor that realizes the assignment of pro-left adjoints. Crucially, since $\mathcal{C}$ is presentable, $c$ is $\kappa$-compact for some cardinal $\kappa$, so the expression does in fact land in accessible functors. Let us justify the terminology \emph{pro-left adjoint}.

\begin{proposition}
Let $F : \mathcal{D} \rightarrow \mathcal{C}$ be a functor satisfying the conditions of Definition \ref{proleft}. Then there exists an adjunction
\[
\begin{tikzcd}
	{\mathcal{C}} & {\mathrm{Pro}(\mathcal{D})}
	\arrow[""{name=0, anchor=center, inner sep=0}, "{F^{\mathrm{pro}\text{-}L}}", curve={height=-12pt}, from=1-1, to=1-2]
	\arrow[""{name=1, anchor=center, inner sep=0}, curve={height=-12pt}, from=1-2, to=1-1]
	\arrow["\dashv"{anchor=center, rotate=-90}, draw=none, from=0, to=1]
\end{tikzcd}
\]
with the right adjoint given by the pro-extension of $F$.
\end{proposition}

\begin{proof}
Fix $c \in \mathcal{C}$ and $d \in \mathcal{D}$. Using the equivalence
\[
\mathrm{Pro}( \mathcal{D} ) \simeq \mathrm{Fun}^{\mathrm{lex, acc}}( \mathcal{D}, \mathrm{An} )^{\op}
\]
together with the Yoneda lemma, we see that we have a natural equivalence
\[
\mathrm{Map}_{\mathrm{Pro}( \mathcal{D} )}( F^{\mathrm{pro}\text{-}L}(c), d )
\simeq
\mathrm{Nat}\bigl( y_d, \mathrm{Map}_\mathcal{C}( c, F(-)) \bigr)
\simeq
\mathrm{Map}_{\mathcal{C}}(c, F(d)).
\]
The same equivalence for an arbitrary pro-object $\text{``}\lim_{i \in I}\text{''} d_i \in \mathrm{Pro}( \mathcal{D} )$ follows formally by cofiltered-limit extension.
\end{proof}

\begin{example} \label{partitions}
Consider the discrete topology functor 
\[ (-)^{\mathrm{disc}} : \mathrm{Set} \rightarrow \mathrm{Top}. \]
This functor has a pro-left adjoint
\[ \pi_0 : \mathrm{Top} \rightarrow \mathrm{Pro}( \mathrm{Set} ) \]
which sends a topological space $X$ to the pro-set
\[ \text{``} \mathrm{lim}_{P \in \mathrm{Part}(X)} \text{''} P \]
where the poset $\mathrm{Part}(X)$ is the set of \emph{partitions} of $X$. A partition $P = \{ U_i \}_{i \in I}$ of $X$ is a set of pairwise disjoint open sets such that $X = \bigcup_{ i\in I} U_i$. We note that the analogous construction also works for locales.
\end{example}

The notion of pro-left adjoint is a conservative extension of the notion of left adjoint, in the following sense.

\begin{lemma} \label{actualleftadjoint}
Let $F : \mathcal{D} \rightarrow \mathcal{C}$ be a functor satisfying the conditions of Definition \ref{proleft}. Then $F$ admits a left adjoint iff the pro-left adjoint $F^{\mathrm{pro}\text{-}L} : \mathcal{C} \rightarrow \mathrm{Pro}(\mathcal{D})$ factors through $y : \mathcal{D} \rightarrow \mathrm{Pro}(\mathcal{D})$, in which case the factorization $F^{\mathrm{pro}\text{-}L} : \mathcal{C} \rightarrow \mathcal{D}$ is the left adjoint to $F$. Moreover, this is the case iff $F$ preserves all limits, iff $F$ preserves products of arbitrary infinite cardinality.
\end{lemma}

\begin{proof}
Suppose $F$ admits a left adjoint $L : \mathcal{C} \rightarrow \mathcal{D}$. Then there is a natural equivalence
\[
\mathrm{Map}_\mathcal{C}( c, F(-)) \simeq \mathrm{Map}_\mathcal{D}( L(c), -)
\]
for all $c \in \mathcal{C}$, hence we see that $F^{\mathrm{pro}\text{-}L}$ factors through $y$ as $L$.

Conversely, assume $F^{\mathrm{pro}\text{-}L}$ factors through $y$ as a functor $L : \mathcal{C} \rightarrow \mathcal{D}$, or equivalently, there is a natural equivalence in $c \in \mathcal{C}$
\[
\mathrm{Map}_\mathcal{C}( c, F(-)) \simeq \mathrm{Map}_\mathcal{D}( L(c), -).
\]
But this is just saying that $L$ is left adjoint to $F$.

The claim that this is equivalent to $F$ preserving all limits follows by the adjoint functor theorem. Since $F$ already preserves finite limits by assumption, preserving all limits is equivalent to preserving products of arbitrary infinite cardinality.
\end{proof}

\begin{definition}
Let $\mathcal{X}$ be a topos. The \emph{local shape} $\Pi_\infty^\mathcal{X} : \mathcal{X} \rightarrow \mathrm{Pro}(\mathrm{An})$ is defined to be the pro-left adjoint to $\mathcal{X}^* : \mathrm{An} \rightarrow \mathcal{X}$.
\end{definition}

By definition, plugging in the terminal object $1 \in \mathcal{X}$, we see that
\[
\mathrm{Map}_\mathcal{X}( 1, \mathcal{X}^*(-)) \simeq \mathcal{X}_* \mathcal{X}^*,
\]
hence $\Pi_\infty^\mathcal{X}( 1 ) \simeq \Pi_\infty(\mathcal{X})$. Thus the local shape contains the shape of $\mathcal{X}$ as a special case, namely the global value.

If $U$ is an object of $\mathcal{X}$, there is yet another way to define the shape of $U$, namely via $\Pi_\infty(\mathcal{X}_{/U})$. Note that we have the adjunctions
\[
\begin{tikzcd}
	{\mathcal{X}_{/U}} & {\mathcal{X}}
	\arrow[""{name=0, anchor=center, inner sep=0}, "{U_*}"', curve={height=12pt}, from=1-1, to=1-2]
	\arrow[""{name=1, anchor=center, inner sep=0}, "{U_!}", curve={height=-12pt}, from=1-1, to=1-2]
	\arrow[""{name=2, anchor=center, inner sep=0}, "{U^*}"{description}, from=1-2, to=1-1]
	\arrow["\dashv"{anchor=center, rotate=-89}, draw=none, from=1, to=2]
	\arrow["\dashv"{anchor=center, rotate=-91}, draw=none, from=2, to=0]
\end{tikzcd}
\]
where $U_!$ is the forgetful functor from the overcategory, and $U^* = (-) \times U$. By composing the associated pro-left adjoints to
\[
\mathrm{An} \xrightarrow{ \mathcal{X}^* } \mathcal{X} \xrightarrow{ U^* } \mathcal{X}_{/U},
\]
we see that
\[
\Pi_\infty(\mathcal{X}_{/U})
\simeq
\Pi_\infty^{\mathcal{X}_{/U}}(1)
\simeq
\Pi_\infty^\mathcal{X}(U),
\]
so also in this case the local and global versions of shape match.

The fact that $\Pi_\infty$ can be locally described as a left adjoint implies the following immediate useful consequence.

\begin{theorem}[Descent for shape] \label{shapemayervietorisdescent}
Let $\mathcal{X}$ be a topos and suppose $U \twoheadrightarrow 1$ is an effective epimorphism. Then
\[
\Pi_\infty(\mathcal{X}) \simeq \colim_{[n] \in \Delta^{\mathrm{op}}} \Pi^\mathcal{X}_\infty(U^{\times n+1}),
\]
where the colimit in question is taken in the $\infty$-category $\mathrm{Pro}(\mathrm{An})$.
\end{theorem}

This follows immediately from the fact that, for an effective epimorphism $U \twoheadrightarrow 1$,
\[
1 \simeq  \colim_{[n] \in \Delta^{\mathrm{op}}} U^{\times n+1}
\]
holds in $\mathcal{X}$. This has an immediate application. Let $L$ be a locale. Let us write $\Pi_\infty(L) = \Pi_\infty( \mathrm{Sh}(L) )$ for the shape of $L$.

\begin{corollary}[Shape-excision for open coverings] \label{shapemayervietorisopens}
Let $X$ be a topological space or a locale, and let $(U_i)_{i \in I}$ be an open cover of $X$. Then
\[
\Pi_\infty(X) \simeq \colim \Pi_\infty(U_{i_0} \cap \dots \cap U_{i_n}),
\]
where the colimit in question ranges over the \v{C}ech nerve of the cover, and is taken in the $\infty$-category $\mathrm{Pro}(\mathrm{An})$.
\end{corollary}

\begin{definition}
Let $L$ be a locale. An open cover $(U_i)_{i \in I}$ of $L$ is called \emph{shape-good}, or simply \emph{good}, if for all non-zero finite intersections
\[
U_{i_0} \cap \dots \cap U_{i_n} \neq 0,
\]
the topoi $\mathrm{Sh}(U_{i_0} \cap \dots \cap U_{i_n})$ have trivial shape.
\end{definition}

Given an open cover $\mathcal{U} = (U_i)_{i \in I}$ of a locale $L$, define an abstract simplicial complex $\check{C}(\mathcal{U})$, called the \emph{\v{C}ech complex} of $\mathcal{U}$, as follows: the underlying set of vertices is the set $I$, and the faces are the finite subsets $\{ i_0, \dots, i_n \}$ such that
\[
U_{i_0} \cap \dots \cap U_{i_n} \neq 0.
\]
The following is an immediate application of Corollary \ref{shapemayervietorisopens}.

\begin{proposition} \label{goodopencover}
Suppose $L$ is a locale admitting a good open cover $(U_i)_{i \in I}$. Then the shape of $L$ is constant, and there is an equivalence
\[
\Pi_\infty( L ) \simeq | \check{C}(\mathcal{U}) |.
\]
\end{proposition}

\begin{example}
Let $M$ be an $n$-dimensional smooth and paracompact manifold. Then $M$ admits a Riemannian structure, for which it makes sense to talk about geodesics. It is then possible to cover $M$ by open, geodesically convex subsets, which produces a good open cover.
\end{example}

\begin{example} \label{abstractsimplicialcomplex}
Let $K$ be an abstract simplicial complex. Then the set of open stars $\mathrm{st}(\sigma, K)$ for all faces $\sigma$ of $K$ forms a good open cover of $|K|$; see \cite[§62, Lemma 62.6]{MunkresEAT} and
\cite[pp.~8--9]{HainSimplicialComplexes}.
\end{example}

\section{Locally contractible topoi}

Let us discuss a class of topoi that behave particularly well for the purposes of shape theory. The main reference for this section is \cite[Appendix A.1]{Lurie2017HA}.

\begin{definition}
A geometric morphism $f : \mathcal{X} \rightarrow \mathcal{Y}$ is called \emph{essential} if $f^*$ has a further left adjoint $f_!$,
\[
\begin{tikzcd}
	{\mathcal{X}} & {\mathcal{Y}.}
	\arrow[""{name=0, anchor=center, inner sep=0}, "{f_*}"', curve={height=12pt}, from=1-1, to=1-2]
	\arrow[""{name=1, anchor=center, inner sep=0}, "{f_!}", curve={height=-12pt}, from=1-1, to=1-2]
	\arrow[""{name=2, anchor=center, inner sep=0}, "{f^*}"{description}, from=1-2, to=1-1]
	\arrow["\dashv"{anchor=center, rotate=-90}, draw=none, from=1, to=2]
	\arrow["\dashv"{anchor=center, rotate=-90}, draw=none, from=2, to=0]
\end{tikzcd}
\]
\end{definition}

We observe that essential geometric morphisms are closed under composition.

\begin{example}
Let $f : C \rightarrow D$ be a functor between small $\infty$-categories. Then the induced geometric morphism on presheaf topoi is essential, as we have the adjunctions
\[
\begin{tikzcd}
	{\mathrm{PSh}(C)} & {\mathrm{PSh}(D).}
	\arrow[""{name=0, anchor=center, inner sep=0}, "{\mathrm{Lan}_f}", curve={height=-12pt}, from=1-1, to=1-2]
	\arrow[""{name=1, anchor=center, inner sep=0}, "{\mathrm{Ran}_f}"', curve={height=12pt}, from=1-1, to=1-2]
	\arrow[""{name=2, anchor=center, inner sep=0}, "{f^\circ}"{description}, from=1-2, to=1-1]
	\arrow["\dashv"{anchor=center, rotate=-90}, draw=none, from=0, to=2]
	\arrow["\dashv"{anchor=center, rotate=-90}, draw=none, from=2, to=1]
\end{tikzcd}
\]
\end{example}

For our purposes, we are mostly interested in the following special case.

\begin{definition}
A topos $\mathcal{X}$ is called locally contractible if the unique geometric morphism $\mathcal{X} \rightarrow \mathrm{An}$ is essential.
\end{definition}

\begin{example}
Presheaf topoi $\mathrm{PSh}(C)$, for $C$ a small $\infty$-category, are locally contractible. This follows directly from the previous observation applied to the unique functor $C \rightarrow \mathrm{pt}$.
\end{example}

\begin{example}
As a special case of the previous example, we see that for an anima $A \in \mathrm{An}$, the overtopos $\mathrm{An}_{/A} \simeq \mathrm{Fun}(A, \mathrm{An})$ is locally contractible. This is actually also a special case of the following example.
\end{example}

\begin{example}
Let $\mathcal{X}$ be locally contractible, and let $U \in \mathcal{X}$ be an object. Then the étale topos $\mathcal{X}_{/U}$ is again locally contractible. This follows since the unique geometric morphism $\mathcal{X}_{/U} \rightarrow \mathrm{An}$ factors as the composition
\[
\mathcal{X}_{/U} \rightarrow \mathcal{X} \rightarrow \mathrm{An}
\]
of two essential geometric morphisms.
\end{example}

The following is a direct application of Lemma \ref{actualleftadjoint}.

\begin{proposition}
Let $\mathcal{X}$ be a topos. The following are equivalent.
\begin{enumerate}
\item $\mathcal{X}$ is locally contractible.
\item The local shape functor $\Pi_\infty^\mathcal{X} : \mathcal{X} \rightarrow \mathrm{Pro}(\mathrm{An})$ takes values in constant pro-objects.
\item The constant object functor $\mathcal{X}^* : \mathrm{An} \rightarrow \mathcal{X}$ preserves limits.
\item The constant object functor $\mathcal{X}^* : \mathrm{An} \rightarrow \mathcal{X}$ preserves arbitrary products.
\item The functor $\mathcal{X}_* \mathcal{X}^* : \mathrm{An} \rightarrow \mathrm{An}$ is corepresentable.
\end{enumerate}
\end{proposition}

Note that if condition (5) holds, it follows that the representing object is automatically $\Pi_\infty(\mathcal{X}) \in \mathrm{An}$. Let us justify the terminology \emph{locally} contractible.

\begin{definition}
Let $\mathcal{X}$ be a topos. We say that $\mathcal{X}$ has \emph{constant shape} if $\Pi_\infty(\mathcal{X}) \in \mathrm{An}$. More generally, an object $U \in \mathcal{X}$ is said to have constant shape if $\Pi^\mathcal{X}_\infty(U) \in \mathrm{An}$.
\end{definition}

We immediately see the following.

\begin{lemma}
A topos $\mathcal{X}$ is locally contractible if and only if all objects $U \in \mathcal{X}$ have constant shape.
\end{lemma}

The main usefulness of this notion comes from the following lemma.

\begin{lemma}[{\cite[Proposition A.1.6]{Lurie2017HA}}]
Let $\mathcal{X}$ be a topos. The full subcategory $\mathcal{X}'$ of $\mathcal{X}$ spanned by objects of constant shape is closed under colimits.
\end{lemma}

\begin{proposition}[{\cite[Corollary A.1.7]{Lurie2017HA}}]
Let $\mathcal{X}$ be a topos. Then $\mathcal{X}$ is locally contractible if and only if there exists an effective epimorphism
\[
\begin{tikzcd}
	{\coprod_{i \in I} U_i} & 1
	\arrow[two heads, from=1-1, to=1-2]
\end{tikzcd}
\]
such that the topoi $\mathcal{X}_{/U_i}$ are locally contractible for all $i \in I$.
\end{proposition}

\begin{example}
Let $L$ be a locale. We see from the above proposition that if $(U_i)_{i \in I}$ is an open covering of $L$ such that each $\mathrm{Sh}(U_i)$ is a locally contractible topos, then $\mathrm{Sh}(L)$ is locally contractible as well. Therefore, local contractibility of certain classes of locales, for example manifolds, which are locally discs in $\mathbb{R}^n$, or CW complexes, can be shown by reducing to local contractibility for a basic class of spaces.
\end{example}

We still have not fully justified the terminology \emph{locally contractible}. To do so, consider the following useful theorem.

\begin{theorem}
Let $\mathcal{X}$ be a topos and assume there exists a collection of objects $\{ U_i \in \mathcal{X} \}_{ i \in I }$ generating $\mathcal{X}$ under colimits, such that $\mathcal{X}_{/U_i}$ is a contractible topos for all $i \in I$. Then $\mathcal{X}$ is a locally contractible topos.
\end{theorem}

\begin{proof}
We need to show that $\mathcal{X}^* : \mathrm{An} \rightarrow \mathcal{X}$ preserves limits. Let $A_\bullet : C \rightarrow \mathrm{An}$ be a diagram and consider the natural comparison map
\[ \mathcal{X}^*( \mathrm{lim}_{c \in C} A_c ) \rightarrow \mathrm{lim}_{c \in C} \mathcal{X}^*(A_c). \]
Since $\mathcal{X}$ is generated by the objects $U_i$ under colimits, it follows that the functors $\mathrm{Map}_\mathcal{X}(U_i, -)$  are jointly conservative, hence it suffices to check the above map is an equivalence after applying $\mathrm{Map}_\mathcal{X}(U_i, -)$ for all $i \in I$. Write $q_i : \mathcal{X}_{/U_i} \rightarrow \mathcal{X}$ for the associated étale geometric morphism. We see that
\[ \begin{array}{rcl}
\mathrm{Map}_\mathcal{X}(U_i, \mathcal{X}^*( \mathrm{lim}_{c \in C} A_c )) & \cong & \mathrm{Map}_{\mathcal{X}_{/U_i}}(1_{U_i}, q_i^* \mathcal{X}^*( \mathrm{lim}_{c \in C} A_c ))  \\ 
& \cong & \mathrm{Map}_{\mathcal{X}_{/U_i}}(1_{U_i}, U_i^*( \mathrm{lim}_{c \in C} A_c )) \\
& \cong & (U_i)_* U_i^*( \mathrm{lim}_{c \in C} A_c ) \\
& \cong & \mathrm{lim}_{c \in C} A_c , \\
\end{array} \]
where the last line follows since the shape of $U_i$ is $\ast$, equivalently $(U_i)_* U_i^* \cong \mathrm{id}_{\mathrm{An}}$. Similarly, we see on the other side that we have
\[ \begin{array}{rcl}
\mathrm{Map}_\mathcal{X}(U_i, \mathrm{lim}_{c \in C} \mathcal{X}^*(A_c) ) & \cong & \mathrm{lim}_{c \in C} \mathrm{Map}_\mathcal{X}(U_i, \mathcal{X}^*(A_c) ) \\
& \cong & \mathrm{lim}_{c \in C} \mathrm{Map}_{\mathcal{X}_{/U_i}}(1_{U_i}, q_i^*\mathcal{X}^*(A_c) ) \\
& \cong & \mathrm{lim}_{c \in C} (U_i)_* U_i^*(A_c)  \\
& \cong & \mathrm{lim}_{c \in C} A_c .
\end{array} \]
\end{proof}

\begin{corollary}
Let $(C,\tau)$ be a Grothendieck topology on an $\infty$-category $C$. Suppose that for each $c \in C$ the topos $\mathrm{Sh}(C,\tau)_{/j(c)}$ is contractible. Then
\[ \mathrm{Sh}(C,\tau) \]
is a locally contractible topos.
\end{corollary}

\begin{corollary} \label{locallycontractibleislocallycontractible}
Let $X$ be a topological space or a locale. Assume $\mathcal{B} \subset \mathcal{O}(X)$ is a basis (closed under binary intersection!) such that $U$ has contractible shape for all $U \in \mathcal{B}$. Then $\mathrm{Sh}(X)$ is a locally contractible topos.
\end{corollary}

\begin{proof}
By the comparison lemma \cite[Lemma C.3]{Hoyois_2014}, for a basis $\mathcal{B}$ of $X$ one has
\[ \mathrm{Sh}(X) \simeq \mathrm{Sh}(\mathcal{B},\tau), \]
where $\tau$ is the induced Grothendieck topology on $\mathcal{B}$.
\end{proof}

Any locally contractible topos admits a canonical geometric morphism to an étale topos over $\mathrm{An}$, as laid out in \cite[Appendix A.1, p.\ 1405--1407]{Lurie2017HA}. Let us give a quick overview. If $\mathcal{X}$ is locally contractible, then by formal nonsense the functor
\[
\Pi^\mathcal{X}_\infty : \mathcal{X} \rightarrow \mathrm{An}
\]
lifts to a functor
\[
\psi_! : \mathcal{X} \rightarrow \mathrm{An}_{/ \Pi^\mathcal{X}_\infty(1)} = \mathrm{An}_{/ \Pi_\infty(\mathcal{X})}.
\]
This functor admits, again purely formally, a right adjoint
\[
\psi^* : \mathrm{An}_{/ \Pi_\infty(\mathcal{X})} \rightarrow \mathcal{X}
\]
which sends a map $Z \rightarrow \Pi_\infty(\mathcal{X})$ of anima to the pullback
\[
\begin{tikzcd}
	{\psi^*(Z)} & {\mathcal{X}^*(Z)} \\
	{1_\mathcal{X}} & {\mathcal{X}^*( \Pi_\infty(\mathcal{X}) )}
	\arrow[from=1-1, to=1-2]
	\arrow[from=1-1, to=2-1]
	\arrow["\lrcorner"{anchor=center, pos=0.125}, draw=none, from=1-1, to=2-2]
	\arrow[from=1-2, to=2-2]
	\arrow[from=2-1, to=2-2]
\end{tikzcd}
\]
Since colimits in $\mathcal{X}$ satisfy descent, we see that $\psi^*$ preserves colimits, and thus admits a further right adjoint $\psi_*$. We therefore obtain the sequence of adjoints
\[
\begin{tikzcd}
	{\mathcal{X}} & {\mathrm{An}_{/ \Pi_\infty( \mathcal{X}) }}
	\arrow[""{name=0, anchor=center, inner sep=0}, "{\psi_*}"', curve={height=12pt}, from=1-1, to=1-2]
	\arrow[""{name=1, anchor=center, inner sep=0}, "{\psi_!}", curve={height=-12pt}, from=1-1, to=1-2]
	\arrow[""{name=2, anchor=center, inner sep=0}, "{\psi^*}"{description}, from=1-2, to=1-1]
	\arrow["\dashv"{anchor=center, rotate=-90}, draw=none, from=1, to=2]
	\arrow["\dashv"{anchor=center, rotate=-90}, draw=none, from=2, to=0]
\end{tikzcd}
\]
in other words, an essential geometric morphism
\[
\psi : \mathcal{X} \rightarrow \mathrm{An}_{/ \Pi_\infty( \mathcal{X}) }.
\]

\begin{theorem}[Galois theory of locally contractible topoi, {\cite[Theorem A.1.15 and Corollary A.1.16]{Lurie2017HA}}]  \label{galoistheory}
Let $\mathcal{X}$ be a locally contractible topos. Then the functor
\[
\Pi^\mathcal{X}_\infty : \mathcal{X} \rightarrow \mathrm{An}
\]
induces a canonical geometric morphism
\[
\psi : \mathcal{X} \rightarrow \mathrm{An}_{/ \Pi_\infty( \mathcal{X} ) }.
\]
The geometric morphism $\psi$ is essential and contractible, hence in particular induces a shape equivalence.
\end{theorem}

Objects in the essential image of $
\psi^* : \mathrm{An}_{/ \Pi_\infty( \mathcal{X} )} \rightarrow \mathcal{X} $
are called \emph{locally constant} objects.

\begin{example}
Let $M$ be an \emph{aspherical} manifold, that is, $M$ is connected and $M \simeq B \pi_1(M)$ is a model for the classifying space of its fundamental group. A particular example is the torus $T^2 = S^1 \times S^1$, with fundamental group $\pi_1(T^2) \simeq \mathbb{Z}^2$.

We then obtain that the full subcategory of $\mathrm{Sh}(M)$ spanned by the locally constant sheaves agrees with the category $\mathrm{Fun}(B \pi_1(M), \mathrm{An})$ of anima with a $\pi_1(M)$-action. This is the $\infty$-topos version of a more classical statement in covering theory that relates the category of covering spaces over $M$ with the category of $\pi_1(M)$-sets.
\end{example}

Let us give an identification of the global sections functor when applied to locally constant objects. Recall the equivalence
\[ \mathrm{An}_{/ \Pi_\infty( \mathcal{X} ) } \simeq \mathrm{Fun}(  \Pi_\infty( \mathcal{X} ) , \mathrm{An}). \]

\begin{proposition} \label{globalsectionslocallycontractible}
Let $\mathcal{X}$ be a locally contractible topos. Then there exists a commuting triangle
\[ \begin{tikzcd}
	{\mathcal{X}} & {\mathrm{An}.} \\
	{\mathrm{Fun}(\Pi_\infty(\mathcal{X}),\mathrm{An})}
	\arrow["{\mathcal{X}_*}", from=1-1, to=1-2]
	\arrow["{\psi^*}", from=2-1, to=1-1]
	\arrow["{\mathrm{lim}_{\Pi_\infty(\mathcal{X})}}"', from=2-1, to=1-2]
\end{tikzcd} \]
\end{proposition}

\begin{proof}
It suffices to show that the triangle of left adjoints
\[\begin{tikzcd}
	{\mathcal{X}} & {\mathrm{An}} \\
	{\mathrm{An}_{/ \Pi_\infty(\mathcal{X})}}
	\arrow["{\psi_!}"', from=1-1, to=2-1]
	\arrow["{\mathcal{X}^*}"', from=1-2, to=1-1]
	\arrow["{- \times \Pi_\infty(\mathcal{X})}"{pos=0.3}, from=1-2, to=2-1]
\end{tikzcd}\]
commutes. To see this, note that the composition
\[ \mathrm{An} \xrightarrow{\mathcal{X}^*} \mathcal{X} \xrightarrow{ \psi_! } \mathrm{An}_{/\Pi_\infty(\mathcal{X})} \rightarrow \mathrm{An} \]
consists of three left adjoints. It is thus colimit preserving and hence determined by its value on the point $\ast$, which is given as
\[ \Pi_\infty^\mathcal{X}( \mathcal{X}^*(\ast) ) \cong \Pi_\infty^\mathcal{X}( 1 ) \cong \Pi_\infty( \mathcal{X} ). \] 
\end{proof}

\section{The shape adjunction}

The following formula describes geometric morphisms into étale topoi.
\begin{theorem}[{\cite[Proposition 4.34]{CnossenHigherToposTheory}}]
Let $\mathcal{X}$ be a topos, and $f : \mathcal{Y} \rightarrow \mathcal{X}$ a geometric morphism. For all $X \in \mathcal{X}$ it holds that
\[ \mathrm{Map}_{\mathrm{RTop}_{/\mathcal{X}}}( \mathcal{Y}, \mathcal{X}_{/X} ) \cong \mathrm{Map}_{\mathcal{Y}}( 1, f^*(X) ). \] 
\end{theorem}

This formula is the key input for the following theorem.

\begin{theorem} \label{mapstoovertoposanima}
Let $\mathcal{X}$ be a topos, and let $A \in \mathrm{An}$. Then there exists a natural equivalence
\[
\mathrm{Map}_{\mathrm{RTop}}(\mathcal{X}, \mathrm{An}_{/A})
\simeq
\mathrm{Map}_{\mathrm{Pro}(\mathrm{An})}( \Pi_\infty(\mathcal{X}), A ).
\]
\end{theorem}

\begin{proof}
Note that $\mathrm{An}$ is the terminal topos. Hence applying the formula from above, we obtain
\[
\mathrm{Map}_{\mathrm{RTop}}(\mathcal{X}, \mathrm{An}_{/A})
\simeq
\mathrm{Map}_\mathcal{X}( 1 , \mathcal{X}^*(A) )
\simeq
\mathcal{X}_* \mathcal{X}^*(A)
\simeq
\mathrm{Map}_{\mathrm{Pro}(\mathrm{An})}( \Pi_\infty(\mathcal{X}), A ),
\]
where in the last equivalence we used the Yoneda lemma.
\end{proof}

\begin{theorem}[The shape-local systems adjunction, see {\cite[p.\ 4]{Hoyois2018HigherGaloisTheory}}] \label{shapeadjunction}
There exists an adjunction
\[
\begin{tikzcd}
	{\mathrm{RTop}} & {\mathrm{Pro}(\mathrm{An})}
	\arrow[""{name=0, anchor=center, inner sep=0}, "{\Pi_\infty}", curve={height=-12pt}, from=1-1, to=1-2]
	\arrow[""{name=1, anchor=center, inner sep=0}, "\beta", curve={height=-12pt}, from=1-2, to=1-1]
	\arrow["\dashv"{anchor=center, rotate=-90}, draw=none, from=0, to=1]
\end{tikzcd}
\]
where $\beta : \mathrm{Pro}(\mathrm{An}) \rightarrow \mathrm{RTop}$ is the pro-extension of the functor
\[
\begin{array}{rcl}
\mathrm{An} & \rightarrow & \mathrm{RTop} \\
A & \mapsto & \mathrm{An}_{/A}.
\end{array}
\]
\end{theorem}

The proof follows immediately from Theorem \ref{mapstoovertoposanima}. We can also restrict this adjunction to obtain the following adjunction.

\begin{theorem}[The shape-local systems adjunction for locally contractible topoi]
The adjunction from Theorem \ref{shapeadjunction} restricts to an adjunction
\[
\begin{tikzcd}
	{\mathrm{RTop}^{\mathrm{lc}}} & {\mathrm{An}}
	\arrow[""{name=0, anchor=center, inner sep=0}, "{\Pi_\infty}", curve={height=-12pt}, from=1-1, to=1-2]
	\arrow[""{name=1, anchor=center, inner sep=0}, curve={height=-12pt}, hook, from=1-2, to=1-1]
	\arrow["\dashv"{anchor=center, rotate=-90}, draw=none, from=0, to=1]
\end{tikzcd}
\]
where $\mathrm{RTop}^{\mathrm{lc}}$ is the full subcategory of $\mathrm{RTop}$ spanned by locally contractible topoi. The right adjoint
\[
\mathrm{An} \rightarrow \mathrm{RTop}^{\mathrm{lc}},
\qquad
A \mapsto \mathrm{An}_{/A}
\]
is fully faithful.
\end{theorem}

The claim about full faithfulness is straightforward since $\Pi_\infty( \mathrm{An}_{/A} ) \simeq A$, as we have shown earlier.

\begin{remark}
Unlike in the locally contractible case, the more general functor $\beta : \mathrm{Pro}(\mathrm{An}) \rightarrow \mathrm{RTop}$ is \emph{not} fully faithful; \cite[Remark 2.8]{Hoyois2018HigherGaloisTheory}. The central issue is that the shape of a topos is not just a pro-anima. Since for a topos $\mathcal{X}$ the shape is given as the endofunctor
\[
\mathcal{X}_* \mathcal{X}^* : \mathrm{An} \rightarrow \mathrm{An}
\]
which is induced from an adjunction between $\mathcal{X}$ and $\mathrm{An}$, we see that this endofunctor is canonically equipped with the structure of a monad. It is perhaps not straightforward to see that for shapes that are either constant or pro-compact, this monad structure is unique, which is why it will not enter the arguments in these cases. However, for general shape theory it would need to be taken into account.
\end{remark}

\subsection{Shape and connected components}

Recall from Example \ref{partitions} that there exist adjunctions
\[ \begin{tikzcd}
	{\mathrm{Top}} & {\mathrm{Loc}} & {\mathrm{Pro}(\mathrm{Set})}
	\arrow[""{name=0, anchor=center, inner sep=0}, "\Omega", curve={height=-12pt}, from=1-1, to=1-2]
	\arrow[""{name=1, anchor=center, inner sep=0}, "{\mathrm{pts}}", curve={height=-12pt}, from=1-2, to=1-1]
	\arrow[""{name=2, anchor=center, inner sep=0}, "{\pi_0}", curve={height=-12pt}, from=1-2, to=1-3]
	\arrow[""{name=3, anchor=center, inner sep=0}, curve={height=-12pt}, from=1-3, to=1-2]
	\arrow["\dashv"{anchor=center, rotate=-90}, draw=none, from=0, to=1]
	\arrow["\dashv"{anchor=center, rotate=-90}, draw=none, from=2, to=3]
\end{tikzcd} \]

Let us highlight another adjunction, namely
\[\begin{tikzcd}
	{\mathrm{An}} & {\mathrm{Set}}
	\arrow[""{name=0, anchor=center, inner sep=0}, "{\pi_0}", curve={height=-12pt}, from=1-1, to=1-2]
	\arrow[""{name=1, anchor=center, inner sep=0}, "{\mathrm{inc}}", curve={height=-12pt}, hook, from=1-2, to=1-1]
	\arrow["\dashv"{anchor=center, rotate=-90}, draw=none, from=0, to=1]
\end{tikzcd}\]
This induces, by virtue of $\mathrm{Pro}$ being a $2$-functor, the adjunction
\[ \begin{tikzcd}
	{\mathrm{Pro}(\mathrm{An})} & {\mathrm{Pro}(\mathrm{Set})}
	\arrow[""{name=0, anchor=center, inner sep=0}, "{\mathrm{Pro}(\pi_0)}", curve={height=-12pt}, from=1-1, to=1-2]
	\arrow[""{name=1, anchor=center, inner sep=0}, "{\mathrm{Pro}(\mathrm{inc})}", curve={height=-12pt}, hook, from=1-2, to=1-1]
	\arrow["\dashv"{anchor=center, rotate=-90}, draw=none, from=0, to=1]
\end{tikzcd} \]
With this out of the way, we can now explain the relation between shape and connected components.

\begin{theorem}
The square of left adjoints
\[\begin{tikzcd}
	{\mathrm{RTop}} & {\mathrm{Pro}(\mathrm{An})} \\
	{\mathrm{Loc}} & {\mathrm{Pro}(\mathrm{Set})}
	\arrow["{\Pi_\infty}", from=1-1, to=1-2]
	\arrow["{\mathrm{Loc}}"', from=1-1, to=2-1]
	\arrow["{\mathrm{Pro}(\pi_0)}", from=1-2, to=2-2]
	\arrow["{\pi_0}", from=2-1, to=2-2]
\end{tikzcd}\]
commutes via a natural isomorphism. Moreover, if $L$ is a locale, then
\[ \pi_0(L) \simeq \mathrm{Pro}(\pi_0)( \Pi_\infty(L) ). \]
\end{theorem}

\begin{proof}
To show that the square of left adjoints commutes, it suffices to show that the corresponding square of right adjoints
\[ \begin{tikzcd}
	{\mathrm{RTop}} & {\mathrm{Pro}(\mathrm{An})} \\
	{\mathrm{Loc}} & {\mathrm{Pro}(\mathrm{Set})}
	\arrow[from=1-2, to=1-1]
	\arrow["{\mathrm{Sh}}", from=2-1, to=1-1]
	\arrow["{\mathrm{Pro}(\mathrm{inc})}"', from=2-2, to=1-2]
	\arrow[from=2-2, to=2-1]
\end{tikzcd} \]
commutes, which, since the two horizontal functors are obtained as pro-extensions, reduces to showing that the square
\[ \begin{tikzcd}
	{\mathrm{RTop}} & {\mathrm{An}} \\
	{\mathrm{Loc}} & {\mathrm{Set}}
	\arrow["{\mathrm{\acute{e}t}}"', from=1-2, to=1-1]
	\arrow["{\mathrm{Sh}}", from=2-1, to=1-1]
	\arrow["{\mathrm{inc}}"', from=2-2, to=1-2]
	\arrow["{\mathrm{disc}}"', from=2-2, to=2-1]
\end{tikzcd} \]
commutes. But this follows from the Basis Theorem, as the discrete topology on a set $S$ admits a basis given by the set of singletons $\{s\}$ together with the empty set $\emptyset$, hence we see
\[ \mathrm{Sh}(S^\mathrm{disc}) \simeq \mathrm{Fun}(S, \mathrm{An}) \simeq \mathrm{An}_{/S}. \]
The statement about $\pi_0$ of a locale now follows from the fully faithfulness of the sheaves functor $\mathrm{Loc} \rightarrow \mathrm{RTop}$.
\end{proof}

\begin{remark}
We remark here that without much additional effort one can analogously show the existence of adjunctions
\[ \begin{tikzcd}
	{\mathrm{RTop}_n} & {\mathrm{Pro}(\mathrm{An}_{\leq n})}
	\arrow[""{name=0, anchor=center, inner sep=0}, "{\Pi_{\leq n}}", curve={height=-12pt}, from=1-1, to=1-2]
	\arrow[""{name=1, anchor=center, inner sep=0}, curve={height=-12pt}, from=1-2, to=1-1]
	\arrow["\dashv"{anchor=center, rotate=-90}, draw=none, from=0, to=1]
\end{tikzcd} \]
where $\mathrm{RTop}_n$ is the $(n+1,1)$-category of $n$-topoi, and $\mathrm{An}_{\leq n}$ is the $\infty$-category of $n$-truncated spaces, see \cite[p.\ 5]{Hoyois2018HigherGaloisTheory}.
\end{remark}

\section{Shape and inverse limits}

While so far, the functor that assigns to a topos its shape appears like a fairly straightforward generalization of the functor that assigns to a (reasonable) topological space its homotopy type, the shape functor is much better behaved with respect to certain classes of limits. We will go into this here. This will also allow us to give comparison theorems between the topos-theoretic approach to shape and more classical approaches to the shape of compact Hausdorff spaces.

\subsection{Compact topoi, perfect and proper geometric morphisms}

\begin{definition}
Let $\mathcal{X}$ be a topos. We call $\mathcal{X}$ \emph{compact} if the terminal object $1 \in \mathcal{X}$ is a compact object.
\end{definition}

Equivalently, a topos $\mathcal{X}$ is compact iff the canonical functor 
\[\mathcal{X}_* = \mathrm{Map}_\mathcal{X}(1,-) : \mathcal{X} \rightarrow \mathrm{An} \] preserves filtered colimits.

Compactness of a topos has an immediate consequence for its shape. Recall that we have the full subcategory
\[ \mathrm{An}^\omega \hookrightarrow \mathrm{An} \]
spanned by compact anima, which gives a fully faithful embedding
\[ \mathrm{Pro}(\mathrm{An}^\omega) \hookrightarrow \mathrm{Pro}(\mathrm{An}). \]
We call objects in the essential image \emph{pro-compact} anima.

\begin{remark}
There is an equivalence of $\infty$-categories $\mathrm{Pro}(\mathrm{An}^\omega) \simeq \mathrm{Pro}(\mathrm{An}^\mathrm{fin})$ since all compact anima arise as retracts of finite anima.
\end{remark}

\begin{lemma}
A pro-anima $X$ is pro-compact iff the corresponding functor $F \in \mathrm{Fun}^{\mathrm{lex},\mathrm{acc}}( \mathrm{An}, \mathrm{An} )$ preserves filtered colimits.
\end{lemma}

\begin{proof}
Note that $\mathrm{Pro}(\mathrm{An}^\omega)^{\mathrm{op}}$ is by definition the closure under filtered colimits of the subcategory of co-representables in $\mathrm{Fun}(\mathrm{An}^\omega, \mathrm{An})$. The identification of $\mathrm{Pro}(\mathrm{An}^\omega)^{\mathrm{op}}$ as a full subcategory of $\mathrm{Pro}(\mathrm{An})^{\mathrm{op}}$ corresponds to left Kan extension of a functor $f : \mathrm{An}^\omega \rightarrow \mathrm{An}$ along the canonical inclusion $i : \mathrm{An}^\omega \hookrightarrow \mathrm{An}$,
\[\begin{tikzcd}
	{\mathrm{An}^\omega} & {\mathrm{An}.} \\
	{\mathrm{An}}
	\arrow["f", from=1-1, to=1-2]
	\arrow["i"', hook, from=1-1, to=2-1]
	\arrow["{\mathrm{Lan}_i(f)}"', from=2-1, to=1-2]
\end{tikzcd}\]
But since $\mathrm{Ind}(\mathrm{An}^\omega) \simeq \mathrm{An}$, by the universal property of the $\mathrm{Ind}$-extension, a functor $F : \mathrm{An} \rightarrow \mathrm{An}$ arises as a left Kan extension along $i$ iff it preserves filtered colimits. 
\end{proof}

From this it is not hard to see the following.

\begin{proposition}
Let $\mathcal{X}$ be a topos. If $\mathcal{X}$ is compact then $\Pi_\infty(\mathcal{X})$ is a pro-compact anima.
\end{proposition}

\begin{proof}
By the previous lemma it suffices to show that if $\mathcal{X}$ is compact then
\[ \mathcal{X}_* \mathcal{X}^* : \mathrm{An} \rightarrow \mathrm{An} \]
preserves filtered colimits. But this is clear: $\mathcal{X}^*$ is a left adjoint and thus preserves all colimits, and $\mathcal{X}_*$ preserves filtered colimits by compactness.
\end{proof}

Let us generalize from topoi to geometric morphisms.

\begin{definition}
A geometric morphism $f : \mathcal{X} \rightarrow \mathcal{Y}$ is called \emph{perfect} if the pushforward part $f_*$ preserves filtered colimits.
\end{definition}

Of course, a topos $\mathcal{X}$ is compact iff the unique geometric morphism $\mathcal{X} \rightarrow \mathrm{An}$ is perfect. The name perfect suggests a relation with perfect maps in topology, which we come to in section \ref{stablycompact}. A special case is that of a proper map. Recall that a map of topological spaces $f : X \rightarrow Y$ is called \emph{proper} if $f$ is closed and all fibers $f^{-1}(\{p\})$ are compact for $p \in Y$. If $Y$ is locally compact Hausdorff, this agrees with the condition that $f^{-1}(K)$ is compact whenever $K \subset Y$ is compact. Note that all continuous maps between compact Hausdorff spaces are automatically proper. Let us introduce the topos-theoretic version of this notion. 
	
\begin{definition} \label{definitionproper}
Let $f : \mathcal{X} \rightarrow \mathcal{Y}$ be a geometric morphism. We call $f$ \emph{proper} if for every diagram of pullbacks
\[ \begin{tikzcd}
	{\mathcal{X}''} & {\mathcal{X}'} & {\mathcal{X}} \\
	{\mathcal{Y}''} & {\mathcal{Y}'} & {\mathcal{Y}}
	\arrow[from=1-1, to=1-2]
	\arrow[from=1-1, to=2-1]
	\arrow["\lrcorner"{anchor=center, pos=0.125}, draw=none, from=1-1, to=2-2]
	\arrow[from=1-2, to=1-3]
	\arrow[from=1-2, to=2-2]
	\arrow["\lrcorner"{anchor=center, pos=0.125}, draw=none, from=1-2, to=2-3]
	\arrow["f", from=1-3, to=2-3]
	\arrow[from=2-1, to=2-2]
	\arrow[from=2-2, to=2-3]
\end{tikzcd} \]	
the left square satisfies base-change.
\end{definition}

Note that in particular, if $f : \mathcal{X} \rightarrow \mathcal{Y}$ is proper, and we have a pullback of topoi
\[\begin{tikzcd}
	{\mathcal{X}'} & {\mathcal{X}} \\
	{\mathcal{Y}'} & {\mathcal{Y}}
	\arrow["{g'}", from=1-1, to=1-2]
	\arrow["{f'}"', from=1-1, to=2-1]
	\arrow["\lrcorner"{anchor=center, pos=0.125}, draw=none, from=1-1, to=2-2]
	\arrow["f", from=1-2, to=2-2]
	\arrow["g"', from=2-1, to=2-2]
\end{tikzcd} \]
then the square
\[ \begin{tikzcd}
	{\mathcal{X}'} & {\mathcal{X}} \\
	{\mathcal{Y}'} & {\mathcal{Y}}
	\arrow["{(f')_*}"', from=1-1, to=2-1]
	\arrow["{(g')^*}"', from=1-2, to=1-1]
	\arrow["{f_*}", from=1-2, to=2-2]
	\arrow["{g^*}", from=2-2, to=2-1]
\end{tikzcd} \]
commutes. This is referred to as \emph{proper base-change}.

\begin{remark} \label{propercontinuousispropergeometric}
If $f : X \rightarrow Y$ is a proper map of topological spaces, then the corresponding geometric morphism
\[ f : \mathrm{Sh}(X) \rightarrow \mathrm{Sh}(Y) \]
is proper, see \cite[Theorem 7.3.1.16]{Lurie2009HTT}. We remark that this includes closed embeddings.
\end{remark}

\begin{proposition}[{\cite[Remark 7.3.1.5]{Lurie2009HTT}}]
Proper maps of topoi are perfect.
\end{proposition}

\begin{warning}
The reverse implication is not true. A counterexample is given by the inclusion of the non-closed point $1 \in \mathbf{S}$, where $\mathbf{S} = \{0,1\}$ is the Sierpinski space, with only non-trivial open given by $\{1\}$. The induced geometric morphism is perfect, but not proper.
\end{warning}

\begin{theorem}[{\cite[Theorem 3.1.8]{martini2025propermorphismsinftytopoi}}] \label{propercompact}
A topos is compact iff it is proper.
\end{theorem}

\subsection{Shape of inverse limits}

The following is the key theorem that controls the shape of compact topoi. Recall that cofiltered limits of topoi are computed as limits of $\infty$-categories along the pushforward parts.

\begin{theorem} \label{shapeinverselimits}
Let $\mathcal{X}_\bullet : I \rightarrow \mathrm{RTop}$ be a diagram of topoi, with $I$ cofiltered, such that:
\begin{itemize}
\item for all $i \in I$ the topos $\mathcal{X}_i$ is compact, and 
\item for all maps $i \rightarrow j$ in $I$ the geometric morphism $\mathcal{X}_i \rightarrow \mathcal{X}_j$ is perfect.
\end{itemize}
Let $\mathcal{X} = \mathrm{lim}_{i \in I} \mathcal{X}_i$. Then
\[ \Pi_\infty( \mathcal{X} ) \cong \mathrm{lim}_{i \in I} \Pi_\infty( \mathcal{X}_i ), \]
with the limit taken in $\mathrm{Pro}(\mathrm{An})$.
\end{theorem}

\begin{remark}
Compactness cannot simply be dropped as an assumption. A counterexample is given by
\[\mathrm{lim}_{n \rightarrow \infty} [n, \infty) \cong  \emptyset,  \]
which induces, when taking sheaves, a diagram of non-compact topoi with perfect transition morphisms, but
\[ \mathrm{lim}_{n \rightarrow \infty} \Pi_\infty( [n, \infty) ) \cong  \mathrm{lim}_{n \rightarrow \infty} \ast \cong \ast \neq  \emptyset \cong \Pi_\infty( \emptyset )  \]
\end{remark}

Recall that cofiltered limits in $\mathrm{RTop}$ are computed as cofiltered limits in $\mathrm{Pr}^R$, equivalently as cofiltered limits in $\widehat{\mathrm{Cat}}_\infty$. The proof of Theorem \ref{shapeinverselimits} rests on the following technical push-pull lemma.

\begin{lemma}[Push-pull formula for inverse limits, {\cite{geiß2026cofilteredlimitsinftycategoriesadjunctions}}] \label{pushpull}
Let $\mathcal{C}_\bullet : I \rightarrow \mathrm{Pr}^R$ be a diagram with $I$ cofiltered. Assume that for all $\alpha : i \rightarrow j$ in $I$ the right adjoint transition functor
\[ \mathcal{C}(\alpha)^R : \mathcal{C}_i \rightarrow \mathcal{C}_j \]
preserves filtered colimits. Denote by $\mathcal{C}(\alpha)^L$ the left adjoint to $\mathcal{C}(\alpha)^R$, by $\mathcal{C} = \mathrm{lim}_{i \in I} \mathcal{C}_i$ the limit, and by
\[ p^R_i : \mathcal{C} \rightarrow \mathcal{C}_i \]
the canonical projection right adjoint, and by  
\[ p^L_i : \mathcal{C}_i \rightarrow \mathcal{C} \]
the corresponding left adjoint. Then for every pair $i,j \in I$, the composite
\[ p^R_j p^L_i : \mathcal{C}_i \rightarrow \mathcal{C}_j \]
is naturally isomorphic to the functor
\[ \mathrm{colim}_{ (k, \alpha, \beta) \in I_{/\{i,j\}} }  \mathcal{C}(\beta)^R  \mathcal{C}(\alpha)^L \]
\end{lemma}

This formula has been used informally for quite some time in the field. To our knowledge, the first clean proof is in the cited reference due to Thorger Geiß.

\begin{example}
We have 
\[ \mathrm{Sp} = \mathrm{lim}_{n \in \mathbb{N}}( \dots \rightarrow \mathrm{An}_{\ast /} \xrightarrow{ \Omega } \mathrm{An}_{\ast /} ). \]
The formula applied to this situation recovers the formula
\[ \Omega^\infty \Sigma^\infty X \cong \mathrm{colim}_{n \in \mathbb{N}} \Omega^n \Sigma^n X \]
for a pointed anima $X$.
\end{example}

From this we can directly prove Theorem \ref{shapeinverselimits}. Under stricter hypotheses, this statement appears in \cite[Proposition 2.11]{Hoyois2018HigherGaloisTheory}.

\begin{proof}[Proof of Theorem \ref{shapeinverselimits}]
Let $\mathcal{X}_\bullet : I \rightarrow \mathrm{RTop}$ be a cofiltered diagram of topoi, satisfying the two given conditions. First note that $\mathrm{RTop} \simeq \mathrm{RTop}_{/\mathrm{An}}$ since $\mathrm{An}$ is the terminal topos. Hence the diagram $I$ extends naturally to 
\[\mathcal{X}_\bullet : I^\triangleright \rightarrow \mathrm{RTop},\]
where $I^\triangleright = I \star \{ \top \}$ is the diagram that adds an additional cone point $\top$, and $\mathcal{X}_\bullet( \top ) = \mathrm{An}$. Note that $I \star \{ \top \}$ is again cofiltered. Since for any arrow $i \rightarrow \top$ in $I^\triangleright$ the induced right adjoint is given by
\[ (\mathcal{X}_i)_* : \mathcal{X}_i \rightarrow \mathrm{An}, \]
we see that the conditions that all the $\mathcal{X}_i$ are compact and transition functors preserve filtered colimits simply translate into the statement that all transition functors of the extended diagram $\mathcal{X}_\bullet : I^\triangleright \rightarrow \mathrm{RTop}$ preserve filtered colimits. Thus we can apply Lemma \ref{pushpull}. Let $\mathcal{X} = \mathrm{lim}_{i \in I} \mathcal{X}_i$. Then
\[ \mathcal{X}_* \mathcal{X}^* \cong \mathrm{colim}_{i \in I} (\mathcal{X}_i)_* (\mathcal{X}_i)^* \]
But under the equivalence $\mathrm{Pro}(\mathrm{An}) \simeq \mathrm{Fun}^\mathrm{lex,acc}( \mathrm{An}, \mathrm{An})^\mathrm{op}$, cofiltered limits in $\mathrm{Pro}(\mathrm{An})$ are computed as filtered colimits in  $\mathrm{Fun}^\mathrm{lex,acc}( \mathrm{An}, \mathrm{An})$, which are computed objectwise, as filtered colimits commute with finite limits in $\mathrm{An}$. Hence we have verified that
\[ \Pi_\infty( \mathcal{X} ) \cong \mathrm{lim}_{i \in I} \Pi_\infty( \mathcal{X}_i ). \]
\end{proof}

A \emph{finite-dimensional polytope} is a subspace $P \subset \mathbb{R}^n$ for some $n \geq 0$ such that $P$ is a finite union of \emph{nondegenerate simplices} $\Delta \subset \mathbb{R}^n$, where with a nondegenerate simplex we mean a convex hull of $n+1$-many points with non-empty interior.

\begin{corollary} \label{shapeinversepolytopes}
Let $X$ be a compact Hausdorff space, and assume that
\[ X \cong \lim_{i \in I} P_i \]
is a cofiltered inverse limit where $P_i$ are finite-dimensional polytopes. Then
\[ \Pi_\infty(X) \simeq \text{``}\lim_{i \in I} \text{''} |P_i| \in \mathrm{Pro}(\mathrm{An}). \]
\end{corollary}

\begin{remark}
In fact, any compact Hausdorff space does arise as an inverse limit of polytopes, see \cite[Chapter I, \S5.2, Theorem 7]{mardesic1982shape}. In some earlier approaches to shape theory, Corollary \ref{shapeinversepolytopes} was used as a definition of the shape of a compact Hausdorff space.
\end{remark}

\subsection{Stably compact spaces} \label{stablycompact}

The conditions for Theorem \ref{shapeinverselimits} fit particularly well for a class of spaces, which is a useful enlargement of that of compact Hausdorff spaces, namely the so-called \emph{stably compact spaces}. As an overview, see \cite[Section 9]{GoubaultLarrecq2013NonHausdorff}, as well as \cite[Section 3.2]{Lehner2026KTheoryStablyCompact}. Before we do so, let us recall some basics about the Alexandroff topology.

\begin{definition}
Let $(P, \leq)$ be a poset. The \emph{Alexandroff topology} on $P$ is given by declaring the open sets to be upward closed sets with respect to $\leq$, i.e. $U \subset P$ is open iff for all $p \leq q$ in $P$ it holds that $p \in U \implies q \in U$. We denote the resulting space by $P_\mathrm{Alex}$.

If $X$ is a topological space, then a \emph{$P$-indexed stratification} is defined to be a continuous map $X \rightarrow P_\mathrm{Alex}$.
\end{definition}

Note that any poset $P$ has a canonical open covering, given by the principal opens $P_{p \geq}$ for $p \in P$. Hence if $f :  X \rightarrow P_\mathrm{Alex}$ is a continuous map, we get an induced cover $U_p := f^{-1}(P_{p \geq})$ of $X$.

\begin{proposition}[{\cite[Proposition 5.2]{Raptis2010HomotopyPosets}}]
The datum of a continuous map $X \rightarrow P_\mathrm{Alex}$ is equivalently described by a $P$-indexed collection of open sets $(U_p)_{p \in P}$ of $X$ such that:
\begin{itemize}
\item $p \leq q$ implies $U_q \subset U_p$.
\item $\bigcup_{p \in P} U_p = X$.
\item For every $x \in X$ it holds that the poset $\{ p \in P ~|~ x \in U_p \}$ has a maximal element. 
\end{itemize}
\end{proposition}

Such $P$-indexed collections of opens are also referred to as well-indexed coverings. 

\begin{definition}
Let $X$ be a topological space. Then 
\[ \gamma(X) := \mathrm{lim}_{ X \rightarrow P_\mathrm{Alex}, P \text{ finite poset}} P_\mathrm{Alex} \in \mathrm{Top}\]
\end{definition}

The space $\gamma(X)$ is always a coherent space and should be understood as the universal way to resolve $X$ via finite, well-indexed coverings. The assignment $X \mapsto \gamma(X)$ is in fact a functor, given by the right Kan extension
\[\begin{tikzcd}
	{\mathrm{FinPoset}} & {\mathrm{Top}.} \\
	{\mathrm{Top}}
	\arrow["{\mathrm{Alex}}", hook, from=1-1, to=1-2]
	\arrow["{\mathrm{Alex}}"', hook', from=1-1, to=2-1]
	\arrow["\gamma"', from=2-1, to=1-2]
\end{tikzcd}\]
It is thus the \emph{codensity} monad for the inclusion of finite posets equipped with the Alexandroff topology into $\mathrm{Top}$. 

A topological space $X$ is called \emph{stably compact} if the canonical map $i : X \rightarrow \gamma(X)$ admits a retraction $p : \gamma(X) \rightarrow X$. We note here that this retraction is unique if it exists, and gives an internal adjunction on the level of locales
\[ \begin{tikzcd}
	{\mathcal{O}(X)} & {\mathcal{O}(\gamma(X))}
	\arrow[""{name=0, anchor=center, inner sep=0}, "{p^*}", curve={height=-12pt}, hook, from=1-1, to=1-2]
	\arrow[""{name=1, anchor=center, inner sep=0}, "{i^*}", curve={height=-12pt}, from=1-2, to=1-1]
	\arrow["\dashv"{anchor=center, rotate=-90}, draw=none, from=0, to=1]
\end{tikzcd} \]
with $p^*$ being fully faithful. A map $f : X \rightarrow Y$ between stably compact spaces is called perfect if the square
\[ \begin{tikzcd}
	{\gamma(X)} & X \\
	{\gamma(Y)} & Y
	\arrow["{p_X}", from=1-1, to=1-2]
	\arrow["{\gamma(f)}"', from=1-1, to=2-1]
	\arrow["f", from=1-2, to=2-2]
	\arrow["{p_Y}"', from=2-1, to=2-2]
\end{tikzcd} \]
commutes. We note that compact Hausdorff spaces are in particular stably compact, and all maps between compact Hausdorff spaces are automatically perfect.
	
\begin{theorem}[{\cite[Theorem 4.29]{aoki2025schwartzcoidempotentscontinuousspectrum}}, also {\cite[Theorem 3.50]{Lehner2026KTheoryStablyCompact}}]
Let $f : X \rightarrow Y$ be a perfect map between stably compact spaces. Then $f : \mathrm{Sh}(X) \rightarrow \mathrm{Sh}(Y)$ is a perfect geometric morphism.
\end{theorem}

\begin{corollary} \label{stablycompactspacecompacttopos}
Let $X$ be a stably compact space. Then $\mathrm{Sh}(X)$ is a compact topos.
\end{corollary}

Since the functor of taking sheaves, as a functor from locales to topoi, is a $2$-functor, it preserves internal adjunctions. Hence we obtain the adjunction 
\[\begin{tikzcd}
	{\mathrm{Sh}(X)} & {\mathrm{Sh}(\gamma(X))}
	\arrow[""{name=0, anchor=center, inner sep=0}, "{p^*}", curve={height=-12pt}, hook, from=1-1, to=1-2]
	\arrow[""{name=1, anchor=center, inner sep=0}, curve={height=12pt}, hook, from=1-1, to=1-2]
	\arrow[""{name=2, anchor=center, inner sep=0}, "{i^*}"{description}, from=1-2, to=1-1]
	\arrow["\dashv"{anchor=center, rotate=-90}, draw=none, from=0, to=2]
	\arrow["\dashv"{anchor=center, rotate=-90}, draw=none, from=2, to=1]
\end{tikzcd}\]
Since then $p^*$ is fully faithful, we obtain the following.
	
\begin{proposition} \label{coveringmonadshapeequivalence}
Let $X$ be stably compact. Then the structure map $\gamma(X) \rightarrow X$ is a shape equivalence.
\end{proposition}

The remaining theorem is an identification of the shape of finite posets equipped with the Alexandroff topology.

\begin{theorem}[{\cite[Example A.12]{Aoki_2023}}] \label{sheavesonfiniteposets}
Let $P$ be a finite poset. Then there exists a natural equivalence $\mathrm{Sh}(P_\mathrm{Alex}) \simeq \mathrm{Fun}(P,\mathrm{An})$. In particular, it holds that $\Pi_\infty( P_\mathrm{Alex} ) \cong |P| \in \mathrm{Pro}(\mathrm{An}).$
\end{theorem}

\begin{warning}
Theorem \ref{sheavesonfiniteposets} contains a subtle point, in that it is false for arbitrary posets, see \cite[Example A.13]{Aoki_2023}.
\end{warning}
	
\begin{theorem} \label{shapeofstablycompact}
Let $X$ be a stably compact space. Then
\[ \Pi_\infty(X) \simeq \text{``}\lim_{(P,\mathcal{U}) \emph{\text{ finite }P\text{-indexed stratification of }}X} \text{''} ~|P| \]
where the limit in question goes over the category of finite stratifications of $X$.
\end{theorem}

\begin{proof}
This now follows directly from combining Proposition \ref{coveringmonadshapeequivalence} with Theorem \ref{shapeinverselimits}.
\end{proof}

More generally, a space $X$ will be called \emph{stably locally compact} if it arises as an open subspace of a stably compact space. This includes for instance locally compact Hausdorff spaces. Again, for more information see \cite[Section 3.2]{Lehner2026KTheoryStablyCompact}. A compact subspace $K \subset X$ is called \emph{saturated} if it holds that
\[ K = \bigcap_{K \subset U \text{ open}} U. \]
Saturated compact subspaces of stably locally compact spaces are automatically stably compact.

\begin{proposition} \label{openapproximation}
Let $X$ be a stably locally compact space, and $K \subset X$ a saturated compact subspace. Then
\[ \Pi_\infty(K) \cong \mathrm{lim}_{K \subset U \emph{\text{ open}}}  \Pi_\infty(U) \]
with the limit taken in $\mathrm{Pro}(\mathrm{An})$.
\end{proposition}

\begin{proof}
Since $X$ is stably locally compact, we can find for every $K \subset U$ with $U$ open, an open $V$ and saturated compact $K'$ such that
\[ K \subset V \subset K' \subset U, \]
see \cite[Proposition 4.8.14]{GoubaultLarrecq2013NonHausdorff}. Hence the systems of open, respectively saturated compact, neighbourhoods of $K$ are mutually cofinal. Therefore 
\[ \mathrm{lim}_{K \subset U \text{ open}}  \Pi_\infty(U) \cong \mathrm{lim}_{K \ll K'}  \Pi_\infty(K') \cong \Pi_\infty(K), \]
where the second limit goes over all saturated compact neighborhoods of $K$, and the last isomorphism follows from Theorem~\ref{shapeinverselimits}. 
\end{proof}

\begin{theorem}
Let $X \subset [0,1]^{S}$ be a closed subset of a Tychonoff cube of cardinality $S$. Then
\[ \Pi_\infty(X) \cong \text{``}\lim_{X \subset U \emph{\text{ open}}} \text{''} \Pi_\infty(U) \in \mathrm{Pro}(\mathrm{An}), \]
where the open sets $U$ have constant shapes.

If $S$ is either finite or $S \cong \mathbb{N}$ is countable, and $X \subset [0,1]^{S}$ is a closed subset, then it also holds that
\[ \Pi_\infty(X) \cong \text{``}\lim_{\varepsilon \rightarrow 0} \text{''} \Pi_\infty(U_\varepsilon(X)) \in \mathrm{Pro}(\mathrm{An}), \]
where the open $\varepsilon$-neighbourhoods $U_\varepsilon(X)$ are taken with respect to a compatible metric and have constant shapes.
\end{theorem}

\begin{proof}
Note first that the Tychonoff cube is locally contractible, since it has a basis of contractible opens, hence every open $U \subset [0,1]^{S}$ has constant shape. The claim thus immediately follows from Proposition \ref{openapproximation}.

In the specific case where $S$ is either finite or countable, the Tychonoff cube $[0,1]^{S}$ is metrizable. Fix a compatible metric. Since $X$ is a compact subset, it follows that for every open $U$ such that $X \subset U$ there exists an $\varepsilon > 0$ such that $X \subset U_\varepsilon(X) \subset U$, hence the system of $\varepsilon$-neighbourhoods sits cofinally in the system of all open neighbourhoods.
\end{proof}

\subsection{Homotopy invariance of shape}

A perhaps surprising fact is that the compatibility of shape with inverse limits of compact Hausdorff spaces by itself already implies homotopy invariance of the shape. Let us first give a general, well-known lemma regarding homotopy invariance.

\begin{lemma}
Let $\mathcal{C}$ be a full subcategory of either the category of topological spaces or the category of locales that is closed under finite products and contains the interval $[0,1]$. Let $F\colon\mathcal{C}^{\op}\rightarrow\mathrm{Set}$ be a functor. Then the following are equivalent.
\begin{enumerate}
\item For every $X \in \mathcal{C}$ the map $p : [0,1] \rightarrow \mathrm{pt}$ induces a bijection
\[F(\mathrm{id}_X \times p) : F(X) \xrightarrow{\cong}  F(X \times [0,1]). \] 
\item For every $X \in \mathcal{C}$ the maps $i_0 : \mathrm{pt} \hookrightarrow [0,1], i_1 : \mathrm{pt} \hookrightarrow [0,1]$ give the same function
\[ F( \mathrm{id}_X \times i_0 ) = F( \mathrm{id}_X \times i_1 ) : F(X \times [0,1]) \rightarrow  F(X). \]
\end{enumerate}
\end{lemma}

\begin{proof}
First note that $\mathrm{id}_X \times i_0$ and $\mathrm{id}_X \times i_1$ are both right inverses to the map $\mathrm{id}_X \times p : X \times [0,1] \rightarrow X$. When applying the contravariant functor $F$, therefore the functions $F( \mathrm{id}_X \times i_0 )$ and $F( \mathrm{id}_X \times i_1 )$ are left inverses to $F(\mathrm{id}_X \times p)$.

First assume (1) holds. If the map $F(\mathrm{id}_X \times p)$ is invertible, its left inverses are unique, hence (2) follows.

Now write 
\[ \begin{array}{rcl}
 m : [0,1]^2 & \rightarrow & [0,1] \\
 	(s,t) & \mapsto & st.
\end{array} \]
Then it holds that
\[m ( \mathrm{id}_{[0,1]} \times i_0 ) =  i_0 p  \qquad \text{ and } \qquad m ( \mathrm{id}_{[0,1]} \times i_1 ) =  \mathrm{id}_{[0,1]}. \]
Assume (2) holds, which implies that
\[ F( \mathrm{id}_X \times \mathrm{id}_{[0,1]} \times  i_0 ) = F( \mathrm{id}_X \times \mathrm{id}_{[0,1]} \times  i_1 ) : F(X \times [0,1] \times [0,1]) \rightarrow  F(X \times [0,1]). \]
Therefore
\[ \begin{array}{rcl}
F( \mathrm{id}_X \times p ) F( \mathrm{id}_X \times i_0 ) & = & F( \mathrm{id}_X \times i_0 p ) \\
& = & F( \mathrm{id}_X \times m ( \mathrm{id}_{[0,1]} \times i_0 ) ) \\
& = & F( \mathrm{id}_X \times \mathrm{id}_{[0,1]} \times i_0  ) F( \mathrm{id}_X \times m ) \\
& = & F( \mathrm{id}_X \times \mathrm{id}_{[0,1]} \times i_1  ) F( \mathrm{id}_X \times m ) \\
& = & F( \mathrm{id}_X \times m ( \mathrm{id}_{[0,1]} \times i_1 ) ) \\
& = & F( \mathrm{id}_X \times \mathrm{id}_{[0,1]} )
\end{array} \]
hence $F( \mathrm{id}_X \times i_0 )$ is also a right inverse to $F( \mathrm{id}_X \times p )$, and thus $F( \mathrm{id}_X \times p )$ is bijective.
\end{proof}

The following lemma is due to Hoyois; see \cite[Proposition 4.7]{arnone2026bredonsheafcohomology} for an alternative reference.\footnote{We thank Thorger Geiß for informing us about this lemma, and its potential application to homotopy invariance of shape.}

\begin{lemma}[Hoyois]\label{HoyoisLemma}
Let $\mathcal{C}$ be a full subcategory of either the category of topological spaces or the category of locales that is closed under finite products and contains the interval $[0,1]$. Let $F\colon\mathcal{C}^{\op}\rightarrow\mathcal{E}$ be a functor, where $\mathcal{E}$ is a compactly assembled presentable $\infty$-category. If $F$ sends cofiltered intersections of closed subspaces to filtered colimits, then $F$ is homotopy-invariant.
\end{lemma}

\begin{proof}
First, write $\mathcal{E}$ as a retract $\mathcal{E}\stackrel{i}{\rightarrow}\mathcal{F}\stackrel{r}{\rightarrow}\mathcal{E}$ of a compactly generated category $\mathcal{F}$ in $\Pr^L$. Then $iF\colon\mathcal{C}^{\op}\rightarrow\mathcal{F}$ still satisfies the hypothesis and if $iF$ is homotopy-invariant, then so is $F\simeq riF$, so assume w.l.o.g.\ that $\mathcal{E}$ is compactly generated. Then, picking a set $X_i,\,i\in I$ of compact generators of $\mathcal{E}$, the functors $\Map_{\mathcal{E}}(X_i,-)\colon\mathcal{E}\rightarrow\An$ preserve filtered colimits and are jointly conservative, so assume w.l.o.g.\ that $\mathcal{E}=\An$. Applying the filtered colimit-preserving functor $\pi_0\colon\An\rightarrow\Set$ and the filtered colimit-preserving functors $\pi_n\colon\An_{\ast}\rightarrow\Set$ (being slightly careful about choosing basepoints), the J.H.C. Whitehead theorem allows us to assume w.l.o.g.\ $\mathcal{E}=\Set$.

We claim that $\mathrm{const} : F( X ) \rightarrow F(X \times [0,1])$ is a bijection for all $X \in \mathcal{C}$. Note that the functor $F^{\prime}\coloneqq F(X\times(-))$ satisfies the very same hypothesis as $F$, hence using the previous lemma, w.l.o.g.\ it suffices to show that the inclusions $i_0, i_1 : \mathrm{pt} \hookrightarrow [0,1]$ induce the same map after applying $F$. Let us fix notation. 

Consider the map $p : [0,1] \rightarrow \mathrm{pt}$, which induces the function
\[ \begin{array}{rcl}
\mathrm{const} : F( \mathrm{pt} ) & \rightarrow & F([0,1]) \\
c & \mapsto & \mathrm{const}_c
\end{array} \]
and the inclusions $\{s\} \hookrightarrow [0,1]$, which induce the functions
\[ \begin{array}{rcl}
F([0,1]) & \rightarrow & F( \mathrm{pt} )  \\
\alpha & \mapsto & \alpha(s).
\end{array} \]

Let $\alpha \in F([0,1])$. We want to argue that $\alpha(0) = \alpha(1) \in F(\mathrm{pt})$. Fix $s \in [0,1]$. Since 
\[ \mathrm{colim}_{ \varepsilon > 0} F([s-\varepsilon, s + \varepsilon] \cap [0,1]) \cong F(\{s\}), \]
there exists some $\varepsilon_s > 0$ such that 
\[ \alpha|_{[s-\varepsilon_s, s + \varepsilon_s] \cap [0,1]} = \mathrm{const}_{c_s} \in F([s-\varepsilon_s, s + \varepsilon_s] \cap [0,1]) \]
where $c_s = \alpha(s)$. The sets $((s-\varepsilon_s, s + \varepsilon_s) \cap [0,1])_{s \in [0,1]}$ form an open cover of $[0,1]$, hence a finite set $\{s_1, \cdots, s_n\} \subset [0,1]$ suffices for the corresponding open intervals to cover $[0,1]$. Since the interval is connected, it thus follows that
\[ c_{s_i} = c_{s_j} \]
for all $1 \leq i, j \leq n$, and thus $\alpha(0) = \alpha(1)$.
\end{proof}

\begin{corollary}
The canonical map $[0,1] \rightarrow \mathrm{pt}$ is a shape equivalence.
\end{corollary}

\begin{proof}
Let $A \in \mathrm{An}$. We can apply Lemma \ref{HoyoisLemma} to the functor
\[ \begin{array}{rcl}
\mathrm{CHaus}^{\mathrm{op}} & \rightarrow & \mathrm{An} \\
X & \mapsto & \mathrm{Map}_{\mathrm{Pro}(\mathrm{An})}( \Pi_\infty(X), A ) 
\end{array} \]
which by Theorem \ref{shapeinverselimits} satisfies the conditions and is thus homotopy invariant. Hence
\[ \mathrm{Map}_{\mathrm{Pro}(\mathrm{An})}( \Pi_\infty([0,1]), A ) \cong A \]
for all animae $A$, and therefore $\Pi_\infty([0,1]) \cong \ast$.
\end{proof}

\begin{corollary}
The topological spaces $[0,1)$ as well as $(0,1)$ have contractible shape.
\end{corollary}

\begin{proof}
We will prove the case for $[0,1)$; the statement for $(0,1)$ follows analogously. Since we have the open cover $[0,a)_{a < 1}$ of $[0,1)$ it holds that
\[ \Pi_\infty([0,1)) \cong \mathrm{colim}_{a < 1} \Pi_\infty([0,a)) \]
Note that for any $a < 1$ there exists $b < 1$ such that $a < b < 1$, hence we have the mutual inclusions
\[ [0,a) \subset [0,a] \subset [0,b) \subset [0,b]. \]
Hence by mutual cofinality it holds that
\[ \mathrm{colim}_{a < 1} \Pi_\infty([0,a)) \cong \mathrm{colim}_{a < 1} \Pi_\infty([0,a]) \cong \mathrm{colim}_{a < 1} \ast \cong \ast. \]
\end{proof}

Since the interval $[0,1]$ admits a basis consisting of open sets homeomorphic to spaces of the above form, Corollary \ref{locallycontractibleislocallycontractible} implies the following.

\begin{corollary}
The topos $\mathrm{Sh}([0,1])$ is contractible as well as locally contractible.
\end{corollary}

We can use this to argue homotopy invariance of shape for arbitrary topoi. To do so, let us remark on another $2$-functor, namely the Lurie tensor product. From the formula
\[ \mathcal{C} \otimes \mathcal{D} \cong \mathrm{Fun}^{R}( \mathcal{C}^\mathrm{op}, \mathcal{D} ) \]
it is apparent that for fixed $\mathcal{C}$ the functor
\[ \mathcal{C} \otimes - : \mathrm{Pr}^L \rightarrow \mathrm{Pr}^L \]
is a $2$-functor, i.e.\ also functorial in natural transformations. Hence just as before we see that $\mathcal{C} \otimes - $ preserves internal embeddings. This has the following immediate consequence.

\begin{proposition} \label{essentialandcontractible}
Let $f : \mathcal{X} \rightarrow \mathcal{Y}$ be an essential and contractible geometric morphism, and $\mathcal{Z}$ a topos. Then
\[ f \times \mathrm{id}_\mathcal{Z} : \mathcal{X} \otimes \mathcal{Z} \rightarrow \mathcal{Y} \otimes \mathcal{Z} \]
is again essential and contractible.
\end{proposition}

\begin{proof}
By assumption we have the adjunctions
\[\begin{tikzcd}
	{\mathcal{X}} & {\mathcal{Y}}
	\arrow[""{name=0, anchor=center, inner sep=0}, "{f_*}"', curve={height=12pt}, from=1-1, to=1-2]
	\arrow[""{name=1, anchor=center, inner sep=0}, "{f_!}", curve={height=-12pt}, from=1-1, to=1-2]
	\arrow[""{name=2, anchor=center, inner sep=0}, "{f^*}"{description}, hook', from=1-2, to=1-1]
	\arrow["\dashv"{anchor=center, rotate=-90}, draw=none, from=1, to=2]
	\arrow["\dashv"{anchor=center, rotate=-90}, draw=none, from=2, to=0]
\end{tikzcd}\]
with $f^*$ fully faithful. By applying the $2$-functor $- \otimes \mathcal{Z}$ we therefore get the adjunctions
\[ \begin{tikzcd}
	{\mathcal{X}\otimes \mathcal{Z}} & {\mathcal{Y}\otimes \mathcal{Z}}
	\arrow[""{name=0, anchor=center, inner sep=0}, "{(f^* \otimes \mathrm{id}_\mathcal{Z})_*}"', curve={height=12pt}, from=1-1, to=1-2]
	\arrow[""{name=1, anchor=center, inner sep=0}, "{f_! \otimes \mathrm{id}_\mathcal{Z}}", curve={height=-24pt}, from=1-1, to=1-2]
	\arrow[""{name=2, anchor=center, inner sep=0}, "{f^* \otimes \mathrm{id}_\mathcal{Z}}"', hook', from=1-2, to=1-1]
	\arrow["\dashv"{anchor=center, rotate=-89}, draw=none, from=1, to=2]
	\arrow["\dashv"{anchor=center, rotate=-91}, draw=none, from=2, to=0]
\end{tikzcd} \]
with $f^* \otimes \mathrm{id}_\mathcal{Z}$ fully faithful.
\end{proof}

\begin{theorem} \label{homotopyinvariance}
The functor $\Pi_\infty : \mathrm{RTop} \rightarrow \mathrm{Pro}(\mathrm{An})$ is homotopy invariant. Concretely, for any topos $\mathcal{X}$, the projection $\mathcal{X} \otimes \mathrm{Sh}([0,1]) \rightarrow \mathcal{X}$ induces an isomorphism in shape
\[ \Pi_\infty(\mathcal{X} \otimes \mathrm{Sh}([0,1])) \cong \Pi_\infty(\mathcal{X}). \]
In particular, homotopy equivalent topoi have isomorphic shape.
\end{theorem}

We can also specialize to the situation of locales and topological spaces.

\begin{corollary}
The functors
\[ \Pi_\infty : \mathrm{Loc} \rightarrow \mathrm{Pro}(\mathrm{An}) \]
as well as 
\[ \Pi_\infty : \mathrm{Top} \rightarrow \mathrm{Pro}(\mathrm{An}) \]
are homotopy invariant.
\end{corollary}

\begin{proof}
The claim about the restriction to locales is immediate, since the inclusion $\mathrm{Loc} \hookrightarrow \mathrm{RTop}$ is a right adjoint and thus preserves products. The statement about topological spaces is somewhat more subtle and holds because the interval $[0,1]$ is core-compact, hence
\[ \Omega( X \times [0,1] ) \cong \Omega( X ) \times \Omega( [0,1] ) \]
holds on the level of locales.
\end{proof}

\begin{remark}
The above techniques illustrate that the interval $[0,1]$ could be replaced by any topos $\mathcal{X}$ that is contractible and locally contractible to get a corresponding version of $\mathcal{X}$-homotopy invariance. One other example is the directed interval $[0,1]_\leq$.
\end{remark}

Since we now finally have verified that contractible topological spaces actually have trivial shape, it follows that the shape of a CW complex is indeed what we had claimed earlier.

\begin{corollary}
Let $X$ be a topological space. An open cover $\mathcal{U} = \{ U_i \}_{i \in I}$ of $X$ such that all non-empty finite intersections
\[ U_{i_0} \cap U_{i_1} \cap \dots \cap U_{i_n} \neq \emptyset \]
are contractible is shape-good. Hence if $X$ admits such a cover, then the shape of $X$ is constant and isomorphic to the geometric realization of the \v{C}ech-nerve of the cover.
\end{corollary}

\section{Künneth theorems for shape}

Classically, the Künneth theorem gives a computation of the homology of a product $X \times Y$ in terms of the homologies of $X$ and $Y$. A similar result holds for shape, but some restrictions apply. For one, the functor
\[ \Pi_\infty : \mathrm{RTop} \rightarrow \mathrm{Pro}(\mathrm{An}) \] 
does not preserve all binary products, see the following counterexample.

\begin{example} \label{counterexample}
Consider $\mathcal{X} = \mathrm{Sh}(\mathbb{N}^\mathrm{disc})$ and $\mathcal{Y} = \mathrm{Sh}(C)$, where $C = \prod_{\mathbb{N}} \mathbf{2}^\mathrm{disc}$ is Cantor space. The space $\mathbb{N}^\mathrm{disc}$ is locally contractible. Since $\mathbb{N} \times C = \coprod_{\mathbb{N}} C$, we see that pullback and pushforward for the canonical geometric morphism are given by
\[
\begin{tikzcd}
	{\mathrm{Sh}(\mathbb{N} \times C)} & {\prod_{\mathbb{N}}\mathrm{Sh}( C)} & {\mathrm{Sh}(C)} & {\mathrm{An}}
	\arrow["\simeq"{description}, draw=none, from=1-1, to=1-2]
	\arrow[""{name=0, anchor=center, inner sep=0}, "{\prod_\mathbb{N}}"', curve={height=12pt}, from=1-2, to=1-3]
	\arrow[""{name=1, anchor=center, inner sep=0}, "{\mathrm{const}}"', curve={height=12pt}, from=1-3, to=1-2]
	\arrow[""{name=2, anchor=center, inner sep=0}, "{C_*}"', curve={height=12pt}, from=1-3, to=1-4]
	\arrow[""{name=3, anchor=center, inner sep=0}, "{C^*}"', curve={height=12pt}, from=1-4, to=1-3]
	\arrow["\dashv"{anchor=center, rotate=-90}, draw=none, from=1, to=0]
	\arrow["\dashv"{anchor=center, rotate=-90}, draw=none, from=3, to=2]
\end{tikzcd} 
\]
and hence the functor for $\Pi_\infty( \mathbb{N} \times C )$ is given by
\[ A \mapsto \prod_\mathbb{N} \mathrm{colim}_{k \in \mathbb{N}} \mathrm{Map}( \mathbf{2}^k, A ) \cong \mathrm{Map}( \mathbb{N}, \mathrm{colim}_{k \in \mathbb{N}} \mathrm{Map}( \mathbf{2}^k, A ) ). \]
On the other hand, the product of the shapes $\Pi_\infty( \mathbb{N})$ and $\Pi_\infty( C )$ is given by the pro-object
\[ \text{``} \mathrm{lim}_{k \in \mathbb{N}}\text{''} \mathbb{N} \times \mathbf{2}^k \]
which corresponds to the functor
\[ A \mapsto \mathrm{colim}_{k \in \mathbb{N}} \mathrm{Map}( \mathbb{N} \times \mathbf{2}^k, A ). \]
These two functors do not agree as $\mathbb{N}$ is not a compact anima.
\end{example}

Nonetheless, for certain classes of topoi product preservation is in fact true. In this section, we will give the Künneth theorem for two such classes, with one of them being the class of locally contractible topoi. Let us set up some definitions.

Suppose that we have a commutative square
\[ \begin{tikzcd}
	{\mathcal{X}'} & {\mathcal{Y}'} \\
	{\mathcal{X}} & {\mathcal{Y}}
	\arrow["{f'}", from=1-1, to=1-2]
	\arrow["{g'}"', from=1-1, to=2-1]
	\arrow["g", from=1-2, to=2-2]
	\arrow["f", from=2-1, to=2-2]
\end{tikzcd} \]
of geometric morphisms. We therefore have a natural isomorphism
\[ (g')^* f^* \cong (f')^* g^*. \]
Taking the adjoint, we get a natural transformation
\[ f^* \implies (g')_* (f')^* g^*. \]
We can compose this with $g_*$ and use the counit to get a natural transformation
\[ f^* g_* \implies  (g')_* (f')^* g^* g_* \implies (g')_* (f')^* \]
called the \emph{mate} of the commuting square of left adjoints.

\begin{definition}
Let
\[ \begin{tikzcd}
	{\mathcal{X}'} & {\mathcal{Y}'} \\
	{\mathcal{X}} & {\mathcal{Y}}
	\arrow["{f'}", from=1-1, to=1-2]
	\arrow["{g'}"', from=1-1, to=2-1]
	\arrow["g", from=1-2, to=2-2]
	\arrow["f", from=2-1, to=2-2]
\end{tikzcd} \]
be a commutative square of geometric morphisms. We say that the square satisfies \emph{base-change}, or equivalently that it is \emph{left-adjointable}, equivalently that it satisfies the \emph{Beck-Chevalley condition}, if the mate is an equivalence.
\end{definition}
In diagrams, this means the square
\[ \begin{tikzcd}
	{\mathcal{X}'} & {\mathcal{Y}'} \\
	{\mathcal{X}} & {\mathcal{Y}}
	\arrow["{(g')_*}"', from=1-1, to=2-1]
	\arrow["{(f')^*}"', from=1-2, to=1-1]
	\arrow["{g_*}", from=1-2, to=2-2]
	\arrow["{f^*}"', from=2-2, to=2-1]
\end{tikzcd} \]
commutes. Let us give an easy example of this phenomenon.

\begin{lemma}[Essential cartesian base-change] \label{essentialbasechange}
Suppose $f : \mathcal{X} \rightarrow \mathcal{Y}$ is an essential geometric morphism, and let $g : \mathcal{W} \rightarrow \mathcal{Z}$ be another geometric morphism. Then the natural square
\[ \begin{tikzcd}
	{\mathcal{X} \times \mathcal{W}} & {\mathcal{Y} \times \mathcal{W}} \\
	{\mathcal{X} \times \mathcal{Z}} & {\mathcal{Y} \times \mathcal{Z}}
	\arrow["{f \times \mathrm{id}_\mathcal{W}}", from=1-1, to=1-2]
	\arrow["{\mathrm{id}_\mathcal{X} \times g}"', from=1-1, to=2-1]
	\arrow["{\mathrm{id}_\mathcal{Y} \times g}", from=1-2, to=2-2]
	\arrow["{f \times \mathrm{id}_\mathcal{Z}}", from=2-1, to=2-2]
\end{tikzcd} \]
of geometric morphisms is left adjointable, i.e. there exists a commutative square
\[
\begin{tikzcd}
	{\mathcal{X} \otimes \mathcal{W}} &  {\mathcal{Y} \otimes \mathcal{W}} \\
	{\mathcal{X} \otimes \mathcal{Z}}  & {\mathcal{Y} \otimes \mathcal{Z}}
	\arrow["{(\mathrm{id}_\mathcal{X} \otimes g)_*}"', from=1-1, to=2-1]
	\arrow["{f^* \otimes \mathrm{id}_\mathcal{W}}"', from=1-2, to=1-1]
	\arrow["{(\mathrm{id}_\mathcal{Y} \otimes g)_*}", from=1-2, to=2-2]
	\arrow["{f^* \otimes \mathrm{id}_\mathcal{Z}}"', from=2-2, to=2-1]
\end{tikzcd}
\]
via the mate.
\end{lemma}

\begin{proof}
Consider the commuting square of left adjoints
\[ \begin{tikzcd}
	{\mathcal{X} \otimes \mathcal{W}} & {\mathcal{Y} \otimes \mathcal{W}} \\
	{\mathcal{X} \otimes \mathcal{Z}} & {\mathcal{Y} \otimes \mathcal{Z}}
	\arrow["{f_! \otimes \mathrm{id}_\mathcal{W}}", from=1-1, to=1-2]
	\arrow["{\mathrm{id}_\mathcal{X} \otimes g^*}", from=2-1, to=1-1]
	\arrow["{f_! \otimes \mathrm{id}_\mathcal{Z}}"', from=2-1, to=2-2]
	\arrow["{\mathrm{id}_\mathcal{Y} \otimes g^*}"', from=2-2, to=1-2]
\end{tikzcd} \]
We obtain the wanted commuting square by passing to right adjoints, by noting that we have the internal adjunction in $\mathrm{Pr}^L$,
\[\begin{tikzcd}
	{\mathcal{X}} & {\mathcal{Y}.}
	\arrow[""{name=0, anchor=center, inner sep=0}, "{f_!}", curve={height=-12pt}, from=1-1, to=1-2]
	\arrow[""{name=1, anchor=center, inner sep=0}, "{f^*}", curve={height=-12pt}, from=1-2, to=1-1]
	\arrow["\dashv"{anchor=center, rotate=-90}, draw=none, from=0, to=1]
\end{tikzcd}\]
and $-\otimes-$ is a $2$-functor in each variable, hence $ f_! \otimes \mathrm{id}_{\mathcal{E}} \dashv f^* \otimes  \mathrm{id}_{\mathcal{E}}$ holds for each presentable $\infty$-category $\mathcal{E}$.
\end{proof}

Note that the $\infty$-category $\mathrm{Fun}^{\mathrm{lex,acc}}(\mathrm{An}, \mathrm{An})$ has a natural product, namely composition $f \circ g$ of functors $f, g : \mathrm{An} \rightarrow \mathrm{An}$. This is a natural example of an \emph{$E_1$-monoidal structure}. By the equivalence with $\mathrm{Pro}(\mathrm{An})$, this gives a (non-symmetric!) product of pro-anima.

\begin{definition} Let $X,Y \in \mathrm{Pro}(\mathrm{An})$. Let $f_X,f_Y : \mathrm{An} \rightarrow \mathrm{An}$ be the corresponding functors. The \emph{substitution product}
\[ X \triangleright Y \]
is defined to be the pro-anima corresponding to the composition $f_X \circ f_Y : \mathrm{An} \rightarrow \mathrm{An}$.
\end{definition}

Let us spell out what this looks like concretely. If $X = \text{``} \mathrm{lim}_{i \in I} \text{''} X_i$ and $Y = \text{``} \mathrm{lim}_{j \in J} \text{''} Y_j$, we see that the functor for $X \triangleright Y$ is given by the formula
\[ A \mapsto \mathrm{colim}_{i \in I} \mathrm{Map}(X_i, \mathrm{colim}_{j \in J} \mathrm{Map}(Y_j, A) ). \]

\begin{proposition} \label{substitutiontensor}
Let $X,Y \in \mathrm{Pro}(\mathrm{An})$. Suppose that either:
\begin{enumerate}
\item $X$ is pro-compact, i.e. $X \in \mathrm{Pro}(\mathrm{An}^\omega)$, or
\item $Y$ is constant, i.e. $Y \in \mathrm{An} \subset \mathrm{Pro}(\mathrm{An})$.
\end{enumerate}
Then
\[ X \triangleright Y \cong X \times Y \]
holds in $\mathrm{Pro}(\mathrm{An})$.
\end{proposition}

\begin{proof}
First assume that $X = \text{``} \mathrm{lim}_{i \in I} \text{''} X_i$ is pro-compact, with $X_i \in \mathrm{An}^\omega$. Then for any $A \in \mathrm{An}$ it holds that
\[ \begin{array}{rcl}
\mathrm{colim}_{i \in I} \mathrm{Map}(X_i, \mathrm{colim}_{j \in J} \mathrm{Map}(Y_j, A) ) & \cong & \mathrm{colim}_{i \in I} \mathrm{colim}_{j \in J} \mathrm{Map}(X_i,  \mathrm{Map}(Y_j, A) ) \\ & \cong & 
\mathrm{colim}_{(i,j) \in I \times J} \mathrm{Map}(X_i \times Y_j, A).
\end{array}  \]
But this is the functor corresponding to $X \times Y$.

Next, assume that $Y$ is constant. Then for any $A \in \mathrm{An}$ it holds that
\[ \mathrm{colim}_{i \in I} \mathrm{Map}(X_i, \mathrm{Map}(Y, A) ) \cong \mathrm{colim}_{i \in I} \mathrm{Map}(X_i \times Y, A). \]
But this is the functor corresponding to $X \times Y$.
\end{proof}

\begin{theorem}[Shape-theoretic Künneth theorem, locally contractible version] \label{kunnethlocallycontractible}
Let $\mathcal{X}, \mathcal{Y}$ be topoi, and assume that $\mathcal{X}$ is locally contractible. Then
\[ \Pi_\infty( \mathcal{X} \times^\mathrm{RTop} \mathcal{Y} ) \cong \Pi_\infty( \mathcal{X} ) \triangleright \Pi_\infty( \mathcal{Y} ). \]
In particular, if either $\mathcal{X}$ has pro-compact shape or the shape of $\mathcal{Y}$ is constant, then
\[ \Pi_\infty( \mathcal{X} \times^\mathrm{RTop}  \mathcal{Y} ) \cong \Pi_\infty( \mathcal{X} ) \times \Pi_\infty( \mathcal{Y} ) \]
holds in $\mathrm{Pro}(\mathrm{An})$.
\end{theorem}

\begin{proof}
Consider the commutative diagram
\[ \begin{tikzcd}
	& {\mathcal{Y}} & {\mathrm{An}} & \\
	{\mathrm{An}} & {\mathcal{X} \otimes \mathcal{Y}} & {\mathcal{X}} & {\mathrm{An}}
	\arrow["{\mathcal{Y}_*}", from=1-2, to=1-3]
	\arrow["{\mathcal{X}^* \otimes \mathrm{id}_\mathcal{Y}}", from=1-2, to=2-2]
	\arrow["{\mathcal{X}^*}", from=1-3, to=2-3]
	\arrow["{\mathcal{Y}^*}", from=2-1, to=1-2]
	\arrow["{(\mathcal{X} \otimes \mathcal{Y})^*}"', from=2-1, to=2-2]
	\arrow["{(\mathrm{id}_\mathcal{X} \otimes \mathcal{Y}^*)_*}", from=2-2, to=2-3]
	\arrow["{(\mathcal{X} \otimes \mathcal{Y})_*}"', curve={height=18pt}, from=2-2, to=2-4]
	\arrow["{\mathcal{X}_*}", from=2-3, to=2-4]
\end{tikzcd} \]
where the central square is commutative by Lemma \ref{essentialbasechange}. Hence it holds that
\[
(\mathcal{X} \otimes \mathcal{Y})_*
(\mathcal{X} \otimes \mathcal{Y})^*
\cong 
\mathcal{X}_* \mathcal{X}^* \mathcal{Y}_* \mathcal{Y}^*,
\]
in other words
\[ \Pi_\infty( \mathcal{X} \times^\mathrm{RTop} \mathcal{Y} ) \cong \Pi_\infty( \mathcal{X} ) \triangleright \Pi_\infty( \mathcal{Y} ). \]
\end{proof}

The other version of the Künneth theorem that we would like to present is concerned with properness, see Definition \ref{definitionproper}. As with the theorem about inverse limits, a version with stronger conditions has appeared in \cite[Proposition 2.11]{Hoyois2018HigherGaloisTheory}. The proof given there is essentially the same.

\begin{theorem}[Shape-theoretic Künneth theorem, proper version] \label{shapekuenneth}
Let $\mathcal{X}, \mathcal{Y}$ be topoi, and suppose $\mathcal{X}$ is proper. Then
\[
\Pi_\infty( \mathcal{X} \times^{\mathrm{RTop}} \mathcal{Y} )
\simeq
\Pi_\infty( \mathcal{Y} ) \triangleright \Pi_\infty( \mathcal{X} ).
\]
In particular, if either $\mathcal{X}$ has constant shape or $\mathcal{Y}$ has pro-compact shape, it holds that
\[
\Pi_\infty( \mathcal{X} \times^{\mathrm{RTop}} \mathcal{Y} )
\simeq
\Pi_\infty( \mathcal{X} ) \times \Pi_\infty( \mathcal{Y} ).
\]
\end{theorem}

\begin{proof}
We have the commuting square of left adjoints
\[
\begin{tikzcd}
	{\mathcal{X} \otimes \mathcal{Y}} & {\mathcal{X}} \\
	{\mathcal{Y}} & {\mathrm{An}}
	\arrow["{\mathrm{id}_\mathcal{X} \otimes \mathcal{Y}^*}"', from=1-2, to=1-1]
	\arrow["{\mathcal{X}^* \otimes \mathrm{id}_\mathcal{Y}}", from=2-1, to=1-1]
	\arrow["{\mathcal{X}^*}"', from=2-2, to=1-2]
	\arrow["{\mathcal{Y}^*}", from=2-2, to=2-1]
\end{tikzcd}
\]
and similarly, the commuting square of right adjoints
\[
\begin{tikzcd}
	{\mathcal{X} \otimes \mathcal{Y}} & {\mathcal{X}} \\
	{\mathcal{Y}} & {\mathrm{An}}.
	\arrow["{(\mathrm{id}_\mathcal{X} \otimes \mathcal{Y}^*)_*}", from=1-1, to=1-2]
	\arrow["{(\mathcal{X}^* \otimes \mathrm{id}_\mathcal{Y})_*}"', from=1-1, to=2-1]
	\arrow["{\mathcal{X}_*}", from=1-2, to=2-2]
	\arrow["{\mathcal{Y}_*}"', from=2-1, to=2-2]
\end{tikzcd}
\]
Applying properness of $\mathcal{X}$, we also see that the square
\[
\begin{tikzcd}
	{\mathcal{X} \otimes \mathcal{Y}} & {\mathcal{X}} \\
	{\mathcal{Y}} & {\mathrm{An}}
	\arrow["{(\mathcal{X}^* \otimes \mathrm{id}_\mathcal{Y})_*}"', from=1-1, to=2-1]
	\arrow["{\mathrm{id}_\mathcal{X} \otimes \mathcal{Y}^*}"', from=1-2, to=1-1]
	\arrow["{\mathcal{X}_*}", from=1-2, to=2-2]
	\arrow["{\mathcal{Y}^*}", from=2-2, to=2-1]
\end{tikzcd}
\]
commutes. Combining these commutativity statements, we see that the shape of $\mathcal{X} \otimes \mathcal{Y}$ is computed from the composition
\[
\mathcal{X}_*
(\mathrm{id}_\mathcal{X} \otimes \mathcal{Y}^*)_*
(\mathrm{id}_\mathcal{X} \otimes \mathcal{Y}^*)
\mathcal{X}^*
\simeq
\mathcal{Y}_*
(\mathcal{X}^* \otimes \mathrm{id}_\mathcal{Y})_*
(\mathrm{id}_\mathcal{X} \otimes \mathcal{Y}^*)
\mathcal{X}^*
\simeq
\mathcal{Y}_* \mathcal{Y}^* \mathcal{X}_* \mathcal{X}^*.
\]
Now this means, translating to pro-anima, that
\[
\Pi_\infty( \mathcal{X} \times^{\mathrm{RTop}} \mathcal{Y} )
\simeq
\Pi_\infty( \mathcal{Y} ) \triangleright \Pi_\infty( \mathcal{X} ).
\]
\end{proof}

We also obtain the perhaps surprising result on commutation of shape with certain infinite products.

\begin{corollary}
Let $\mathcal{X}_i, i \in I$ be a family of proper topoi. Then
\[ \Pi_\infty( \prod_{i \in I} \mathcal{X}_i ) \simeq \prod_{i \in I} \Pi_\infty( \mathcal{X}_i ). \]
\end{corollary}

\begin{proof}
This follows directly from combining Theorem \ref{shapeinverselimits} with Theorem \ref{shapekuenneth}.
\end{proof}

We record the corresponding versions for topological spaces and/or locales.

\begin{corollary}[Shape-theoretic Künneth theorem, localic version]
Let $X$ be a stably compact space and $L$ a locale. Assume that either $X$ has constant shape or $L$ has pro-compact shape. Then
\[
\Pi_\infty( X \times^{\mathrm{Loc}} L )
\simeq
\Pi_\infty( X ) \times \Pi_\infty( L ).
\]
\end{corollary}

\begin{corollary}[Shape-theoretic Künneth theorem, topological version] \label{shapekunnethtopological}
Let $X$ be a stably compact space and $Y$ a topological space. Assume that either $X$ has constant shape or $Y$ has pro-compact shape. Then
\[
\Pi_\infty( X \times^{\mathrm{Top}} Y )
\simeq
\Pi_\infty( X ) \times \Pi_\infty( Y ).
\]
\end{corollary}

\begin{proof}
The claim about locales follows from the statement that $\mathrm{Loc} \hookrightarrow \mathrm{RTop}$ preserves limits, together with the fact that for $X$ stably compact it holds that $\mathrm{Sh}(X)$ is proper by Corollary \ref{stablycompactspacecompacttopos} and Theorem \ref{propercompact}.

The topological case follows since $X$ is core-compact, and therefore
\[ \Omega(X \times^{\mathrm{Top}} Y) \simeq \Omega(X) \times^{\mathrm{Loc}} \Omega(Y). \]
\end{proof}

\section{Cohomology of topoi and shape} \label{cohomology}

Let $\mathcal{E}$ be a presentable $\infty$-category. If $\mathcal{X}$ is a topos, we define
\[ \mathrm{Sh}(\mathcal{X}, \mathcal{E}) := \mathcal{X} \otimes \mathcal{E} \cong \mathrm{Fun}^R( \mathcal{X}^\mathrm{op}, \mathcal{E} ) \]
as the $\infty$-category of $\mathcal{E}$-valued sheaves on $\mathcal{X}$. The following instances are worth highlighting.

\begin{example}
The case $\mathcal{E} = [1]$ recovers the frame of opens
\[ \mathrm{Open}(\mathcal{X}) = \mathrm{Fun}^R( \mathcal{X}^\mathrm{op}, [1]). \]
\end{example}

\begin{example}
The case $\mathcal{E} = \mathrm{Set}$ gives the $1$-truncation
\[ \mathcal{X}_{ \leq 1 } = \mathrm{Fun}^R( \mathcal{X}^\mathrm{op}, \mathrm{Set}) \]
consisting of the $0$-truncated objects of $\mathcal{X}$. The $1$-truncation of a topos is always a $1$-topos. (And every $1$-topos arises in such a way, by an adjunction analogous to the localic reflection for $0$-topoi.)
\end{example}

\begin{example}
The case $\mathcal{E} = \mathrm{Sp}$ defines the $\infty$-category of spectra-valued sheaves on $\mathcal{X}$.
\end{example}

\begin{example}
Let $R$ be a ring spectrum. The case $\mathcal{E} = \mathrm{Mod}_R$ defines the $\infty$-category of sheaves of $R$-modules on $\mathcal{X}$.
\end{example}

\begin{example}
The case $\mathcal{E} = \mathrm{Ab}$ defines the $\infty$-category of sheaves of abelian groups on $\mathcal{X}$. It only depends on the $1$-truncation of $\mathcal{X}$, as $\mathrm{Ab} \simeq \mathrm{Ab} \otimes \mathrm{Set}$.
\end{example}

By functoriality of the Lurie tensor product, the canonical adjunction
\[ \begin{tikzcd}
	{\mathcal{X}} & {\mathrm{An}}
	\arrow[""{name=0, anchor=center, inner sep=0}, "{\mathcal{X}_*}"', curve={height=12pt}, from=1-1, to=1-2]
	\arrow[""{name=1, anchor=center, inner sep=0}, "{\mathcal{X}^*}"', curve={height=12pt}, from=1-2, to=1-1]
	\arrow["\dashv"{anchor=center, rotate=-90}, draw=none, from=1, to=0]
\end{tikzcd} \]
induces the adjunction
\[ \begin{tikzcd}
	{\mathcal{X} \otimes \mathcal{E}} & {\mathcal{E}}
	\arrow[""{name=0, anchor=center, inner sep=0}, "{(\mathcal{X} \otimes \mathrm{id}_\mathcal{E})_*}"', curve={height=12pt}, from=1-1, to=1-2]
	\arrow[""{name=1, anchor=center, inner sep=0}, "{\mathcal{X}^* \otimes \mathrm{id}_\mathcal{E}}"', curve={height=12pt}, from=1-2, to=1-1]
	\arrow["\dashv"{anchor=center, rotate=-90}, draw=none, from=1, to=0]
\end{tikzcd} \]
The left adjoint can be understood as the assignment of an object $E \in \mathcal{E}$ to its constant sheaf $\underline{E}$. The right adjoint can be understood as the global sections functor.

\begin{definition}
Let $\mathcal{E}$ be a presentable $\infty$-category and let $E \in \mathcal{E}$. Let $\mathcal{X}$ be a topos. The \emph{$E$-valued cohomology} of $\mathcal{X}$ is defined to be
\[ H^\bullet(\mathcal{X}, E) := (\mathcal{X} \otimes \mathrm{id}_\mathcal{E})_*( \mathcal{X}^* \otimes \mathrm{id}_\mathcal{E} )(E) \in \mathcal{E}. \]
\end{definition}

\begin{remark}
The expression $H^\bullet(\mathcal{X}, E)$ is contravariantly functorial in the topos $\mathcal{X}$, and covariantly functorial in $E$. The proof of this is completely analogous to how one shows that shape is functorial.
\end{remark}

\begin{example}
The case $\mathcal{E} = \mathrm{Mod}_{H\mathbb{Z}} = D(\mathbb{Z})$, and $E = H\mathbb{Z}$ recovers the classically known \emph{$\mathbb{Z}$-valued sheaf cohomology} of a topos as
\[ H^n(\mathcal{X}, \mathbb{Z}) = \pi_{-n} H^\bullet(\mathcal{X}, \mathbb{Z}) \]
for $n \in \mathbb{Z}$.
\end{example}

Any presentable $\infty$-category $\mathcal{E}$ admits a \emph{tensoring} and \emph{cotensoring} by $\mathrm{An}$. Define for $X \in \mathrm{An}$ and $E \in \mathcal{E}$ the expressions
\[ X \otimes E := \mathrm{colim}_{X} E \qquad \text{ and } \qquad  E^X := \mathrm{lim}_{X} E \]
where the colimit, respectively limit, is over the constant diagram with value $E$. Alternatively, $X \otimes E$, respectively $E^X$, are defined via the universal properties
\[ \mathrm{Map}_{\mathrm{An}}( X, \mathrm{Map}_\mathcal{E}( E, D ) ) \cong \mathrm{Map}_\mathcal{E}( X \otimes E, D ) \]
and
\[ \mathrm{Map}_{\mathrm{An}}( X, \mathrm{Map}_\mathcal{E}( D, E ) ) \cong \mathrm{Map}_\mathcal{E}( D, E^X ), \]
natural in $D \in \mathcal{E}$. One observes that (co)-tensoring comes from functors
\[ \begin{array}{rcl}
\mathrm{An} \times \mathcal{E} & \rightarrow & \mathcal{E} \\
(X,E) & \mapsto & X \otimes E
\end{array} \]
which is the left Kan extension of the identity on $\mathcal{E}$ along 
\[\mathcal{E} \cong \ast \times \mathcal{E} \rightarrow \mathrm{An} \times \mathcal{E},\]
as well as
\[ \begin{array}{rcl}
\mathrm{An}^\mathrm{op} \times \mathcal{E} & \rightarrow & \mathcal{E} \\
(X,E) & \mapsto & E^X
\end{array} \]
which is the right Kan extension of the identity on $\mathcal{E}$ along 
\[\mathcal{E} \cong \ast \times \mathcal{E} \rightarrow \mathrm{An}^\mathrm{op} \times \mathcal{E}.\]
It is immediate from the given universal properties that for fixed $X \in \mathrm{An}$ we have an induced adjunction
\[ \begin{tikzcd}
	{\mathcal{E}} & {\mathcal{E}.}
	\arrow[""{name=0, anchor=center, inner sep=0}, "{X \otimes (-)}", curve={height=-12pt}, from=1-1, to=1-2]
	\arrow[""{name=1, anchor=center, inner sep=0}, "{(-)^X}", curve={height=-12pt}, from=1-2, to=1-1]
	\arrow["\dashv"{anchor=center, rotate=-90}, draw=none, from=0, to=1]
\end{tikzcd} \]

\begin{theorem} \label{shapeandcohomologylocallycontractible}
Let $\mathcal{X}$ be a locally contractible topos, and $\mathcal{E}$ a presentable $\infty$-category. Let $E \in \mathcal{E}$. Then there exists a natural isomorphism
\[ H^\bullet(\mathcal{X},E) \cong E^{\Pi_\infty(\mathcal{X})}. \]
\end{theorem}

\begin{proof}
Applying the $2$-functor $- \otimes \mathcal{E} : \mathrm{Pr}^L \rightarrow \mathrm{Pr}^L$ to the internal adjunction
\[ \begin{tikzcd}
	{\mathcal{X}} & {\mathrm{An}}
	\arrow[""{name=0, anchor=center, inner sep=0}, "{\Pi_\infty^\mathcal{X}}", curve={height=-12pt}, from=1-1, to=1-2]
	\arrow[""{name=1, anchor=center, inner sep=0}, "{\mathcal{X}_*}"', curve={height=12pt}, from=1-1, to=1-2]
	\arrow[""{name=2, anchor=center, inner sep=0}, "{\mathcal{X}^*}"{description}, from=1-2, to=1-1]
	\arrow["\dashv"{anchor=center, rotate=-90}, draw=none, from=0, to=2]
	\arrow["\dashv"{anchor=center, rotate=-90}, draw=none, from=2, to=1]
\end{tikzcd} \]
we obtain the internal adjunction
\[ \begin{tikzcd}
	{\mathrm{Sh}(\mathcal{X},\mathcal{E})} && {\mathcal{E}.}
	\arrow[""{name=0, anchor=center, inner sep=0}, "{\Pi_\infty^\mathcal{X} \otimes \mathrm{id}_\mathcal{E}}", curve={height=-12pt}, from=1-1, to=1-3]
	\arrow[""{name=1, anchor=center, inner sep=0}, "{( \mathcal{X} \otimes \mathrm{id}_\mathcal{E})_*}"', curve={height=12pt}, from=1-1, to=1-3]
	\arrow[""{name=2, anchor=center, inner sep=0}, "{\mathcal{X}^* \otimes \mathrm{id}_\mathcal{E}}"{description}, from=1-3, to=1-1]
	\arrow["\dashv"{anchor=center, rotate=-90}, draw=none, from=0, to=2]
	\arrow["\dashv"{anchor=center, rotate=-90}, draw=none, from=2, to=1]
\end{tikzcd} \]
These induce the adjunction
\[ \begin{tikzcd}
	{\mathcal{E}} & {\mathcal{E}.}
	\arrow[""{name=0, anchor=center, inner sep=0}, "{\Pi_\infty(\mathcal{X}) \otimes (-)}", curve={height=-12pt}, from=1-1, to=1-2]
	\arrow[""{name=1, anchor=center, inner sep=0}, "{H^\bullet(\mathcal{X}, - )}", curve={height=-12pt}, from=1-2, to=1-1]
	\arrow["\dashv"{anchor=center, rotate=-90}, draw=none, from=0, to=1]
\end{tikzcd}\]
But the right adjoint of $\Pi_\infty(\mathcal{X}) \otimes (-)$ is given by $(-)^{\Pi_\infty(\mathcal{X})}$, hence the claim follows.
\end{proof}

We can extend the cotensoring by animae on a presentable $\infty$-category $\mathcal{E}$ to a cotensoring by pro-animae. If $X = \text{``}\mathrm{lim}_{i \in I}\text{''} X_i \in \mathrm{Pro}(\mathrm{An})$ we extend the definition of the cotensoring to
\[ E^X := \mathrm{colim}_{i \in I} E^{X_i}. \]
This is given by a functor 
\[ \begin{array}{rcl}
\mathrm{Pro}(\mathrm{An})^\mathrm{op} \times \mathcal{E} & \rightarrow & \mathcal{E} \\
(X,E) & \mapsto & E^X,
\end{array} \]
now obtained via left Kan extending the constant case.

\begin{theorem} \label{shapeandcohomology}
Let $\mathcal{X}$ be a topos, and $\mathcal{E}$ a compactly generated presentable $\infty$-category. Let $E \in \mathcal{E}$. Then there exists a natural isomorphism
\[ H^\bullet(\mathcal{X},E) \cong E^{\Pi_\infty(\mathcal{X})}. \]
\end{theorem}

Recall that $\mathcal{E}$ is a compactly generated $\infty$-category iff it is of the form
\[ \mathcal{E} \simeq \mathrm{Ind}(\mathcal{E}^\omega) \simeq \mathrm{Fun}^\mathrm{lex}( (\mathcal{E}^\omega)^\mathrm{op}, \mathrm{An} ). \]
Moreover, for any topos $\mathcal{X}$ we can then see, by symmetry of the Lurie tensor product,  that
\[ \mathcal{X} \otimes \mathcal{E} \simeq \mathrm{Fun}^R( \mathcal{E}^\mathrm{op}, \mathcal{X} ) \simeq \mathrm{Fun}^\mathrm{lex}( (\mathcal{E}^\omega)^\mathrm{op}, \mathcal{X} ). \]
Since both $\mathcal{X}^*$ and $\mathcal{X}_*$ are finite-limit preserving, we see that we can understand the adjunction
\[ \begin{tikzcd}
	{\mathcal{X} \otimes \mathcal{E}} & {\mathcal{E}}
	\arrow[""{name=0, anchor=center, inner sep=0}, "{(\mathcal{X} \otimes \mathrm{id}_\mathcal{E})_*}"', curve={height=12pt}, from=1-1, to=1-2]
	\arrow[""{name=1, anchor=center, inner sep=0}, "{\mathcal{X}^* \otimes \mathrm{id}_\mathcal{E}}"', curve={height=12pt}, from=1-2, to=1-1]
	\arrow["\dashv"{anchor=center, rotate=-90}, draw=none, from=1, to=0]
\end{tikzcd} \] 
as induced via post-composition
\[ \begin{tikzcd}
	{\mathcal{X} \otimes \mathcal{E}} & {\mathrm{Fun}^\mathrm{lex}((\mathcal{E}^\omega)^\mathrm{op}, \mathcal{X})} & {\mathrm{Fun}^\mathrm{lex}((\mathcal{E}^\omega)^\mathrm{op}, \mathrm{An})} & {\mathcal{E}}
	\arrow["\simeq"{description}, draw=none, from=1-1, to=1-2]
	\arrow[""{name=0, anchor=center, inner sep=0}, "{(\mathcal{X}_*)_\circ}"', curve={height=12pt}, from=1-2, to=1-3]
	\arrow[""{name=1, anchor=center, inner sep=0}, "{(\mathcal{X}^*)_\circ}"', curve={height=12pt}, from=1-3, to=1-2]
	\arrow["\simeq"{description}, draw=none, from=1-4, to=1-3]
	\arrow["\dashv"{anchor=center, rotate=-90}, draw=none, from=1, to=0]
\end{tikzcd} \]	
With this being understood, we can now provide the proof of Theorem \ref{shapeandcohomology}.

\begin{proof}
Write $\Pi_\infty(\mathcal{X}) = \text{``}\mathrm{lim}_{i \in I}\text{''} X_i \in \mathrm{Pro}(\mathrm{An})$. Note that $E \in \mathcal{E}$ corresponds under the identification 
\[\mathcal{E} \simeq \mathrm{Fun}^\mathrm{lex}((\mathcal{E}^\omega)^\mathrm{op}, \mathrm{An})\]
to the functor 
\[\mathrm{Map}_\mathcal{E}( - , E ) : (\mathcal{E}^\omega)^\mathrm{op} \rightarrow \mathrm{An}. \]
Hence the object $H^\bullet(\mathcal{X},E)$ corresponds to
\[\begin{array}{rcl}
\mathcal{X}_* \mathcal{X}^* \mathrm{Map}_\mathcal{E}( - , E ) & \simeq & \mathrm{colim}_{i \in I} \mathrm{Map}_{\mathrm{An}}( X_i, \mathrm{Map}_\mathcal{E}( - , E )) \\ 
& \simeq & \mathrm{colim}_{i \in I}  \mathrm{Map}_\mathcal{E}( - , E^{X_i} ) \\
& \simeq &  \mathrm{Map}_\mathcal{E}( - , \mathrm{colim}_{i \in I} E^{X_i} )
\end{array} \]
where in the final line we are allowed to pull the colimit through as we only evaluate against compact objects in $\mathcal{E}$. Hence it follows that
\[ H^\bullet(\mathcal{X},E) \cong \mathrm{colim}_{i \in I} E^{X_i} \cong E^{\Pi_\infty(\mathcal{X})}. \]
\end{proof}

The statement of Theorem \ref{shapeandcohomology} works more generally for \emph{compactly assembled} target categories.

\begin{definition}[{\cite[Section 21.1.2]{Lurie2018SAG}}]
A presentable $\infty$-category $\mathcal{E}$ is called \emph{compactly assembled} if there exist adjunctions
\[ \begin{tikzcd}
	{\mathcal{E}} & {\mathcal{E}'}
	\arrow[""{name=0, anchor=center, inner sep=0}, "{i_!}", curve={height=-12pt}, hook, from=1-1, to=1-2]
	\arrow[""{name=1, anchor=center, inner sep=0}, "{i_*}"', curve={height=12pt}, hook, from=1-1, to=1-2]
	\arrow[""{name=2, anchor=center, inner sep=0}, "{i^*}"{description}, from=1-2, to=1-1]
	\arrow["\dashv"{anchor=center, rotate=-90}, draw=none, from=0, to=2]
	\arrow["\dashv"{anchor=center, rotate=-90}, draw=none, from=2, to=1]
\end{tikzcd} \]
with $\mathcal{E}'$ compactly generated, and $i^*$ both a left and right Bousfield localization.
\end{definition}

Note that this can be expressed as saying that $\mathcal{E}$ arises from an \emph{internal} right Bousfield localization of a compactly generated $\infty$-category in $\mathrm{Pr}^L$.

\begin{remark}
Compactly assembled presentable $\infty$-categories are the natural $\infty$-categorical generalizations of continuous lattices. The condition is closely related to local compactness. In fact, a topos $\mathcal{X}$ is exponentiable in $\mathrm{RTop}$ iff it is compactly assembled, see \cite{anel2018exponentiablehighertoposes}.
\end{remark}

\begin{example}
Any compactly generated presentable $\infty$-category is compactly assembled. This includes the $\infty$-category $\mathrm{Mod}_R$ for a ring spectrum $R$, in particular the derived $\infty$-category of a ring.
\end{example}

\begin{example}
Any locally compact frame is compactly assembled.
\end{example}

\begin{example}[{\cite[Theorem 3.50]{Lehner2026KTheoryStablyCompact}}]
If $X$ is a stably locally compact space, then $\mathrm{Sh}(X)$ is compactly assembled.
\end{example}

\begin{theorem} \label{shapeandcohomologycompactlyassembled}
Let $\mathcal{X}$ be a topos, and $\mathcal{E}$ a compactly assembled presentable $\infty$-category. Let $E \in \mathcal{E}$. Then there exists a natural isomorphism
\[ H^\bullet(\mathcal{X},E) \cong E^{\Pi_\infty(\mathcal{X})}. \]
\end{theorem}

\begin{proof}
Choose a compactly generated presentable $\infty$-category $\mathcal{E}'$ together with adjunctions
\[ \begin{tikzcd}
	{\mathcal{E}} & {\mathcal{E}'}
	\arrow[""{name=0, anchor=center, inner sep=0}, "{i_!}", curve={height=-12pt}, hook, from=1-1, to=1-2]
	\arrow[""{name=1, anchor=center, inner sep=0}, "{i_*}"', curve={height=12pt}, hook, from=1-1, to=1-2]
	\arrow[""{name=2, anchor=center, inner sep=0}, "{i^*}"{description}, from=1-2, to=1-1]
	\arrow["\dashv"{anchor=center, rotate=-90}, draw=none, from=0, to=2]
	\arrow["\dashv"{anchor=center, rotate=-90}, draw=none, from=2, to=1]
\end{tikzcd} \]
Note the following. Since $i^*$ preserves both limits and colimits, it follows that for $F \in \mathcal{E}'$, and $X = \text{``}\mathrm{lim}_{i \in I}\text{''} X_i \in \mathrm{Pro}(\mathrm{An})$ it holds that
\[ i^*( F^X ) \cong i^*( \mathrm{colim}_{i \in I} F^{X_i} ) \cong  \mathrm{colim}_{i \in I} i^*(F)^{X_i}  \cong i^*(F)^X. \]	
	
For the sake of this proof, denote the induced adjunction from the geometric morphism $\mathcal{X} \rightarrow \mathrm{An}$ by
\[ \begin{tikzcd}
	{\mathcal{X} \otimes \mathcal{E}} & {\mathcal{E}}
	\arrow[""{name=0, anchor=center, inner sep=0}, "{\mathcal{E}_*}"', curve={height=12pt}, from=1-1, to=1-2]
	\arrow[""{name=1, anchor=center, inner sep=0}, "{\mathcal{E}^*}"', curve={height=12pt}, from=1-2, to=1-1]
	\arrow["\dashv"{anchor=center, rotate=-90}, draw=none, from=1, to=0]
\end{tikzcd} \]
and similarly for $\mathcal{E}'$. We have the two commuting squares of left adjoints
\[ \begin{tikzcd}
	{\mathcal{X} \otimes \mathcal{E}} & {\mathcal{E}} \\
	{\mathcal{X} \otimes \mathcal{E}'} & {\mathcal{E}'}
	\arrow[""{name=0, anchor=center, inner sep=0}, "{i_!}"', curve={height=6pt}, hook, from=1-1, to=2-1]
	\arrow["{\mathcal{E}^*}"', from=1-2, to=1-1]
	\arrow[""{name=1, anchor=center, inner sep=0}, "{i_!}"', curve={height=6pt}, hook, from=1-2, to=2-2]
	\arrow[""{name=2, anchor=center, inner sep=0}, "{i^*}"', curve={height=6pt}, from=2-1, to=1-1]
	\arrow[""{name=3, anchor=center, inner sep=0}, "{i^*}"', curve={height=6pt}, from=2-2, to=1-2]
	\arrow["{\mathcal{E}'^*}", from=2-2, to=2-1]
	\arrow["\dashv"{anchor=center}, draw=none, from=0, to=2]
	\arrow["\dashv"{anchor=center}, draw=none, from=1, to=3]
\end{tikzcd} \]
from which it follows that we also have the commuting square of right adjoints
\[ \begin{tikzcd}
	{\mathcal{X} \otimes \mathcal{E}} & {\mathcal{E}} \\
	{\mathcal{X} \otimes \mathcal{E}'} & {\mathcal{E}'}
	\arrow["{\mathcal{E}_*}", from=1-1, to=1-2]
	\arrow["{i^*}", from=2-1, to=1-1]
	\arrow["{\mathcal{E}'_*}", from=2-1, to=2-2]
	\arrow["{i^*}"', from=2-2, to=1-2]
\end{tikzcd} \]
Now put everything together to observe that for $E \in \mathcal{E}$,
\[ \begin{array}{rcl}
 H^\bullet(\mathcal{X},E) = \mathcal{E}_* \mathcal{E}^*(E) & \cong & \mathcal{E}_* \mathcal{E}^* i^* i_! (E) \\ 
 & \cong & \mathcal{E}_*  i^* \mathcal{E}'^* i_! (E) \\
 & \cong &  i^* \mathcal{E}'_* \mathcal{E}'^* i_! (E) \\
 & \cong & i^*( i_! (E)^{\Pi_\infty(\mathcal{X})} ) \\
 & \cong & ( i^* i_! (E))^{\Pi_\infty(\mathcal{X})} \\ 
 & \cong & E^{\Pi_\infty(\mathcal{X})}.
\end{array} \]
\end{proof}

\section{Homology of topoi and shape}

We have discussed the relationship of shape with cohomology in Section \ref{cohomology}. Here we would like to discuss the relationship with the dual notion: \emph{homology}. Surprisingly enough, this is somewhat more subtle than the story for cohomology.

\subsection{Pro-homology}

Before we get to actual homology, there is a canonical and well-behaved sense in which shape determines homology, if one is willing to work with \emph{pro-homology}.

Dual to the notion of cotensoring introduced earlier, there is the notion of a tensoring. In the following, let $\mathcal{E}$ be a presentable $\infty$-category. If $X \in \mathrm{An}$ and $E \in \mathcal{E}$ define
\[ X \otimes E := \mathrm{colim}_X \mathrm{const}_E \in \mathcal{E}. \]
This can be seen to arise from a functor
\[ \mathrm{An} \times \mathcal{E} \rightarrow \mathcal{E} \]
preserving colimits in both variables, whose induced left adjoint functor
\[ \mathrm{An} \otimes \mathcal{E} \simeq \mathcal{E} \]
is the functor identifying $\mathrm{An}$ as the unit for $\otimes$ on $\mathrm{Pr}^L$.

Now fix $E \in \mathcal{E}$. Then tensoring gives an adjunction
\[ \begin{tikzcd}
	{\mathrm{An}} & {\mathcal{E}.}
	\arrow[""{name=0, anchor=center, inner sep=0}, "{-\otimes E}", curve={height=-12pt}, from=1-1, to=1-2]
	\arrow[""{name=1, anchor=center, inner sep=0}, "{\mathrm{Map}_\mathcal{E}(E,-)}", curve={height=-12pt}, from=1-2, to=1-1]
	\arrow["\dashv"{anchor=center, rotate=-90}, draw=none, from=0, to=1]
\end{tikzcd} \]
Applying the $2$-functor $\mathrm{Pro}$, we get the induced adjunction
\[ \begin{tikzcd}
	{\mathrm{Pro}(\mathrm{An})} & {\mathrm{Pro}(\mathcal{E}).}
	\arrow[""{name=0, anchor=center, inner sep=0}, "{\mathrm{Pro}(-\otimes E)}", curve={height=-12pt}, from=1-1, to=1-2]
	\arrow[""{name=1, anchor=center, inner sep=0}, "{\mathrm{Pro}(\mathrm{Map}_\mathcal{E}(E,-))}", curve={height=-12pt}, from=1-2, to=1-1]
	\arrow["\dashv"{anchor=center, rotate=-90}, draw=none, from=0, to=1]
\end{tikzcd} \]
Concretely, this sends a pro-anima $\text{``} \mathrm{lim}_{i \in I} \text{''} X_i \in \mathrm{Pro}(\mathrm{An})$ to the pro-object
\[ 	\text{``} \mathrm{lim}_{i \in I} \text{''} (X_i \otimes E) \]
in $\mathrm{Pro}(\mathcal{E})$. Composing this adjunction with the shape adjunction gives
	\[ \begin{tikzcd}
	{\mathrm{RTop}} & {\mathrm{Pro}(\mathrm{An})} & {\mathrm{Pro}(\mathcal{E})}
	\arrow["{\Pi_\infty}", curve={height=-12pt}, from=1-1, to=1-2]
	\arrow["\beta", curve={height=-12pt}, from=1-2, to=1-1]
	\arrow[""{name=0, anchor=center, inner sep=0}, "{\mathrm{Pro}(-\otimes E)}", curve={height=-12pt}, from=1-2, to=1-3]
	\arrow[""{name=1, anchor=center, inner sep=0}, "{\mathrm{Pro}(\mathrm{Map}_\mathcal{E}(E,-))}", curve={height=-12pt}, from=1-3, to=1-2]
	\arrow["\dashv"{anchor=center, rotate=-90}, draw=none, from=0, to=1]
\end{tikzcd} \]

\begin{definition}
Let $\mathcal{X}$ be a topos, $\mathcal{E}$ a presentable $\infty$-category and $E \in \mathcal{E}$. The \emph{pro-homology} of $\mathcal{X}$ with values in $E$ is defined as
\[ H_\bullet^\mathrm{pro}( \mathcal{X}, E ) := \mathrm{Pro}(-\otimes E)(\Pi_\infty(\mathcal{X})) \in \mathrm{Pro}(\mathcal{E}). \]
\end{definition}

By construction, pro-homology is covariant and colimit preserving in the topos $\mathcal{X}$.

\subsection{Homology} \label{homology}

While pro-homology has good formal properties, it lands in pro-objects, and not in $\mathcal{E}$. Topos \emph{homology} does not have this defect. Let us set up the theory.

\begin{definition}
Let $\mathcal{X}$ be a topos and $\mathcal{E}$ be a presentable $\infty$-category. Then the $\infty$-category of $\mathcal{E}$-valued \emph{cosheaves} on $\mathcal{X}$ is defined to be
\[ \mathrm{CoSh}(\mathcal{X}, \mathcal{E}) := \mathrm{Fun}^L( \mathcal{X}, \mathcal{E} ). \]
\end{definition}

The usage of cosheaves is dual to that of sheaves. The intuition behind them is perhaps most cleanly summarized in the following table of corresponding terms.

\begin{table}[h]
\centering
\begin{tabular}{c|c}
Function & Measure \\
\hline
Sheaf & Cosheaf \\
\hline
Cohomology & Homology
\end{tabular}
\end{table}

\begin{remark}
The analogy of cosheaves with measures in the setting of topos theory was originally suggested by Lawvere in the context of 1-topoi, see e.g.\ \cite{Bunge1995} for an overview.
\end{remark}

\begin{example}
Let $\mathcal{X}$ be locally contractible. Then 
\[ \Pi_\infty^\mathcal{X} : \mathcal{X} \rightarrow \mathrm{An} \]
is the \emph{canonical} cosheaf on $\mathcal{X}$.
\end{example}

\begin{example}
Let $\mathcal{X}$ be a topos and $x$ a point of $\mathcal{X}$. Then the stalk functor
\[ x^* : \mathcal{X} \rightarrow \mathrm{An} \]
is the \emph{Dirac} cosheaf based at $x$.
\end{example}

Since $\mathrm{Fun}^L(-,\mathcal{E}) : (\mathrm{Pr}^L)^\mathrm{op} \rightarrow \mathrm{Pr}^L$ is a contravariant functor, the geometric morphism 
\[ \begin{tikzcd}
	{\mathcal{X}} & {\mathrm{An}}
	\arrow[""{name=0, anchor=center, inner sep=0}, "{\mathcal{X}_*}"', curve={height=12pt}, from=1-1, to=1-2]
	\arrow[""{name=1, anchor=center, inner sep=0}, "{\mathcal{X}^*}"', curve={height=12pt}, from=1-2, to=1-1]
	\arrow["\dashv"{anchor=center, rotate=-90}, draw=none, from=1, to=0]
\end{tikzcd} \]
induces the adjunction
\[ \begin{tikzcd}
	{\mathrm{Fun}^L(\mathcal{X}, \mathcal{E})} & {\mathcal{E}}
	\arrow[""{name=0, anchor=center, inner sep=0}, "{(\mathcal{X}^*)^\circ}", curve={height=-12pt}, from=1-1, to=1-2]
	\arrow[""{name=1, anchor=center, inner sep=0}, "{(\mathcal{X}^*)_*}", curve={height=-12pt}, from=1-2, to=1-1]
	\arrow["\dashv"{anchor=center, rotate=-90}, shift left=2, draw=none, from=0, to=1]
\end{tikzcd} \]
with the left adjoint being given by precomposition by $\mathcal{X}^*$. Since the notation is becoming cumbersome, let us write
\[\begin{tikzcd}
	{\mathrm{CoSh}(\mathcal{X}, \mathcal{E})} & {\mathrm{Fun}^L(\mathcal{X}, \mathcal{E})} & {\mathcal{E}}
	\arrow["{=}"{description}, draw=none, from=1-1, to=1-2]
	\arrow[""{name=0, anchor=center, inner sep=0}, "{\mathcal{X}_+}", curve={height=-12pt}, from=1-2, to=1-3]
	\arrow[""{name=1, anchor=center, inner sep=0}, "{\mathcal{X}^+}", curve={height=-12pt}, from=1-3, to=1-2]
	\arrow["\dashv"{anchor=center, rotate=-90}, shift left=2, draw=none, from=0, to=1]
\end{tikzcd}\]

Given its analogy with cohomology, the following definition should not be a surprise.

\begin{definition}
Let $\mathcal{E}$ be a presentable $\infty$-category and let $E \in \mathcal{E}$. Let $\mathcal{X}$ be a topos. The \emph{$E$-valued homology} of $\mathcal{X}$ is defined to be
\[ H_\bullet(\mathcal{X}, E) := \mathcal{X}_+ \mathcal{X}^+ (E) \in \mathcal{E}. \]
\end{definition}

Analogously to the situation of shape and cohomology, one can argue that topos homology is a covariant functor in both $\mathcal{X}$ and $E$. 

\begin{remark}
If $\mathcal{X} = \mathrm{Sh}(X)$ where $X$ is a locally compact Hausdorff space, and $\mathcal{E}$ is presentable and stable, then the adjunction
\[\begin{tikzcd}
	{\mathrm{CoSh}(\mathcal{X}, \mathcal{E})}  & {\mathcal{E}}
	\arrow[""{name=0, anchor=center, inner sep=0}, "{\mathcal{X}_+}", curve={height=-12pt}, from=1-1, to=1-2]
	\arrow[""{name=1, anchor=center, inner sep=0}, "{\mathcal{X}^+}", curve={height=-12pt}, from=1-2, to=1-1]
	\arrow["\dashv"{anchor=center, rotate=-90}, shift left=2, draw=none, from=0, to=1]
\end{tikzcd}\]
identifies, using Verdier duality, with the adjunction
\[ \begin{tikzcd}
	{\mathrm{Sh}(X,\mathcal{E})} & {\mathcal{E}}
	\arrow[""{name=0, anchor=center, inner sep=0}, "{X_!}", curve={height=-12pt}, from=1-1, to=1-2]
	\arrow[""{name=1, anchor=center, inner sep=0}, "{X^!}", curve={height=-12pt}, from=1-2, to=1-1]
	\arrow["\dashv"{anchor=center, rotate=-90}, draw=none, from=0, to=1]
\end{tikzcd} \]
Hence, in this case, topos homology of $\mathrm{Sh}(X)$ identifies with homology as given by the six functor formalism on locally compact Hausdorff spaces, see \cite[Section 4]{krause_nikolaus_puetzstueck}, also \cite{Volpe2021SixOperationsInTopology}.
\end{remark}

Homology satisfies some natural colimit preservation.

\begin{lemma} \label{homologycolimits}
Let $\mathcal{E}$ be a presentable $\infty$-category and let $E \in \mathcal{E}$. The functor
\[ H_\bullet( - , E) : \mathrm{RTop} \rightarrow \mathcal{E} \]
preserves colimits indexed by animae.
\end{lemma}

\begin{remark}
The special case of $S$ being a (discrete) set, viewed as an anima, implies that topos homology is \emph{additive}, i.e. that we have an isomorphism
\[ \coprod_{s \in S} H_\bullet( \mathcal{X}_s , E) \cong H_\bullet( \coprod_{s \in S}^\mathrm{RTop} \mathcal{X}_s , E) \]
for any $S$-indexed family $(\mathcal{X}_s)_{s \in S}$ of topoi.
\end{remark}

\begin{proof}
Let $\mathcal{X}_\bullet : A \rightarrow \mathrm{RTop}$ be a diagram for $A \in \mathrm{An}$ and write $\mathcal{X} = \mathrm{colim}_{a \in A} \mathcal{X}_a $. Keeping in mind that $\mathrm{An}$ is the terminal topos, the natural transformation of geometric morphisms
\[ \mathcal{X}_a \rightarrow \mathrm{An} \]
induces a factorization of geometric morphisms
\[ \mathcal{X} = \mathrm{colim}_{a \in A} \mathcal{X}_a \rightarrow \mathrm{colim}_{a \in A} \mathrm{An} \simeq \mathrm{Fun}(A, \mathrm{An}) \rightarrow \mathrm{An}. \]

Recall that the inclusion $\mathrm{RTop} \rightarrow \mathrm{Pr}^R \simeq (\mathrm{Pr}^L)^\mathrm{op}$ preserves colimits, and that $A$-indexed limits in $\mathrm{Pr}^L$ are computed equivalently as $A$-indexed colimits. The functor $\mathrm{Fun}^L(-, \mathcal{E}) : (\mathrm{Pr}^L)^{\mathrm{op}} \rightarrow \mathrm{Pr}^L$ is an internal hom for a symmetric monoidal structure and therefore sends colimits to limits.\footnote{The reason being that $\mathrm{Fun}^L(-, \mathcal{E})$ is \emph{right adjoint to itself.}} Also recall that limits in $\mathrm{Pr}^L$ are computed as limits in $\widehat{\mathrm{Cat}}_\infty$. Thus by applying the functor $\mathrm{Fun}^L(-, \mathcal{E})$ we get the following adjunctions
\[ \begin{tikzcd}
	{\mathrm{CoSh}(\mathcal{X},\mathcal{E})} & {\mathrm{lim}_{a \in A} \mathrm{CoSh}(\mathcal{X}_a,\mathcal{E})} & {\mathrm{lim}_{a \in A} \mathcal{E}} & {\mathrm{Fun}(A,\mathcal{E})} & {\mathcal{E}} \\
	&&& {}
	\arrow["\simeq"{description}, draw=none, from=1-1, to=1-2]
	\arrow[""{name=0, anchor=center, inner sep=0}, "{({\mathcal{X}_a}_+)_{a \in A}}", curve={height=-12pt}, from=1-2, to=1-3]
	\arrow[""{name=1, anchor=center, inner sep=0}, "R", curve={height=-12pt}, from=1-3, to=1-2]
	\arrow["\simeq"{description}, draw=none, from=1-3, to=1-4]
	\arrow[""{name=2, anchor=center, inner sep=0}, "{\mathrm{colim}}", curve={height=-12pt}, from=1-4, to=1-5]
	\arrow[""{name=3, anchor=center, inner sep=0}, "{\mathrm{const}}", curve={height=-12pt}, from=1-5, to=1-4]
	\arrow["\dashv"{anchor=center, rotate=-90}, shift left=3, draw=none, from=0, to=1]
	\arrow["\dashv"{anchor=center, rotate=-90}, draw=none, from=2, to=3]
\end{tikzcd} \]
where the composite of the left adjoints agrees with $\mathcal{X}_+$, and the composite of the right adjoints agrees with $\mathcal{X}^+$.

Since $A$ is an anima, every morphism $f \colon a \rightarrow a'$ in $A$
is an equivalence. Consequently, the transition geometric morphism
$\mathcal{X}_a \rightarrow \mathcal{X}_{a'}$ and the induced transition
functor on cosheaf categories are equivalences. Taking mates of the
compatibility equivalences for the functors $(\mathcal{X}_a)_+$ shows
that their right adjoints $(\mathcal{X}_a)^+$ also assemble compatibly.
It follows that $R$ is computed objectwise:
\[
R((E_a)_{a \in A})
\simeq
((\mathcal{X}_a)^+(E_a))_{a \in A}.
\]

Putting everything together, we see that
\[ H_\bullet(\mathcal{X},E) = \mathcal{X}_+ \mathcal{X}^+(E) \cong \mathrm{colim}_{a \in A} (\mathcal{X}_a)_+ (\mathcal{X}_a)^+(E) = \mathrm{colim}_{a \in A} H_\bullet(\mathcal{X}_a,E). \]
\end{proof}

\begin{theorem} \label{localcontractiblehomology}
Let $\mathcal{X}$ be a locally contractible topos, $\mathcal{E}$ a presentable $\infty$-category, and $E \in \mathcal{E}$. Then
\[ H_\bullet( \mathcal{X}, E ) \cong \Pi_\infty(\mathcal{X}) \otimes E. \]
\end{theorem}

\begin{proof}
The adjunctions
\[ \begin{tikzcd}
	{\mathcal{X}} & {\mathrm{An},}
	\arrow[""{name=0, anchor=center, inner sep=0}, "{\Pi_\infty^\mathcal{X}}", curve={height=-12pt}, from=1-1, to=1-2]
	\arrow[""{name=1, anchor=center, inner sep=0}, curve={height=12pt}, from=1-1, to=1-2]
	\arrow[""{name=2, anchor=center, inner sep=0}, "{\mathcal{X}^*}"{description}, from=1-2, to=1-1]
	\arrow["\dashv"{anchor=center, rotate=-90}, draw=none, from=0, to=2]
	\arrow["\dashv"{anchor=center, rotate=-90}, draw=none, from=2, to=1]
\end{tikzcd} \]
with both $\Pi_\infty^\mathcal{X}$ and $\mathcal{X}^*$ being left adjoints, leads to the composite
\[ \Pi_\infty^\mathcal{X} \mathcal{X}^* \cong \Pi_\infty(\mathcal{X}) \otimes - : \mathrm{An} \rightarrow \mathrm{An}. \]
Applying the $2$-functor $\mathrm{Fun}^L( - , \mathcal{E}) : (\mathrm{Pr}^L)^\mathrm{op} \rightarrow \mathrm{Pr}^L$ then gives the adjunctions
\[ \begin{tikzcd}
	{\mathrm{Fun}^L(\mathcal{X}, \mathcal{E})} & {\mathcal{E}}
	\arrow[""{name=0, anchor=center, inner sep=0}, "{(\mathcal{X}^*)^\circ}", curve={height=-24pt}, from=1-1, to=1-2]
	\arrow[""{name=1, anchor=center, inner sep=0}, curve={height=24pt}, from=1-1, to=1-2]
	\arrow[""{name=2, anchor=center, inner sep=0}, "{(\Pi_\infty^\mathcal{X})^\circ}"', from=1-2, to=1-1]
	\arrow["\dashv"{anchor=center, rotate=-90}, shift right=4, draw=none, from=0, to=2]
	\arrow["\dashv"{anchor=center, rotate=-90}, shift right=4, draw=none, from=2, to=1]
\end{tikzcd} \]
Hence $H_\bullet( \mathcal{X}, E )$ is given by precomposition with $\Pi_\infty(\mathcal{X}) \otimes - : \mathrm{An} \rightarrow \mathrm{An}$ on $\mathrm{Fun}^L( \mathrm{An}, \mathcal{E} ) \simeq \mathcal{E}$. But this just identifies with tensoring with $\Pi_\infty(\mathcal{X})$ on $\mathcal{E}$.
\end{proof}

The tensoring can of course be extended to $\mathrm{Pro}(\mathrm{An})$ as
\[ \text{``} \mathrm{lim}_{i \in I} X_i \text{''} \otimes E := \mathrm{lim}_{i \in I} (X_i \otimes E) \]
which arises as a functor
\[ \mathrm{Pro}(\mathrm{An}) \times \mathcal{E} \rightarrow \mathcal{E} \]

\subsection{Homology of compact topoi}

It would be pleasant to have, for a general topos $\mathcal{X}$, an analogous formula for homology as was the case for cohomology, saying that it holds that
\[ H_\bullet( \mathcal{X}, E ) = \Pi_\infty(\mathcal{X}) \otimes E. \]
Unfortunately, we need some restrictions on both $\mathcal{X}$ and $\mathcal{E}$ for this to hold.

Recall the notion of a \emph{stable $\infty$-category}. The standard reference is \cite[Chapter 1]{Lurie2017HA}. As additional references, see \cite[Chapter 4]{CnossenStableHomotopy}. We note that a functor $f : \mathcal{C} \rightarrow \mathcal{D}$ between stable $\infty$-categories is called \emph{exact} if it preserves finite limits, or equivalently finite colimits. The $\infty$-category of spectra 
\[ \mathrm{Sp} := \mathrm{lim}( \cdots \rightarrow \mathrm{An}_{\ast /} \xrightarrow{\Omega} \mathrm{An}_{\ast /} ) \]
plays a special role within the context of stable $\infty$-categories: All stable $\infty$-categories are canonically \emph{enriched} over $\mathrm{Sp}$.

\begin{definition}
Let $\mathcal{C}$ be a stable $\infty$-category and $c,d \in \mathcal{C}$. The \emph{mapping spectrum} from $c$ to $d$ is defined as
\[ {\mathrm{map}(c,d)_{\mathcal{C}}}_n := \mathrm{Map}_{\mathcal{C}}(\Omega^n c, d ).  \]
\end{definition}

The mapping spectrum between two objects is in fact functorial in both $c$ and $d$.

\begin{proposition}[{\cite[Proposition 4.4.4]{CnossenStableHomotopy}}]
Let $\mathcal{C}$ be a stable $\infty$-category. There exists a unique functor
$\mathrm{map}_{\mathcal{C}}(-,-) \colon
\mathcal{C}^{\mathrm{op}} \times \mathcal{C} \rightarrow \mathrm{Sp}$
that is exact in both variables and fits into a commutative triangle
\[
\begin{tikzcd}
    & \mathrm{Sp} \arrow[d, "\Omega^\infty"] \\
    \mathcal{C}^{\mathrm{op}} \times \mathcal{C}
        \arrow[ur, "{\mathrm{map}_{\mathcal{C}}(-,-)}"]
        \arrow[r, "{\mathrm{Map}_{\mathcal{C}}(-,-)}"']
    & \mathrm{An}.
\end{tikzcd}
\]
\end{proposition}

\begin{proposition} \label{ambidexterity} Let $\mathcal{C}$ be a stable $\infty$-category, and $X$ a compact anima. Then there exist adjunctions
\[ \begin{tikzcd}
	{\mathcal{C}} & {\mathcal{C}}
	\arrow[""{name=0, anchor=center, inner sep=0}, "{X \otimes - }", curve={height=-24pt}, from=1-1, to=1-2]
	\arrow[""{name=1, anchor=center, inner sep=0}, "{X \otimes - }"', curve={height=12pt}, from=1-1, to=1-2]
	\arrow[""{name=2, anchor=center, inner sep=0}, "{(-)^X}"', from=1-2, to=1-1]
	\arrow["\dashv"{anchor=center, rotate=-90}, draw=none, from=0, to=2]
	\arrow["\dashv"{anchor=center, rotate=-90}, draw=none, from=2, to=1]
\end{tikzcd} \]
\end{proposition}

Let me stress that the adjunction $X \otimes - \dashv (-)^X$ always exists as soon as $\mathcal{C}$ admits tensors and cotensors by finite animae. It is the adjunction $(-)^X \dashv X \otimes - $ that is special in the stable context.

\begin{proof}
Let $c,d \in \mathcal{C}$. Every compact anima is a retract of a finite anima. Since for compact $X$, the natural transformations
\[ \mathrm{map}_\mathcal{C}(c^X,d) \leftarrow X \otimes \mathrm{map}_\mathcal{C}(c,d) \rightarrow  \mathrm{map}_\mathcal{C}(c,X \otimes d) \]
are retracts of the analogous maps involving a finite anima instead of $X$, and a retract of an isomorphism is an isomorphism, it suffices to show the claim under the assumption that $X$ is a finite anima.

Since $\mathrm{map}_\mathcal{C}(-,-)$ is exact in both variables, it follows that
\[ \mathrm{map}_\mathcal{C}(c^X,d) \cong X \otimes \mathrm{map}_\mathcal{C}(c,d) \cong  \mathrm{map}_\mathcal{C}(c,X \otimes d). \]
The statement about mapping animae, which shows that $(-)^X$ is left adjoint to $X \otimes - $, follows by applying 
\[ \Omega^\infty : \mathrm{Sp} \rightarrow \mathrm{An}_{\ast /} \rightarrow \mathrm{An}. \]
\end{proof}

\begin{corollary} Let $\mathcal{E}$ be a presentable stable $\infty$-category and $X \in \mathrm{Pro}(\mathrm{An}^\omega)$. Then there exists an adjunction
\[\begin{tikzcd}
	{\mathcal{E}} & {\mathcal{E}}
	\arrow[""{name=0, anchor=center, inner sep=0}, "{(-)^X}", curve={height=-12pt}, from=1-1, to=1-2]
	\arrow[""{name=1, anchor=center, inner sep=0}, "{X \otimes -}", curve={height=-12pt}, from=1-2, to=1-1]
	\arrow["\dashv"{anchor=center, rotate=-90}, draw=none, from=0, to=1]
\end{tikzcd}\]
\end{corollary}

\begin{proof}
Write $X = \text{``} \mathrm{lim}_{i \in I} \text{''} X_i$, with $X_i \in \mathrm{An}^\omega$. Let $D,E \in \mathcal{E}$. Then
\[ \begin{array}{rcl}
 \mathrm{Map}_\mathcal{E}( D^X, E ) & \cong & \mathrm{Map}_\mathcal{E}(\mathrm{colim}_{i \in I}  D^{X_i}, E ) \\
 & \cong & \mathrm{lim}_{i \in I} \mathrm{Map}_\mathcal{E}(  D^{X_i}, E ) \\
 & \cong & \mathrm{lim}_{i \in I} \mathrm{Map}_\mathcal{E}(  D, X_i \otimes E ) \\ 
  & \cong &  \mathrm{Map}_\mathcal{E}(  D, \mathrm{lim}_{i \in I} (X_i \otimes E )) \\ 
  & \cong &  \mathrm{Map}_\mathcal{E}(  D, X \otimes E ). \\ 
\end{array} \]
\end{proof}

The $\infty$-category $\mathrm{Sp}$ has a special role when viewed as a \emph{presentable} $\infty$-category as well. Note that the definition of $\mathrm{Sp}$ as
\[ \mathrm{Sp} := \mathrm{lim}( \cdots \xrightarrow{ \Omega } \mathrm{An}_{\ast /} \xrightarrow{\Omega} \mathrm{An}_{\ast /} ) \]
means equivalently that
\[ \mathrm{Sp} \cong \mathrm{colim}^{\mathrm{Pr}^L}( \mathrm{An}_{\ast /} \xrightarrow{ \Sigma } \mathrm{An}_{\ast /} \xrightarrow{\Sigma} \cdots  ) \]
with the colimit taken in $\mathrm{Pr}^L$. Since taking Lurie tensor products of presentable $\infty$-categories preserves colimits in either variable, the following claim follows easily.

\begin{proposition}[{\cite[Example 5.2.6(4)--(5), together with the discussion preceding
Example 5.1.1]{CarmeliSchlankYanovski2021}}] \label{stablepresentableproperty}
Let $\mathcal{C}$ be a presentable $\infty$-category.
\begin{itemize}
\item $\mathcal{C}$ is pointed iff the tensor product of $\mathcal{C}$ with the left adjoint $(-)_+ : \mathrm{An} \rightarrow \mathrm{An}_{\ast /}$ induces an equivalence
\[ \mathcal{C} \simeq \mathcal{C} \otimes \mathrm{An}_{\ast /}. \]
\item $\mathcal{C}$ is stable iff the tensor product of $\mathcal{C}$ with the left adjoint $ \Sigma^\infty_+ : \mathrm{An} \rightarrow \mathrm{Sp}$ induces an equivalence
\[ \mathcal{C} \simeq \mathcal{C} \otimes \mathrm{Sp}. \]
\end{itemize}
\end{proposition}

Therefore the functor $- \otimes \mathrm{Sp} : \mathrm{Pr}^L \rightarrow \mathrm{Pr}^L$ can be understood as the \emph{stabilization} of a presentable $\infty$-category. The following statement is almost trivial, but worth pointing out.

\begin{lemma} \label{filteredcolimitsexactfunctor}
Let $f \colon \mathcal{C} \rightarrow \mathcal{D}$ be a left-exact functor between presentable stable $\infty$-categories that preserves filtered colimits. Then $f$ preserves all colimits and admits a right adjoint.
\end{lemma}

\begin{proof}
As $f$ is left exact, it also preserves all finite colimits. But a functor preserving filtered colimits and finite colimits preserves all colimits. Since $\mathcal{C}$ is presentable, it follows that $f$ admits a right adjoint by the Adjoint Functor Theorem.
\end{proof}

The following is a technical lemma that will be quite useful. Let us set up some terminology.

\begin{definition}
Let $\mathcal{C}$ be an $\infty$-category admitting filtered colimits and finite limits. We say that in $\mathcal{C}$ \emph{filtered colimits are exact} if the functor $ k : \mathrm{Ind}(\mathcal{C}) \rightarrow \mathcal{C} $ preserves finite limits.
\end{definition}

\begin{example}
If $\mathcal{C}$ is either a topos or a presentable stable $\infty$-category, then filtered colimits in $\mathcal{C}$ are exact.
\end{example}

\begin{lemma} \label{perfectfunctorstable}
Suppose
\[\begin{tikzcd}
	{\mathcal{C}} & {\mathcal{D}}
	\arrow[""{name=0, anchor=center, inner sep=0}, "{f^*}", curve={height=-12pt}, from=1-1, to=1-2]
	\arrow[""{name=1, anchor=center, inner sep=0}, "{f_*}", curve={height=-12pt}, from=1-2, to=1-1]
	\arrow["\dashv"{anchor=center, rotate=-90}, draw=none, from=0, to=1]
\end{tikzcd}\]
is an adjunction between presentable $\infty$-categories $\mathcal{C}, \mathcal{D}$ that both satisfy the condition that filtered colimits are exact. Assume $f_*$ preserves filtered colimits. Then tensoring with $\mathrm{Sp}$ induces a further adjunction
\[	\begin{tikzcd}
	{\mathcal{C} \otimes \mathrm{Sp}} & {\mathcal{D} \otimes \mathrm{Sp}.}
	\arrow[""{name=0, anchor=center, inner sep=0}, "{f^* \otimes \mathrm{id}_\mathrm{Sp}}", curve={height=-24pt}, from=1-1, to=1-2]
	\arrow[""{name=1, anchor=center, inner sep=0}, "{(f \otimes \mathrm{id}_\mathrm{Sp})^!}"', curve={height=12pt}, from=1-1, to=1-2]
	\arrow[""{name=2, anchor=center, inner sep=0}, "{(f \otimes \mathrm{id}_\mathrm{Sp})_*}"', from=1-2, to=1-1]
	\arrow["\dashv"{anchor=center, rotate=-91}, draw=none, from=0, to=2]
	\arrow["\dashv"{anchor=center, rotate=-88}, draw=none, from=2, to=1]
\end{tikzcd}\]
\end{lemma}

\begin{proof}
Note that 
\[ \mathcal{C} \otimes \mathrm{Sp} \cong \mathrm{lim}( \cdots \xrightarrow{ \Omega } \mathcal{C}_{1 /} \xrightarrow{\Omega} \mathcal{C}_{1 /} ). \]
The functor $f_*$ preserves limits, and thus is in particular left exact, hence the induced right adjoint $ \mathcal{D} \otimes \mathrm{Sp} \rightarrow \mathcal{C} \otimes \mathrm{Sp}$ is obtained by applying $f_*$ objectwise. The assumption that filtered colimits in $\mathcal{C}$ and $\mathcal{D}$ are exact implies that $\mathcal{C}_{1 /} \xrightarrow{\Omega} \mathcal{C}_{1 /}$ preserves filtered colimits, and similarly for $\mathcal{D}$, hence filtered colimits in $\mathcal{C} \otimes \mathrm{Sp}$ and $\mathcal{D} \otimes \mathrm{Sp}$ are computed objectwise as well. Hence the induced right adjoint 
\[ (f \otimes \mathrm{id}_\mathrm{Sp})_* : \mathcal{D} \otimes \mathrm{Sp} \rightarrow \mathcal{C} \otimes \mathrm{Sp} \]
preserves filtered colimits. Now Lemma \ref{filteredcolimitsexactfunctor} guarantees the existence of the further right adjoint.
\end{proof}

Let us summarize the immediate consequence for sheaves on a topos with values in a presentable stable $\infty$-category.

\begin{corollary}
Let $\mathcal{E}$ be a presentable stable $\infty$-category. The functor
\[ \begin{array}{rcl}
\mathrm{LTop} & \rightarrow & \mathrm{Pr}^L \\
\mathcal{X} & \mapsto & \mathcal{X} \otimes \mathcal{E}
\end{array} \]
sends perfect geometric morphisms to internal left adjoints, i.e. for every perfect geometric morphism of topoi $f : \mathcal{X} \rightarrow \mathcal{Y}$ there are induced adjunctions
\[ \begin{tikzcd}
	{\mathrm{Sh}(\mathcal{X},\mathcal{E})} & {\mathrm{Sh}(\mathcal{Y},\mathcal{E})}
	\arrow[""{name=0, anchor=center, inner sep=0}, "{f_*}"{description}, from=1-1, to=1-2]
	\arrow[""{name=1, anchor=center, inner sep=0}, "{f^*}"', curve={height=12pt}, from=1-2, to=1-1]
	\arrow[""{name=2, anchor=center, inner sep=0}, "{f^!}", curve={height=-12pt}, from=1-2, to=1-1]
	\arrow["\dashv"{anchor=center, rotate=-91}, draw=none, from=0, to=2]
	\arrow["\dashv"{anchor=center, rotate=-89}, draw=none, from=1, to=0]
\end{tikzcd} \]
\end{corollary}

\begin{proof}
Since $\mathcal{E} \simeq \mathcal{E} \otimes \mathrm{Sp}$, and the Lurie tensor product is a $2$-functor in both variables, the statement reduces to the case $\mathcal{E} = \mathrm{Sp}$, which is exactly Lemma \ref{perfectfunctorstable}.
\end{proof}

We can now prove the following statement about homology of compact topoi.

\begin{theorem} \label{compacthomology}
Let $\mathcal{X}$ be a compact topos, and let $\mathcal{E}$ be a compactly assembled stable and presentable $\infty$-category. Let $E \in \mathcal{E}$. Then
\[ H_\bullet( \mathcal{X}, E ) \cong \Pi_\infty( \mathcal{X} ) \otimes E \in \mathcal{E}. \]
\end{theorem}

\begin{proof}
Since $\mathrm{Fun}^L(-, \mathcal{E}) : (\mathrm{Pr}^L)^{\mathrm{op}} \rightarrow \mathrm{Pr}^L$ is a $2$-functor, the adjunctions
\[\begin{tikzcd}
	{\mathrm{Sh}(\mathcal{X}, \mathrm{Sp})} & {\mathrm{Sp}}
	\arrow[""{name=0, anchor=center, inner sep=0}, "{\mathcal{X}_*}"{description}, from=1-1, to=1-2]
	\arrow[""{name=1, anchor=center, inner sep=0}, "{\mathcal{X}^*}"', curve={height=18pt}, from=1-2, to=1-1]
	\arrow[""{name=2, anchor=center, inner sep=0}, "{\mathcal{X}^!}", curve={height=-12pt}, from=1-2, to=1-1]
	\arrow["\dashv"{anchor=center, rotate=-90}, draw=none, from=0, to=2]
	\arrow["\dashv"{anchor=center, rotate=-90}, draw=none, from=1, to=0]
\end{tikzcd}\]
induce, under the equivalence $\mathcal{E} \simeq \mathrm{Fun}^L(\mathrm{Sp},\mathcal{E})$, the adjunctions
\[\begin{tikzcd}
	{\mathrm{CoSh}(\mathcal{X}, \mathcal{E})} & {\mathcal{E}}
	\arrow[""{name=0, anchor=center, inner sep=0}, "{(\mathcal{X}^*)^\circ}"{description}, from=1-1, to=1-2]
	\arrow[""{name=1, anchor=center, inner sep=0}, "{(\mathcal{X}_*)^\circ}"', curve={height=18pt}, from=1-2, to=1-1]
	\arrow[""{name=2, anchor=center, inner sep=0}, "{\mathcal{X}^+}", curve={height=-18pt}, from=1-2, to=1-1]
	\arrow["\dashv"{anchor=center, rotate=-90}, draw=none, from=0, to=2]
	\arrow["\dashv"{anchor=center, rotate=-90}, draw=none, from=1, to=0]
\end{tikzcd}\]
where $(\mathcal{X}^*)^\circ = \mathcal{X}_+$. From this we see that we have the adjunction
	\[ (\mathcal{X}_* \mathcal{X}^* )^\circ \cong (\mathcal{X}^*)^\circ (\mathcal{X}_*)^\circ \dashv \mathcal{X}_+ \mathcal{X}^+ \]
of functors $\mathcal{E} \rightarrow \mathcal{E}$. But since $\mathcal{E}$ is compactly assembled, the left-hand side agrees with the functor $(-)^{\Pi_\infty(\mathcal{X})}$ and therefore its right adjoint $\mathcal{X}_+ \mathcal{X}^+ $ agrees with the functor $\Pi_\infty(\mathcal{X}) \otimes - $.
\end{proof}

\begin{example}
Without either compactness or local contractibility of the topos $\mathcal{X}$, the statement of Theorems \ref{localcontractiblehomology} and \ref{compacthomology} becomes false. There always exists a comparison map
\[ H_\bullet( \mathcal{X}, E ) \rightarrow \Pi_\infty( \mathcal{X} ) \otimes E. \]
For the case $\mathcal{E} = D(\mathbb{Z})$, and $E = H\mathbb{Z}$, the right-hand side is referred to as \emph{strong homology} in earlier literature, see \cite{LisicaMardesic1985StrongHomologyI}.
As we have seen in Lemma \ref{homologycolimits}, the left-hand side is always additive. However, the right-hand side is not. Assuming the continuum hypothesis, Marde{\v{s}}i{\'c} and Prasolov \cite{MardesicPrasolov1988StrongHomology} gave a counterexample given by a countable wedge of $2$-dimensional Hawaiian earrings. Later, Prasolov gave a counterexample which is independent of the continuum hypothesis, see \cite{Prasolov2005NonAdditivity}.
\end{example}

\begin{remark}
The proof of Theorem \ref{compacthomology} also shows that 
\[ H_\bullet( \mathcal{X}, E ) \cong \mathcal{X}_* \mathcal{X}^!(E).\] Hence for compact Hausdorff spaces $X$, the topos definition of homology of $\mathrm{Sh}(X)$ agrees with the definition of homology that arises from the six-functor formalism on locally compact Hausdorff spaces, see e.g.\ \cite{Volpe2021SixOperationsInTopology}.
\end{remark}

\begingroup
\setlength{\emergencystretch}{8em}
\printbibliography
\endgroup

\end{document}